\documentclass[reqno, 12pt]{amsart} %reqno, 公式编号靠右; 12pt, 正文字体; amsart, 美国数学会数学类文章

\usepackage[expansion=false]{microtype} % 提高微排版质量，使文档更美观，减少断行问题
\usepackage{amsfonts,amsthm,amsmath,amssymb,amscd,mathrsfs} % 这些是 AMS（American Mathematical Society）数学排版宏包
\usepackage{mathtools,bm,aliascnt,enumitem,array} % 数学公式增强、数学粗体、计数器别名、自定义列表及数组/表格功能
\allowdisplaybreaks   %允许公式断行
\usepackage{latexsym} % 提供额外的数学符号
\usepackage[colorlinks=true,linkcolor=blue,citecolor=red,urlcolor=black]{hyperref} % 处理超链接
\usepackage[nameinlink,noabbrev]{cleveref} % 智能交叉引用：引用名称包含在超链接中，并使用完整名称而非缩写
\usepackage{graphicx}
\usepackage{indentfirst} % 让段落的第一行自动缩进
\usepackage{cite} % 改进文献引用格式，使多个引用合并，如 [1,2,3] 变为 [1–3]
\usepackage{color} % 允许在文档中使用颜色
\usepackage{lmodern} % 使用 Latin Modern 字体，提高可读性和输出质量
\usepackage[pagewise]{lineno} % 提供开启行号功能
\theoremstyle{plain} % 使用 plain 样式：标题通常为粗体，正文为斜体
\newtheorem{Thm}{Theorem}[section] % 定义定理环境 Thm，按 section 分节编号，如 Theorem 2.1

\newaliascnt{Prop}{Thm} % Proposition 与 Theorem 共用同一计数器
\newtheorem{Prop}[Prop]{Proposition} % 定义命题环境 Proposition
\aliascntresetthe{Prop} % 恢复 Proposition 的正确引用编号格式，供 cleveref 使用

\newaliascnt{Lem}{Thm} % Lemma 与 Theorem 共用同一计数器
\newtheorem{Lem}[Lem]{Lemma} % 定义引理环境 Lemma
\aliascntresetthe{Lem} % 恢复 Lemma 的正确引用编号格式，供 cleveref 使用

\newaliascnt{Cor}{Thm} % Corollary 与 Theorem 共用同一计数器
\newtheorem{Cor}[Cor]{Corollary} % 定义推论环境 Corollary
\aliascntresetthe{Cor} % 恢复 Corollary 的正确引用编号格式，供 cleveref 使用

\theoremstyle{definition} % 使用 definition 样式：标题通常为粗体，正文为正体

\newaliascnt{Def}{Thm} % Definition 与 Theorem 共用同一计数器
\newtheorem{Def}[Def]{Definition} % 定义 Definition 环境
\aliascntresetthe{Def} % 恢复 Definition 的正确引用编号格式，供 cleveref 使用

\theoremstyle{remark} % 使用 remark 样式：标题通常为斜体，正文为正体

\newaliascnt{Rem}{Thm} % Remark 与 Theorem 共用同一计数器
\newtheorem{Rem}[Rem]{Remark} % 定义 Remark 环境
\aliascntresetthe{Rem} % 恢复 Remark 的正确引用编号格式，供 cleveref 使用

\crefname{Thm}{theorem}{theorems} % \cref 引用定理时使用 theorem/theorems
\Crefname{Thm}{Theorem}{Theorems} % \Cref 引用定理时首字母大写

\crefname{Prop}{proposition}{propositions} % Proposition 的小写引用名称
\Crefname{Prop}{Proposition}{Propositions} % Proposition 的大写引用名称

\crefname{Lem}{lemma}{lemmas} % Lemma 的小写引用名称
\Crefname{Lem}{Lemma}{Lemmas} % Lemma 的大写引用名称

\crefname{Cor}{corollary}{corollaries} % Corollary 的小写引用名称
\Crefname{Cor}{Corollary}{Corollaries} % Corollary 的大写引用名称

\crefname{Rem}{remark}{remarks} % Remark 的小写引用名称
\Crefname{Rem}{Remark}{Remarks} % Remark 的大写引用名称

\crefname{Def}{definition}{definitions} % Definition 的小写引用名称
\Crefname{Def}{Definition}{Definitions} % Definition 的大写引用名称

\numberwithin{equation}{section} % 公式按 section 分节编号，如 (2.1)、(2.2)

\newcommand{\beq}{\begin{equation}}
	\newcommand{\eeq}{\end{equation}}
\newcommand{\ben}{\begin{eqnarray}}
	\newcommand{\een}{\end{eqnarray}}
\newcommand{\beno}{\begin{eqnarray*}}
	\newcommand{\eeno}{\end{eqnarray*}}

\newcommand{\px}{\partial_x}
\newcommand{\py}{\partial_y}
\newcommand{\pz}{\partial_z}
\newcommand{\pt}{\partial_t}

\newcommand{\abs}[1]{\lvert#1\rvert}  % \abs{x} 生成 |x|，确保正确的数学格式

\newcommand{\R}{\mathbb R}
\newcommand{\T}{\mathbb T}
\newcommand{\ee}{\mathrm e}

\newcommand{\pa}{\partial}
\newcommand{\nablaP}{\nabla_{\!\perp}}
\newcommand{\DeltaP}{\Delta_{\!\perp}}
\newcommand{\PP}{\mathbb P_{\!\perp}}

\newcommand{\norm}[2]{\|#1\|_{#2}}

\usepackage[
letterpaper,                 % 纸张尺寸 (amsart 默认，但明确指定更好)
textheight=8.35in,           % 正文高度
headsep=20pt,                % 页眉与正文间距 (您之前注释掉的)
footskip=36pt,               % 页脚与正文间距
marginparwidth=0pt,          % 禁用页边注释
marginparsep=0pt,            % 禁用页边注释
left=0.75in, 
right=0.75in
]{geometry}
\title[Nonlinear stability of 3D Couette flow]{Nonlinear stability of 3D Couette flow on $\mathbb T\times\mathbb R^2$ under critical transverse low-frequency control}

\author{
    Yubo~Chen$^{1}$\and
	Guoxu~Yang$^{1}$\and
    Jianbo~Yu$^{1,*}$
}

\thanks{$^1$School of Mathematical Sciences, Dalian University of Technology, Dalian 116024, China.}

\thanks{E-mail addresses:  
    1220823215@mail.dlut.edu.cn (Yubo~Chen),
	guoxu\_dlut@outlook.com (Guoxu~Yang),
    yujb@mail.dlut.edu.cn (Jianbo~Yu)
    .}

\thanks{$^*$Author to whom correspondence should be addressed.}



\hypersetup{pdftitle={Nonlinear stability of 3D Couette flow on $\mathbb T\times\mathbb R^2$ under critical transverse low-frequency control}, pdfauthor={Yubo Chen, Guoxu Yang, Jianbo Yu}}

\begin{document}
	
	\begin{abstract}
We establish nonlinear stability of three-dimensional Couette flow
for the incompressible Navier--Stokes equations on
$\mathbb T\times\mathbb R^2$.
Divergence-free perturbations satisfying
$\|u_{\rm in}\|_{H^2}
+\|u^\perp_{0,\rm in}\|_{\dot B^{-2}_{2,\infty}}\leq c_*\nu$
generate unique global $H^2$ solutions that remain close to
Couette flow, where $u^\perp_{0,\rm in}$ denotes the streamwise
average of the transverse velocity.
This result retains the $O(\nu)$ stability threshold despite
the absence of a spanwise spectral gap.
The nonzero streamwise modes exhibit enhanced dissipation,
their wall-normal component satisfies a viscosity-uniform
time-integrated inviscid-damping estimate, and the transverse
zero mode decays in $L^2$ at rate $(1+\nu t)^{-1}$.
The endpoint space $\dot B^{-2}_{2,\infty}$ controls the
transverse low frequencies, accommodates divergence-form
nonlinear forcing, and is critical for linear lift-up
from the wall-normal zero mode.
The proof combines propagation of this endpoint norm,
Fourier splitting, and a two-level roll-to-streak representation
with an adaptation of the nonzero-mode framework of Wei--Zhang (CPAM, 2021).
\end{abstract}
	
	\maketitle

	\tableofcontents
	
	\section{Introduction}

\subsection{Background and previous results}

Consider the incompressible Navier--Stokes equations on
$\Omega=\T_x\times\R_y\times\R_z$:
\begin{equation*}
 \pt v-\nu\Delta v+v\cdot\nabla v+\nabla p=0,
 \qquad \nabla\cdot v=0,
\end{equation*}
where $0<\nu\leq1$ is the viscosity. The variables $x$, $y$, and $z$
denote the streamwise, wall-normal, and spanwise directions, respectively,
and $\T=\mathbb R/\mathbb Z$, so $|\T|=1$.
We use angular Fourier frequencies $k\in2\pi\mathbb Z$ and
$(\eta,\ell)\in\mathbb R^2$, with $\partial_x,\partial_y,\partial_z$
represented by $ik,i\eta,i\ell$, respectively.
The Couette flow $U=(y,0,0)$ is a stationary solution.
With $v=U+u$, the perturbation equation becomes
\begin{equation}
 \pt u-\nu\Delta u+y\px u+u^2e_1+u\cdot\nabla u+\nabla p=0,
 \qquad \nabla\cdot u=0.
 \label{eq:pert}
\end{equation}
We ask how small the initial perturbation must be, as
$\nu\to0$, for the solution to remain close to Couette flow for all time.

The linear dynamics are governed by enhanced dissipation, inviscid
damping, and lift-up. These mechanisms act on different components,
so the streamwise average and its complement are distinguished from
the outset. For a scalar
or vector field $f$, define
\begin{equation}
 f_{0}(y,z):=P_{0}f:=\int_{\T}f(x,y,z)\,dx,
 \qquad f_{\neq}:=P_{\neq}f:=f-f_{0}.
 \label{eq:mode-convention}
\end{equation}
Write
\[
 x_\perp=(y,z),\qquad u^\perp=(u^2,u^3),\qquad
 \nablaP=(\partial_y,\partial_z),\qquad
 \DeltaP=\partial_y^2+\partial_z^2.
\]
The transverse average $u^\perp_{0}$ is called the roll, and
$u^1_{0}$ the streak. Spectral stability alone does not describe the linear evolution. Plane Couette
flow in a finite channel is spectrally stable at every Reynolds number
\cite{Romanov1973}, but its nonnormal linearized operator permits
substantial transient growth \cite{TrefethenEtAl1993}. The three
mechanisms can be seen directly from the linearized equations.

\emph{Enhanced dissipation.}
For the scalar equation $(\pt-\nu\Delta+y\px)f=0$, the coordinates
$X=x-ty$ remove the transport. If $(k,\eta,\ell)$ are the Fourier
frequencies conjugate to $(X,y,z)$, the viscous multiplier is
\[
 \exp\!\left[-\nu\int_0^t
 \bigl(k^2+(\eta-ks)^2+\ell^2\bigr)\,ds\right]
 \leq \exp\!\left(-\frac{\nu k^2t^3}{12}\right).
\]
For $k\neq0$, shear mixing creates progressively shorter spatial
scales, on which diffusion becomes effective. This gives the time scale
$\nu^{-1/3}$, much shorter than the diffusive scale $\nu^{-1}$ at unit
frequency. Mixing itself does not dissipate scalar $L^2$ energy; the
faster decay comes from its interaction with viscosity. General connections between mixing and enhanced dissipation were
established in \cite{ConstantinKiselevRyzhikZlatos2008}.
Quantitative bounds relating mixing rates to dissipation times
were developed in
\cite{FengIyer2019,CotiZelatiDelgadinoElgindi2020}.
For shear flows, effective approaches include hypocoercivity
\cite{BedrossianCotiZelati2017}, resolvent estimates
\cite{Wei2021Resolvent}, and quantitative versions of
H\"ormander's hypoellipticity theory
\cite{AlbrittonBeekieNovack2022}.
More recently, enhanced dissipation and Taylor dispersion have
been studied together for higher-dimensional parallel shear flows
\cite{CotiZelatiGallay2023}, and a probabilistic approach based on
Malliavin calculus has been developed in \cite{Villringer2025}.
For a recent survey of the underlying mechanisms, quantitative
methods, and applications to nonlinear PDEs, see
\cite{MazzucatoFengNobili2026} and the references therein.

\emph{Inviscid damping.}
Mixing also reduces velocity through the elliptic relation between
velocity and vorticity. For the Euler equations linearized around
three-dimensional Couette flow, the wall-normal component satisfies
\[
 (\pt+y\px)\Delta u^2_{\neq}=0.
\]
Although $\Delta u^2_{\neq}$ is transported without dissipation,
recovering $u^2_{\neq}$ introduces the multiplier
$[k^2+(\eta-kt)^2+\ell^2]^{-1}$. The resulting cancellation gives
velocity damping even when viscosity is absent. In the nonlinear
problem, this effect supplies time-integrated bounds that are uniform
as $\nu\to0$.

In two dimensions, linear inviscid damping for monotone shear profiles
under suitable regularity and spectral assumptions has been established in
\cite{Zillinger2017,WeiZhangZhao2018,Jia2020Monotone}.
Connections with enhanced dissipation for the linearized viscous
equations are studied in
\cite{GrenierNguyenRoussetSoffer2020,ChenWeiZhang2023}.
For non-monotone shear profiles, linear inviscid damping and
vorticity depletion at stationary streamlines are investigated in
\cite{WeiZhangZhao2019,IonescuIyerJia2024}.
More recently, viscosity-uniform damping and vorticity depletion
estimates, together with enhanced dissipation, have been obtained
for a class of periodic shear flows with two non-degenerate
critical points \cite{BeekieChenJia2026}.
Nonlinear results are discussed below.

\emph{Lift-up and the low-frequency difficulty.}
The zero modes are unaffected by shear mixing. Instead, the
wall-normal velocity forces the streak through the term $u^2_{0}e_1$.
This lift-up mechanism can produce linear-in-time growth
\cite{EllingsenPalm1975,Landahl1980}. Indeed, the linearized zero-mode
system is
\begin{equation*}
 \pt u^1_{0}-\nu\DeltaP u^1_{0}=-u^2_{0},
 \qquad \pt u^\perp_{0}-\nu\DeltaP u^\perp_{0}=0,
\end{equation*}
and hence
\begin{equation*}
 u^1_{0}(t)=\ee^{\nu t\DeltaP}u^1_{0,\mathrm{in}}
 -t\ee^{\nu t\DeltaP}u^2_{0,\mathrm{in}}.
\end{equation*}
For $\xi_\perp=(\eta,\ell)\in\R^2\setminus\{0\}$,
\begin{equation}
 \sup_{t>0}t\ee^{-\nu t|\xi_\perp|^2}
 =\frac{1}{\ee\nu|\xi_\perp|^2}.
 \label{eq:rho-sing}
\end{equation}
When $z$ is periodic, incompressibility and square integrability in $y$
force the spanwise average of $u^2_{0}$ to vanish. The remaining
spanwise frequencies are bounded away from zero, and lift-up is
bounded by $O(\nu^{-1})$ in $L^2$. On $\T\times\R^2$, both transverse
frequencies can approach zero. Diffusion then becomes arbitrarily
slow, and ordinary positive-order Sobolev control does not bound the
lift-up term uniformly in time. Thus removing spanwise periodicity
changes the zero-mode problem even though the nonzero streamwise
modes retain their spectral gap.

This difficulty motivates a condition that measures the transverse
low frequencies. For a homogeneous dyadic decomposition $\dot\Delta_j$
on $\R^2$, define
\begin{equation}
 \|f\|_{\dot B^{-2}_{2,\infty}}
 :=\sup_{j\in\mathbb Z}2^{-2j}\|\dot\Delta_jf\|_{L^2(\R^2)}.
 \label{eq:intro-besov-definition}
\end{equation}
The heat-semigroup characterization gives
\begin{equation}
 \sup_{t>0}\|t\ee^{\nu t\DeltaP}f\|_{L^2}
 \simeq\nu^{-1}\|f\|_{\dot B^{-2}_{2,\infty}}.
 \label{eq:intro-liftup-characterization}
\end{equation}
This norm is therefore exactly the condition for bounded linear
lift-up of the wall-normal zero mode in $L^2$. It is also compatible
with nonlinear forcing: the projected divergence of an $L^1$ tensor
belongs to the same Besov space by \eqref{eq:div-B}. The proof below
uses this compatibility to propagate the low-frequency condition.

The competition between transient growth, mixing, and diffusion is
central to the stability-threshold problem. Formal asymptotic
predictions were developed in \cite{Chapman2002}; for a broader account
of the mathematical theory and further references, see
\cite{BGM2017,BGMReview2019}.

For three-dimensional Couette flow on $\T\times\R\times\T$, global
stability for perturbations of size $O(\nu)$ in Gevrey regularity was
proved by Bedrossian--Germain--Masmoudi \cite{BGM2020}. Their Sobolev
theory gives stability for $O(\nu^{3/2})$ perturbations in $H^\sigma$,
$\sigma>9/2$ \cite{BGM2017}. The $O(\nu)$ stability scale was subsequently
obtained in $H^2$ by Wei--Zhang \cite{WeiZhang2021}. Their argument
combines enhanced dissipation and inviscid damping with coordinates and
good unknowns adapted to the evolving streamwise-independent flow.
For the finite channel $\T\times[-1,1]\times\T$ with no-slip boundary
conditions, an $H^2$ stability result at the same $O(\nu)$ scale was
established by Chen--Wei--Zhang \cite{ChenWeiZhang2024}.

The behavior above this scale has also been studied. For Gevrey
perturbations of size $\varepsilon\lesssim\nu^{2/3+\delta}$, with
$0<\delta<1/3$, the analysis in \cite{BGMAbove2015} describes the growth of
streamwise-independent structures up to times of order
$\varepsilon^{-1}$. This is a finite-time description, rather than a
global stability result at that larger scale. Related three-dimensional
threshold results have been proved for Kolmogorov flow
\cite{LiWeiZhang2020} and, more recently, plane Poiseuille flow
\cite{ChenDingLinZhang2026}. For further stability-threshold results beyond Couette flow, we refer to the references therein.

In two dimensions, nonlinear inviscid damping near Couette flow was
established for the Euler equations in Gevrey regularity
\cite{BedrossianMasmoudi2015}. For Navier--Stokes, stability and enhanced
dissipation were proved for Gevrey perturbations whose size is
independent of viscosity \cite{BMV2016}. For a finite channel with
Navier boundary conditions, \cite{BHIW2024} established nonlinear
inviscid damping, enhanced dissipation, and a long-time inviscid limit
for small Gevrey initial vorticity supported away from the boundaries,
with smallness independent of viscosity. In Sobolev regularity, the
$O(\nu^{1/2})$ result of \cite{BVW2018} was followed by
$O(\nu^{1/3})$ criteria in suitable vorticity norms
\cite{MasmoudiZhao2022,WeiZhang2023}. Li--Masmoudi--Zhao
\cite{LMZ2025} established regularity-dependent stability thresholds
connecting the Sobolev and Gevrey regimes.

Boundary conditions also matter:
an $O(\nu^{1/2})$ velocity criterion in $H^2$ was obtained in a finite
channel with no-slip conditions \cite{ChenLiWeiZhang2020}, whereas
recent results with Navier boundary conditions treat both nearly
Couette shear flows \cite{BedrossianHeIyerWang2025} and an
$O(\nu^{1/3})$ criterion for exact Couette flow
\cite{WeiZhang2026}. Stability thresholds for close-to-Couette shear
flows with no-slip boundary conditions were studied in
\cite{BHILW2025}. Boundary-layer formation can also drive instability
near Couette flow, as shown in \cite{BGMZ2025}.

On the whole plane, the absence of a streamwise spectral gap requires
additional low-frequency control \cite{LiLiuZhao2026}.
Arbon--Bedrossian \cite{AB2025} studied quantitative stability on
unbounded domains with Navier boundary conditions, combining enhanced
dissipation with Taylor dispersion. For streamwise-unbounded channels,
recent results treat both no-slip \cite{CLM2026a} and Navier-slip
\cite{CLM2026b} boundary conditions.
For further nonlinear inviscid-damping results near monotone shears,
see \cite{IonescuJia2023,MasmoudiZhao2024} and the references therein.
Viscous nonlinear stability near monotone shear flows in a critical
regularity class was established in \cite{LZ2024}.
Enhanced dissipation and stability thresholds for two-dimensional
Kolmogorov and Poiseuille flows are treated in
\cite{WeiZhangZhao2020,CotiZelatiElgindiWidmayer2020}.
More recently, Chen--Jia--Wei--Zhang \cite{CJWZ2025} proved nonlinear
asymptotic stability of Kolmogorov flow on a non-square torus for
$H^3$ vorticity perturbations at the $\nu^{1/3}$ scale.

Recent three-dimensional studies also show how additional physical
effects change the zero-mode dynamics. Rotation-induced dispersion is
used in \cite{CotiZelatiDelZottoWidmayer2025,LiSunWangWeiZhang2025},
including on $\T\times\R^2$ in the latter work. For stably stratified
Couette flow, suppression of lift-up and linear enhanced dissipation
were established in \cite{CZDZ2025}; dispersion and mixing are combined
in the nonlinear analysis of \cite{CotiZelatiDelZottoWidmayer2026}.
The combined effects of rotation and stratification were studied in
\cite{HLSX2026}. When the shear and stratification directions are
perpendicular, Yao \cite{Y2026} identified a different zero-mode
lift-up mechanism and used a potential-vorticity formulation together
with weighted initial conditions in the stability analysis.
For the three-dimensional Boussinesq system in a finite no-slip channel,
Li--Lin--Wang \cite{LLW2026} proved stability near Couette flow and a
constant temperature state, with $H^2$ perturbations of size $O(\nu)$
in velocity and $O(\nu^2)$ in temperature.
These additional couplings are absent from
\eqref{eq:pert}. The result below treats the ordinary Navier--Stokes
equations on $\T\times\R^2$ by combining the endpoint zero-mode
condition with the nonzero-mode estimates. In particular, the
transverse low frequencies are controlled without a spanwise spectral
gap or an additional dispersive mechanism.

\subsection{Main results}

The first theorem gives global existence and stability at the
$O(\nu)$ initial scale. The second records the different decay and
time-integrability properties of the nonzero modes, roll, and streak.

\begin{Thm}
 \label{thm:full-couette-final}
There exist absolute constants $c_*>0$ and $C>0$ such that the
following holds. Let $0<\nu\leq1$ and let $u_{\rm in}\in H^2(\Omega)$
be divergence free. If
\begin{equation}
 \varepsilon:=\|u_{\rm in}\|_{H^2(\Omega)}
 +\|u^\perp_{0,\rm in}\|_{\dot B^{-2}_{2,\infty}(\R^2)}
 \leq c_*\nu,
 \label{eq:main-smallness-final}
\end{equation}
then \eqref{eq:pert} has a unique global solution
\[
 u\in C([0,\infty);H^2(\Omega))
 \cap L^2_{\rm loc}([0,\infty);H^3(\Omega)),
\]
which satisfies
\begin{equation}
 \sup_{t\geq0}\|u(t)\|_{H^2(\Omega)} \leq Cc_*.
 \label{eq:true-Couette-stability-final}
\end{equation}
Moreover, for every $\tau_{0}>0$,
\begin{equation}
 u\in C([\tau_{0},\infty);H^4(\Omega))
 \cap L^2_{\rm loc}([\tau_{0},\infty);H^5(\Omega)).
 \label{eq:full-positive-time-H4}
\end{equation}
\end{Thm}

\begin{Rem}
The threshold $\varepsilon\lesssim\nu$ has the same exponent as
the results of Wei--Zhang \cite{WeiZhang2021} on
$\T\times\R\times\T$ and Chen--Wei--Zhang
\cite{ChenWeiZhang2024} in a finite channel with no-slip
boundary conditions.
This scale is also consistent with the sharp linear lift-up
mechanism on $\T\times\R^2$.
Indeed, the linearized zero mode satisfies
\[
 u^{1}_{0}(t)
 =\ee^{\nu t\Delta_\perp}
 \bigl(u^1_{0,\mathrm{in}}-t u^2_{0,\mathrm{in}}\bigr).
\]
For suitable roll data supported in a fixed transverse
frequency annulus, an initial perturbation of size $\varepsilon$
produces a streak of size $\varepsilon/\nu$ at times
$t\sim\nu^{-1}$.
Thus $\varepsilon\lesssim\nu$ is the natural scale for remaining
in a fixed small neighborhood of Couette flow.
The present theorem attains this scale despite the absence of
a spanwise spectral gap, with the transverse low frequencies
controlled by the additional endpoint Besov assumption.
\end{Rem}

\begin{Rem}
 \label{rem:full-roll-sufficiency-final}
The criticality in \eqref{eq:intro-liftup-characterization} concerns
linear lift-up of $u^2_{0}$. The condition on $u^3_{0}$ is used to
propagate the full nonlinear roll and its quadratic stresses. Thus
\eqref{eq:main-smallness-final} is a sufficient nonlinear stability
criterion.
\end{Rem}

\begin{Rem}
 \label{rem:stability-scale-final}
At $\varepsilon=c_*\nu$, lift-up permits a streak of size
$\varepsilon/\nu=O(c_*)$, while the roll and the low-order nonzero-mode
bounds remain $O(\varepsilon)$. The two scales in
\eqref{eq:full-energy-final} reflect this distinction. The theorem
therefore gives uniform smallness relative to the fixed Couette flow,
with enhanced dissipation of the streamwise-dependent part and
algebraic decay of the roll.
\end{Rem}

\begin{Thm}
 \label{thm:global-decay-final}
Under the assumptions of \Cref{thm:full-couette-final}, there exist
absolute constants $\lambda_*>0$ and $C>0$ such that the following
bounds hold.
\begin{enumerate}
\item The nonzero streamwise modes satisfy enhanced dissipation:
\begin{equation}
 \|u_{\neq}(t)\|_{L^2(\Omega)}
 \leq C\varepsilon\ee^{-\lambda_*\nu^{1/3}t},\qquad t\geq0,
 \label{eq:full-ED-final}
\end{equation}
and
\begin{align}
 \nu^{1/6}\|\ee^{\lambda_*\nu^{1/3}t}u_{\neq}\|_{L^2(0,\infty;L^2)}
 +\nu^{1/2}
 \|\ee^{\lambda_*\nu^{1/3}t}\nabla u_{\neq}\|_{L^2(0,\infty;L^2)}
 \leq C\varepsilon.
 \label{eq:global-ED-spacetime}
\end{align}
\item The wall-normal nonzero mode satisfies the inviscid-damping
estimate
\begin{align}
 \|\ee^{\lambda_*\nu^{1/3}t}u^2_{\neq}\|_{L^2(0,\infty;H^1)}
 +\|\ee^{\lambda_*\nu^{1/3}t}
 \partial_x\nabla u^2_{\neq}\|_{L^2(0,\infty;L^2)}
 \leq C\varepsilon.
 \label{eq:global-inviscid-damping}
\end{align}
\item The roll has algebraic decay and retains its endpoint control:
\begin{equation}
 \|u^\perp_{0}(t)\|_{L^2(\R^2)}
 \leq\frac{C\varepsilon}{1+\nu t},\qquad t\geq0,
 \label{eq:global-roll-decay}
\end{equation}
\begin{equation}
 \sup_{t\geq0}\|u^\perp_{0}(t)\|_{\dot B^{-2}_{2,\infty}}
 +\nu^{1/2}\|u^\perp_{0}\|_{L^2(0,\infty;L^2)}
 \leq C\varepsilon.
 \label{eq:global-roll-endpoint}
\end{equation}
The streak satisfies, after positive-time smoothing,
\begin{align}
 \sup_{t\geq1}\|u^1_{0}(t)\|_{H^4}
 +\nu^{1/2}\|\nablaP\DeltaP^2u^1_{0}\|_{L^2(1,\infty;L^2)}
 +\nu^{-1}\sup_{t\geq1}\|\partial_tu^1_{0}(t)\|_{H^2}
 \leq C\frac{\varepsilon}{\nu}.
 \label{eq:full-zero-final}
\end{align}
\end{enumerate}
\end{Thm}

\begin{Rem}
 \label{rem:no-physical-H1-decay-revised}
Estimate \eqref{eq:global-inviscid-damping} is a time-integrated form
of inviscid damping: its constant is independent of viscosity, and no
positive power of $\nu$ is placed before the norm. It is stronger than
the $L^2_t$ bound obtained by integrating enhanced dissipation alone,
which loses $\nu^{-1/6}$. A pointwise $t^{-2}$ nonlinear damping rate
is not asserted here. The exponential bound with $O(\varepsilon)$
amplitude in \eqref{eq:full-ED-final} is stated in physical $L^2$;
physical wall-normal derivatives also contain shear-induced growth.
\end{Rem}

Two sufficient conditions allow the Besov hypothesis to be checked
in more familiar spaces. The homogeneous Sobolev norm and the spatial
weight used below are
\[
 \|f\|_{\dot H^{-2}(\R^2)}
 :=\bigl\||\xi_\perp|^{-2}\widehat f(\xi_\perp)\bigr\|_{L^2_{\xi_\perp}},
 \qquad \langle x_\perp\rangle:=(1+|x_\perp|^2)^{1/2}.
\]

\begin{Cor}
 \label{cor:homogeneous-Hminus2}
There exists an absolute constant $c_H>0$ such that, for
$0<\nu\leq1$ and divergence-free $u_{\rm in}\in H^2(\Omega)$,
\begin{equation}
 \varepsilon_H:=\|u_{\rm in}\|_{H^2(\Omega)}
 +\|u^\perp_{0,\rm in}\|_{\dot H^{-2}(\R^2)}\leq c_H\nu
 \label{eq:Hminus2-smallness}
\end{equation}
implies all conclusions of
\Cref{thm:full-couette-final,thm:global-decay-final}, with
$\varepsilon$ replaced by $\varepsilon_H$ and an adjusted absolute
constant $C$.
\end{Cor}

\begin{Rem}
 \label{rem:homogeneous-versus-inhomogeneous}
The space in \Cref{cor:homogeneous-Hminus2} is homogeneous.
The inhomogeneous $H^{-2}$ norm is already controlled by $L^2$ and
adds no condition at the transverse origin. Although
$\dot H^{-2}=\dot B^{-2}_{2,2}$ embeds into
$\dot B^{-2}_{2,\infty}$, it does not contain general projected
divergences of $L^1$ tensors; an explicit example is given after
\Cref{prop:linear-criticality}. Consequently, the stronger Sobolev
condition is imposed only on the initial data, and propagation is
carried out in the endpoint Besov space.
\end{Rem}

\begin{Cor}
 \label{cor:weighted-initial-data}
There exists an absolute constant $c_m>0$ such that, for
$0<\nu\leq1$ and divergence-free $u_{\rm in}\in H^2(\Omega)$,
\begin{equation}
 \varepsilon_m:=\|u_{\rm in}\|_{H^2(\Omega)}
 +\||x_\perp|u^\perp_{0,\rm in}\|_{L^1(\R^2)}\leq c_m\nu
 \label{eq:moment-smallness-global}
\end{equation}
implies all conclusions of
\Cref{thm:full-couette-final,thm:global-decay-final}, with
$\varepsilon$ replaced by $\varepsilon_m$ and an adjusted absolute
constant $C$. In particular, for each fixed $\sigma>2$, it suffices that
\begin{equation}
 \varepsilon_\sigma:=\|u_{\rm in}\|_{H^2(\Omega)}
 +\|\langle x_\perp\rangle^\sigma
 u^\perp_{0,\rm in}\|_{L^2(\R^2)}\leq c_\sigma\nu.
 \label{eq:weighted-L2-smallness}
\end{equation}
In this case the bounds hold with $\varepsilon_\sigma$ in place of
$\varepsilon$, and the constants may depend on $\sigma$.
\end{Cor}

\begin{Rem}
 \label{rem:weighted-roll-condition}
The weights in \Cref{cor:weighted-initial-data} are required only for
the initial roll, not for the full three-dimensional perturbation.
Incompressibility and the first moment imply zero transverse mean
and $|\widehat u^\perp_{0,\rm in}(\xi_\perp)|\lesssim|\xi_\perp|$,
which is sufficient for the Besov condition. Such first-order
vanishing need not give a finite $\dot H^{-2}$ norm. The weighted
criterion therefore covers data not included in
\Cref{cor:homogeneous-Hminus2}; no propagation of the spatial moment
is needed.
\end{Rem}

\begin{Rem}
 \label{rem:streak-endpoint-decay}
Uniform streak control should be distinguished from streak decay.
Even for the linearized equations, the endpoint condition does not
force $t\ee^{\nu t\DeltaP}u^2_{0,\rm in}$ to tend to zero in $L^2$.
For example, let $\chi\in C_c^\infty(\R^2)$ be real, even, and equal to one
near the origin and set
$\widehat u^\perp_{0,\rm in}(\eta,\ell)
=i(-\ell,\eta)\chi(\eta,\ell)$.
This is a divergence-free Schwartz roll, but a change of variables
$\xi_\perp=(\nu t)^{-1/2}\zeta$ shows that
$\|t\ee^{\nu t\DeltaP}u^2_{0,\rm in}\|_{L^2}$ tends to a positive
multiple of $\nu^{-1}$. This explains why
\eqref{eq:full-zero-final} states boundedness rather than a decay
rate for the streak.
\end{Rem}

\subsection{Main ideas of the proof}

The main contribution is to retain the $O(\nu)$ stability threshold
when the spanwise direction is unbounded. The difficulty lies in the
streamwise-independent modes: arbitrarily low transverse frequencies
can sustain lift-up for arbitrarily long times. Our proof controls this
accumulation through a low-frequency condition that is preserved by
the nonlinear evolution. Three features of the argument are central.

\emph{An endpoint condition compatible with the nonlinearity.}
The linear lift-up formula determines the required low-frequency
control. By \eqref{eq:heat-char} and \Cref{prop:linear-criticality},
\[
 \sup_{t>0}\|t\ee^{\nu t\DeltaP}u^2_{0,\rm in}\|_{L^2}
 \simeq \nu^{-1}
 \|u^2_{0,\rm in}\|_{\dot B^{-2}_{2,\infty}}.
\]
This equivalence identifies the precise low-frequency condition for
uniformly bounded linear lift-up.
The same space also accommodates the nonlinear forcing. With $\PP$
denoting the Leray projector on $\R^2$, \eqref{eq:div-B} gives
\[
 \|\PP\nablaP\cdot F\|_{\dot B^{-2}_{2,\infty}}
 \lesssim \|F\|_{L^1(\R^2)}.
\]
The quadratic stress
$u^\perp_{0}\otimes u^\perp_{0}
+(u^\perp_{\neq}\otimes u^\perp_{\neq})_{0}$ therefore enters this
space using only $L^2$ bounds on the velocities.
This compatibility makes the endpoint condition useful for the full
nonlinear problem. The stronger space $\dot H^{-2}$ also controls
linear lift-up, but general projected divergences of $L^1$ tensors
need not belong to it; see the example after
\Cref{prop:linear-criticality}. We impose the Besov condition on the
full roll to propagate it through the coupled equations. Its linear
necessity concerns only the wall-normal component $u^2_{0}$.

\emph{Closing low-frequency propagation and roll decay together.}
Propagation of the endpoint norm requires time-integrated control of
the roll, which is precisely what the missing spectral gap prevents
us from obtaining by a Poincar\'e inequality. We recover this control
by Fourier splitting: the Besov norm bounds the low frequencies,
and viscosity controls the remaining frequencies. The roll
self-interaction cancels in the $L^2$ energy identity, so only the
wave stress enters as an external source.
Under \eqref{eq:stress-L1-main}--\eqref{eq:stress-L2-main},
\eqref{eq:B-prop} and \eqref{eq:fourier-splitting} give
\begin{align*}
 \sup_{0\leq t\leq T}
 \|u^\perp_{0}(t)\|_{\dot B^{-2}_{2,\infty}}
 &\lesssim \varepsilon
 +\int_0^T\|u^\perp_{0}(t)\|_{L^2}^2\,dt
 +\nu^{-1/3}\varepsilon^2,\\
 \int_0^T\|u^\perp_{0}(t)\|_{L^2}^2\,dt
 &\lesssim \nu^{-1}\left(
 \varepsilon^2+
 \sup_{0\leq t\leq T}
 \|u^\perp_{0}(t)\|_{\dot B^{-2}_{2,\infty}}^2\right)
 +\nu^{-3}\varepsilon^4.
\end{align*}
These inequalities close together for $\varepsilon\leq c\nu$ and yield
\[
 \sup_{0\leq t\leq T}
 \|u^\perp_{0}(t)\|_{\dot B^{-2}_{2,\infty}}
 +\nu^{1/2}\|u^\perp_{0}\|_{L^2(0,T;L^2)}
 \lesssim\varepsilon;
\]
see \eqref{eq:roll-B-main}--\eqref{eq:roll-L2-main}.
In particular, the roll self-interaction changes the endpoint norm
by at most $O(\varepsilon^2/\nu)$, which is smaller than its initial
size at the stated threshold. With the weighted wave-stress estimate
\eqref{eq:wave-roll-forcing-revised}, Fourier splitting also gives
$\|u^\perp_{0}(t)\|_{L^2}\lesssim\varepsilon(1+\nu t)^{-1}$,
as stated in \eqref{eq:global-roll-decay}.

\emph{Controlling accumulated lift-up without integrable roll decay.}
The decay just obtained is not integrable in time. Consequently,
estimating the streak forcing by
$\int_0^t\|u^2_{0}(s)\|_{L^2}\,ds$ would lose uniform control.
We instead substitute the roll equation into the streak equation
before estimating either time integral. The heat-semigroup property
then gives the exact representation \eqref{eq:superposition-identity}:
\[
\begin{aligned}
 u^1_{0}(t)={}&\ee^{\nu t\DeltaP}u^1_{0,\rm in}
 -t\ee^{\nu t\DeltaP}u^2_{0,\rm in}\\
 &+\int_0^t(t-s)\ee^{\nu(t-s)\DeltaP}
 \left[\PP\nablaP\cdot
 \bigl(u^\perp_{0}\otimes u^\perp_{0}
 +(u^\perp_{\neq}\otimes u^\perp_{\neq})_{0}\bigr)(s)\right]_y\,ds\\
 &-\int_0^t\ee^{\nu(t-s)\DeltaP}\nablaP\cdot
 \bigl(u^\perp_{0}u^1_{0}
 +(u^\perp_{\neq}u^1_{\neq})_{0}\bigr)(s)\,ds,
\end{aligned}
\]
where $[\cdot]_y$ denotes the wall-normal component.
The extra factor $t-s$ records the accumulated lift-up. Crucially,
this factor is balanced by the heat decay of a divergence-form source:
\[
 \sup_{\tau>0}
 \|\tau\ee^{\nu\tau\DeltaP}\PP\nablaP\cdot F\|_{L^2}
 \lesssim\nu^{-1}\|F\|_{L^1};
\]
see \eqref{eq:S-low}. The nonlinear lift-up is therefore controlled
by the time integral of the quadratic stress, which is finite by
the preceding estimates. In particular, the roll self-interaction
contributes at most $O(\varepsilon^2/\nu^2)$ to the streak.
The higher-order estimate \eqref{eq:Sbound} preserves this same
structure when derivatives are included.

These zero-mode estimates supply the coefficient bounds needed in
the Wei--Zhang nonzero-mode framework. The streak-adapted coordinates
and good unknowns control mixing and enhanced dissipation; their
linear frequency bounds remain uniform as the spanwise frequency
approaches zero. The mixed nonlinear estimates pair higher-order factors with
lower-order factors, as in
\eqref{eq:mixed-Y2-final}. This preserves the two distinct energy
scales
\[
 E_{2}(T)+E_{4}(T)+E_{6}(T)\lesssim\varepsilon,
 \qquad
 E_{1}(T)+E_{3}(T)+E_{5}(T)\lesssim\frac{\varepsilon}{\nu};
\]
see \eqref{eq:full-energy-final} for the functionals defined in
\eqref{eq:A-final}--\eqref{eq:E6-final}.
After short-time smoothing of the $H^2$ data, the energies are
initialized at physical time one. We close the lower-order
nonzero-mode and roll bounds first, followed by the higher-order
bounds and the streak. The resulting bootstrap improvement is proved
in \Cref{lem:six-functional-bootstrap-final} for
$\varepsilon\leq c_*\nu$.

The remainder of the paper is organized as follows. Section~\ref{sec:full-system} defines the six energies, states their
closure estimates, and proves both main theorems and the corollaries.
Section~\ref{sec:zero-endpoint} develops the endpoint zero-mode
estimates, including linear criticality and a physical-space
sufficient condition. The nonlinear interactions and energy
estimates are treated in Sections~\ref{sec:interactions-final}
and~\ref{sec:energy-estimates}. The linear, good-unknown, and product
estimates are collected in Appendices~\ref{sec:nonzero-linear}--\ref{app:functional-tools};
the local theory and approximation arguments are given in
Appendix~\ref{app:local-H4}.

\section{Energy framework and proof of the main results}
\label{sec:full-system}
% The notation and six energy functionals are collected first. Their closure
% estimates give global stability and the quantitative bounds in
% \Cref{thm:full-couette-final,thm:global-decay-final}.

\subsection{Notation and auxiliary equations}
\label{sec:framework-notation}
We first collect the notation used in this paper.

\begin{enumerate}[label=\textup{(\roman*)},leftmargin=2em,itemsep=0.6\baselineskip,topsep=0.4\baselineskip]
\item \emph{Space-time norms and projections.}
Write $x_\perp=(y,z)$. Spatial $L^p$ and $H^m$ norms of
streamwise-independent fields are taken on $\mathbb R^2$, and those
of three-dimensional fields on $\Omega$. We write $L_t^pL_x^q$ for
the corresponding mixed norm over $[0,T]$; $L^q(\mathbb T_x)$ refers
to the streamwise variable alone. The two interpretations of a
zero-mode norm agree because $|\mathbb T|=1$. The constant $C$ is independent of $\nu$ and $T$ unless stated otherwise and may change between estimates. The full and transverse Leray projections are
\begin{equation*}
 \mathbb P:=I-\nabla\Delta^{-1}\nabla\cdot,
 \qquad
 \PP:=I-\nablaP\DeltaP^{-1}\nablaP\cdot.
\end{equation*}

\item \emph{The Besov norm.}
Choose radial functions $\chi,\varphi\in C_c^\infty(\mathbb R^2)$, supported in a ball and an annulus, respectively, such that
\begin{equation*}
 \chi(\xi)+\sum_{j\geq0}\varphi(2^{-j}\xi)=1,
 \qquad
 \sum_{j\in\mathbb Z}\varphi(2^{-j}\xi)=1\quad(\xi\neq0).
\end{equation*}
Let $\dot\Delta_jf=\mathcal F^{-1}[\varphi(2^{-j}\cdot)\widehat f]$ and
\begin{equation}
 \norm{f}{\dot B^{-2}_{2,\infty}}
 :=\sup_{j\in\mathbb Z}2^{-2j}\norm{\dot\Delta_jf}{L^2(\mathbb R^2)}.
 \label{eq:Bdef}
\end{equation}
% Use the homogeneous realization on $\mathcal S'(\mathbb R^2)/\mathcal P$, where $\mathcal P$ is the space of polynomials. Solution variables also belong to $L^2$, which fixes their representatives. The dyadic definition also applies to divergences of $L^1$ tensors, as in \Cref{lem:div-B}. For the Littlewood--Paley theory and Bernstein inequalities, see \cite{BahouriCheminDanchin2011}.

\item \emph{Couette coordinates.}
Let $L_{0}:=\pt-\nu\Delta+y\px$ and define
\begin{equation}
 \widetilde f(t,X,Y,Z):=f(t,X+tY,Y,Z).
 \label{eq:couette-transform}
\end{equation}
The transformed derivatives are
\begin{equation*}
 \nabla_L:=(\pa_X,\pa_Y-t\pa_X,\pa_Z),
 \qquad \Delta_L:=\nabla_L\cdot\nabla_L,
\end{equation*}
so that $(\pt+y\px)f\longleftrightarrow\pt\widetilde f$ and $\nabla f\longleftrightarrow\nabla_L\widetilde f$.

\item \emph{Energy norms.}
The mode subscripts follow \eqref{eq:mode-convention}. For an $x$-nonzero function $f$, $0\leq\gamma\leq1$, and $I=[0,T]$, denote
\begin{align}
 \norm{f}{X_\gamma(I)}^2:={}&
 \norm{\ee^{\gamma\nu^{1/3}t}f}{L^\infty(I;L^2)}^2
 +\norm{\ee^{\gamma\nu^{1/3}t}\nabla\Delta^{-1}\px f}{L^2(I;L^2)}^2
 \nonumber\\
 &+\nu^{1/3}\norm{\ee^{\gamma\nu^{1/3}t}f}{L^2(I;L^2)}^2
 +\nu\norm{\ee^{\gamma\nu^{1/3}t}\nabla f}{L^2(I;L^2)}^2.
 \label{eq:Xgamma-def}
\end{align}
We omit $I$ when it is $[0,T]$. The four terms give uniform energy, inviscid damping, enhanced dissipation, and viscous dissipation. The inverse Laplacian acts on the $x$-nonzero subspace.
For either mode class and an integer $m\geq0$, define
\begin{equation}
 \|f\|_{Y^m}^2:=\|f\|_{L_t^\infty H^m}^2
 +\nu\|\nabla f\|_{L_t^2H^m}^2.
 \label{eq:Ym-final}
\end{equation}
For $f=f_{\neq}$, the streamwise Poincar\'e inequality yields
\begin{equation*}
 \nu^{1/2}\|f\|_{L_t^2H^{m+1}}
 \leq C\nu^{1/2}\|\nabla f\|_{L_t^2H^m}
 \leq C\|f\|_{Y^m}.
\end{equation*}
% For $f=f_{0}$, the dissipative term controls only $\nu^{1/2}\|\nabla_\perp f\|_{L_t^2H^m}$. The undifferentiated $L_t^2L^2$ component must be estimated separately whenever a full $L_t^2H^{m+1}$ bound is needed.
\end{enumerate}

We now turn to the equations. Taking the streamwise average of \eqref{eq:pert} gives the roll and streak systems
\begin{equation*}
 \pt u^\perp_{0}-\nu\DeltaP u^\perp_{0}
 =-\PP\nablaP\cdot\bigl(u^\perp_{0}\otimes u^\perp_{0}
 +(u^\perp_{\neq}\otimes u^\perp_{\neq})_{0}\bigr),
 \qquad \nablaP\cdot u^\perp_{0}=0,
\end{equation*}
\begin{equation}
 \pt u^1_{0}-\nu\DeltaP u^1_{0}
 =-u^2_{0}-\nablaP\cdot\bigl(u^\perp_{0}u^1_{0}
 +(u^\perp_{\neq}u^1_{\neq})_{0}\bigr).
 \label{eq:a-exact}
\end{equation}
% The quadratic roll stress belongs naturally to $L^1(\mathbb R^2)$. Its divergence is controlled by
% \begin{equation}
%  \norm{\PP\nablaP\cdot F}{\dot B^{-2}_{2,\infty}(\mathbb R^2)}
%  \lesssim\norm{F}{L^1(\mathbb R^2)}.
%  \label{eq:endpoint-div-intro}
% \end{equation}
% The same Besov norm controls the lift-up term in the streak equation: the heat-semigroup characterization
% \begin{equation*}
%  \norm{f}{\dot B^{-2}_{2,\infty}}
%  \simeq\sup_{\tau>0}\tau\norm{\ee^{\tau\DeltaP}f}{L^2}
% \end{equation*}
% implies
% \begin{equation}
%  \sup_{t>0}\norm{t\ee^{\nu t\DeltaP}f}{L^2}
%  \lesssim\nu^{-1}\norm{f}{\dot B^{-2}_{2,\infty}}.
%  \label{eq:liftup-besov-intro}
% \end{equation}
% Thus the space propagated by the roll forcing is also the input space for lift-up. Section~\ref{sec:zero-endpoint} proves these endpoint estimates and combines them with Fourier splitting and the two-level Duhamel formula.
For the nonzero modes, by \eqref{eq:pert} we obtain
\begin{equation}
\begin{aligned}
 &\partial_t u_{\neq}-\nu\Delta u_{\neq}
 +y\partial_xu_{\neq}+u^2_{\neq}e_1+\nabla p_{\neq}\\
 &\qquad=-u^1_{0}\partial_xu_{\neq}
 -u^\perp_{0}\cdot\nabla_\perp u_{\neq}
 -u^\perp_{\neq}\cdot\nabla_\perp u_{0}
 -(u_{\neq}\cdot\nabla u_{\neq})_{\neq},
 \qquad \nabla\cdot u_{\neq}=0.
\end{aligned}
\label{eq:nonzero-before-streak}
\end{equation}
The following construction is inspired by the work of Wei--Zhang \cite{WeiZhang2021}. One may absorb the streak term $u^1_{0}\partial_xu_{\neq}$
into the transport and let
\begin{equation*}
 V(t,y,z):=y+u^1_{0}(t,y,z),
 \qquad
 L_V:=\partial_t-\nu\Delta+V\partial_x.
\end{equation*}
If $\|u^1_{0}\|_{L^\infty(0,T;H^4(\R^2))}$ is sufficiently small,
Sobolev embedding gives $\|\nablaP u^1_{0}\|_{L^\infty}\leq1/2$,
and hence $V_y\geq1/2$. The shear coordinates and the tangential
derivative can then be defined by
\begin{equation*}
 \Psi_t(x,y,z):=(x,V(t,y,z),z),
 \qquad \kappa:=\frac{V_z}{V_y},
 \qquad D_z^V:=\pz-\kappa\py.
\end{equation*}
In particular, $D_z^VV=0$.

For the pressure, taking the divergence of \eqref{eq:pert} gives, with repeated indices summed over $i,j\in\{1,2,3\}$ and $\alpha,\beta\in\{2,3\}$,
\begin{equation*}
 \Delta p=-2\partial_xu^2_{\neq}
 -\partial_i u^j\partial_j u^i.
\end{equation*}
We decompose the pressure as
\begin{equation*}
 p=p_{0}+p_{\neq},
 \qquad
 p_{\neq}=p^V+p^{0\neq}+p^{\neq\neq}.
\end{equation*}
The zero-mode pressure contains both the zero--zero interaction
and the zero-mode output of the wave--wave interaction:
\begin{equation*}
 \Delta_\perp p_{0}
 =-\partial_\alpha\partial_\beta
 \bigl(u^\alpha_{0}u^\beta_{0}
 +(u^\alpha_{\neq}u^\beta_{\neq})_{0}\bigr).
\end{equation*}
For the nonzero pressure, $p^V$ combines the linear Couette
contribution with the streak--wave interaction.
The component $p^{0\neq}$ comes from the roll--wave interaction,
and $p^{\neq\neq}$ is the nonzero output of the wave--wave interaction.
Using $\partial_xu_{0}=0$, we obtain
\begin{align}
 \Delta p^V
 &=-2\partial_x\bigl(V_yu^2_{\neq}+V_zu^3_{\neq}\bigr),
 \label{eq:pV-final}\\
 \Delta p^{0\neq}
 &=-2\partial_\alpha u^\beta_{0}
       \partial_\beta u^\alpha_{\neq},
 \label{eq:pR-final}\\
 \Delta p^{\neq\neq}
 &=-\bigl(\partial_i u^j_{\neq}
          \partial_j u^i_{\neq}\bigr)_{\neq}.
 \label{eq:pW-final}
\end{align}
For comparison with constant Couette transport, we further write
$p^V=p^C+p^{\rm st}$, where
\begin{align*}
 \Delta p^C
 &=-2\partial_xu^2_{\neq},\\
 \Delta p^{\rm st}
 &=-2\partial_x\bigl(
   (\partial_yu^1_{0})u^2_{\neq}
   +(\partial_zu^1_{0})u^3_{\neq}\bigr).
\end{align*}

The nonlinear forcing is organized by its two input modes. For $j=2,3$, define
\begin{align}
 g_{0\neq}^j&:=u^\perp_{0}\cdot\nabla_\perp u^j_{\neq},\qquad
 g_{\neq0}^j:=u^\perp_{\neq}\cdot\nabla_\perp u^j_{0}, \nonumber\\
 g_{\neq\neq}^j&:=(u_{\neq}\cdot\nabla u^j_{\neq})_{\neq}
 +\partial_jp^{\neq\neq}, \qquad g^j:=g_{0\neq}^j+g_{\neq0}^j+g_{\neq\neq}^j
      +\partial_jp^{0\neq}.\label{eq:gj-final}
\end{align}
The transverse nonzero-mode equations then read
\begin{equation}
 L_Vu^j_{\neq}+\partial_jp^V=-g^j,
 \qquad j\in\{2,3\}.
 \label{eq:nonzero-transverse-final}
\end{equation}
The pressure source involves the combination
$V_yu^2_{\neq}+V_zu^3_{\neq}
=V_y(u^2_{\neq}+\kappa u^3_{\neq})$.
This suggests combining the two transverse components. To formulate the construction for both the linear and nonlinear estimates, let $(w^2,w^3)$ be a transverse field with zero streamwise average; in the nonlinear application, $w^j=u^j_{\neq}$. Define
\begin{equation*}
 q:=w^2+\kappa w^3.
\end{equation*}
The corresponding pressure coupling is incorporated into the operator
\begin{equation*}
 \mathscr L_Vq
 :=L_Vq-2(\py+\kappa\pz)\Delta^{-1}(V_y\px q).
\end{equation*}
The variation of $\kappa$ also produces viscous terms involving
$\nablaP\kappa\cdot\nablaP w^3$.
To separate their normal and tangential contributions, write
\begin{equation*}
 \nablaP\kappa\cdot\nablaP
 =\rho_1\nablaP V\cdot\nablaP
  +\rho_2(\pz-\kappa\py),
\end{equation*}
where
\begin{equation}
 \rho_1:=\frac{\py\kappa+\kappa\pz\kappa}
                  {V_y(1+\kappa^2)},
 \qquad
 \rho_2:=\frac{\pz\kappa-\kappa\py\kappa}
                  {1+\kappa^2}.
 \label{eq:rho-definitions}
\end{equation}
For the normal contribution, introduce the parabolic correctors
\begin{align}
 L_Vq^{(2)}
 &=-\rho_1\nablaP V\cdot\nablaP w^3, \qquad
  q^{(2)}(0)=0,
 \label{eq:q2-corrector}\\
 L_Vq^{(3)}
 &=-\nablaP V\cdot\nablaP w^3, \qquad
  q^{(3)}(0)=0,
 \label{eq:q3-corrector}
\end{align}
and define the second good unknown by
\begin{equation}
 q^{(1)}:=q-\nu q^{(2)}.
 \label{eq:q1-second-good}
\end{equation}
% The equations for these auxiliary quantities and the required commutator estimates are recorded in Appendix~\ref{sec:second-good}. 
Fix a positive constant $\gamma$, independent of $\nu$ and $T$, such that
\begin{equation*}
 0<4\gamma<\min\{1,\gamma_0,\gamma_1,\gamma_2,\gamma_3,
 \sigma_{\rm c},\sigma_{\rm h}\},\qquad \lambda_*:=\gamma.
\end{equation*}
Here $\gamma_0$ is chosen below both weight thresholds denoted by this
symbol in Appendix~\ref{sec:nonzero-linear}; the other thresholds are
those of Appendices~\ref{sec:nonzero-linear} and~\ref{sec:second-good}.
Thus all required linear estimates apply at weights $2\gamma$ and
$3\gamma$. This choice is fixed before the bootstrap smallness constants.
To measure the nonlinear forcing $(g^2,g^3)$ in
\eqref{eq:nonzero-transverse-final}, define
\begin{align}
 \mathfrak F(T):={}&\nu^{-1/2}\|\ee^{3\gamma\nu^{1/3}t}\nabla(g^2+\kappa g^3)\|_{L_t^2L_x^2}
 +\nu^{-1/2}\|\ee^{3\gamma\nu^{1/3}t}\partial_xg^3\|_{L_t^2L_x^2}
 \nonumber\\
 &+\nu^{-1/6}\|\ee^{3\gamma\nu^{1/3}t}\nabla g^3\|_{L_t^2L_x^2}.
 \label{eq:Fforcing-final}
\end{align}

\subsection{Energy functionals and initialization}
We use three energies for the zero modes and three for the nonzero streamwise modes. All time norms below are taken over $[0,T]$. For the roll, let $\nabla_\perp^\perp=(-\partial_z,\partial_y)$ and define
\begin{equation*}
 \omega_{0}:=\partial_yu^3_{0}-\partial_zu^2_{0},\qquad
 \mu_\nu(t):=\min\{\nu^{2/3}+\nu t,1\},\qquad
 m_\nu(t):=\min\{\nu(1+t),1\}.
\end{equation*}
% For a zero-mode field $f$ and a scalar time weight $\vartheta=\vartheta(t)$, the convention
% \begin{equation*}
%  \|\vartheta f\|_{Y^0}^2
%  :=\|\vartheta f\|_{L_t^\infty L^2}^2
%  +\nu\|\vartheta\nabla_\perp f\|_{L_t^2L^2}^2
% \end{equation*}
% agrees with \eqref{eq:Ym-final}; derivatives of the weight occur only in the weighted energy identity. 
For the transverse field introduced above, denote
\begin{equation*}
 \mathcal E_{\perp,\gamma}(T)
 :=\sum_{j=2}^3\left(
 \norm{\px^2 u^j_{\neq}}{X_\gamma}^2
 +\norm{\px D_z^V u^j_{\neq}}{X_\gamma}^2\right)
 +\norm{\px\nabla (u^2_{\neq}+\kappa u^3_{\neq})}{X_\gamma}^2.
\end{equation*}
Define
\begin{align}
 E_{1}(T):={}&\|u^1_{0}\|_{L_t^\infty H^4}
 +\nu^{1/2}\|\nabla_\perp\Delta_\perp^2u^1_{0}\|_{L_t^2L^2}
 +\nu^{-1}\|\partial_tu^1_{0}\|_{L_t^\infty H^2},
 \label{eq:A-final}\\
 E_{2}(T):={}&\|u^\perp_{0}\|_{Y^0}+\|\omega_{0}\|_{Y^0}
 +\|\partial_z\omega_{0}\|_{Y^0}
 +\|\mu_\nu^{1/2}\partial_y\omega_{0}\|_{Y^0}
 \nonumber\\*
 &+\sup_{0\leq t\leq T}\|u^\perp_{0}(t)\|_{\dot B^{-2}_{2,\infty}}
 +\nu^{1/2}\|u^\perp_{0}\|_{L_t^2L^2},
 \label{eq:E2-definition}\\
 E_{3}(T):={}&\|u^\perp_{0}\|_{L_t^\infty H^4(\mathbb R^2)}
 +\nu^{1/2}\|\nabla_\perp u^\perp_{0}\|_{L_t^2H^4(\mathbb R^2)},
 \label{eq:E3-definition}\\
 E_{4}(T):={}&\|\Delta u^2_{\neq}\|_{X_{2\gamma}}
 +\|(\partial_x^2+\partial_z^2)u^3_{\neq}\|_{X_{2\gamma}}
 +\nu^{2/3}\|\Delta u^3_{\neq}\|_{X_{3\gamma}},
 \label{eq:E4-definition}\\
 E_{5}(T):={}&\|\ee^{2\gamma\nu^{1/3}t}(\partial_x,\partial_z)u_{\neq}\|_{L_t^\infty H^3}
 +\nu^{1/2}\|\ee^{2\gamma\nu^{1/3}t}\nabla(\partial_x,\partial_z)u_{\neq}\|_{L_t^2H^3},
 \label{eq:E5-definition}\\
 E_{6}(T)^2:={}&\mathcal E_{\perp,3\gamma}(T)
 +\nu^{4/3}\|\Delta u^3_{\neq}\|_{X_{3\gamma}}^2.
 \label{eq:E6-final}
\end{align}
When $T$ is fixed, we write $E_{j}$ for $E_{j}(T)$.

The zero-mode energies distinguish the streak from the low- and high-order roll estimates. The $H^4$ component of $E_{1}$ controls the geometry of $V=y+u^1_{0}$, while the factor $\nu^{-1}$ on the time derivative records its slow variation. The roll energy $E_{2}$ combines velocity, vorticity, and anisotropic vorticity derivatives. The weight $\mu_\nu^{1/2}$ allows weaker initial control of the highest wall-normal derivative. Its Besov component propagates the transverse low frequencies, and Fourier splitting supplies the separate $\nu^{1/2}L_t^2L^2$ term. The higher-order roll energy $E_{3}$ supplies the coefficients needed in tame product estimates without weakening the low-order bound for $E_{2}$. For the nonzero modes, $E_{4}$ follows the primary coupled variables $\Delta u^2_{\neq}$ and $-(\partial_x^2+\partial_z^2)u^3_{\neq}$. The latter operator has multiplier $k^2+\ell^2$, whose lower bound uses only $k\neq0$. The factor $\nu^{2/3}$ compensates for the weaker estimate of $\Delta u^3_{\neq}$. The energy $E_{5}$ adds fourth-order control through the physical derivatives $(\partial_x,\partial_z)$. Finally, $E_{6}$ uses the tangent derivative $D_z^V$ and the combination $q$ to exploit the pressure structure. Its estimate is obtained through the parabolic correctors defined above. 
% These choices separate the two sizes required by the bootstrap:
% \begin{equation*}
%  E_{2}+E_{4}+E_{6}\lesssim\varepsilon,
%  \qquad E_{1}+E_{3}+E_{5}\lesssim\frac{\varepsilon}{\nu}.
% \end{equation*}
% The proof closes $(E_{6},E_{4})$ together, followed by $E_{2}$, $E_{5}$, $E_{3}$, and $E_{1}$. Thus the numbering groups the modes, while the argument follows the dependence of the estimates.

The Biot--Savart identity $u^\perp_{0}=\nabla_\perp^\perp\Delta_\perp^{-1}\omega_{0}$, together with the separate $L^2$ control of $u^\perp_{0}$, gives the following recovery bounds:
	\begin{align}
		&\|u^2_{0}(t)\|_{H^2}+\|u^3_{0}(t)\|_{H^1}+\|\partial_zu^3_{0}(t)\|_{H^1}
		+\mu_\nu(t)^{1/2}\|u^3_{0}(t)\|_{H^2}
		\leq C E_{2}(T),
		\label{eq:roll-recovery-final}
		\\
		&\nu^{1/2}\Big(
		\|\nabla_\perp u^2_{0}\|_{L_t^2H^2}
		+\|\nabla_\perp u^3_{0}\|_{L_t^2H^1}
		+\|\nabla_\perp\partial_zu^3_{0}\|_{L_t^2H^1}
		+\|\mu_\nu^{1/2}\nabla_\perp u^3_{0}\|_{L_t^2H^2}
		\Big)
		\leq C E_{2}(T).
		\label{eq:roll-recovery-time-final}
	\end{align}
	% Indeed, in transverse Fourier variables,
	% \begin{equation*}
	% 	|\xi_\perp|^2\widehat{u^2_{0}}=i\ell\widehat\omega_{0},
	% 	\qquad
	% 	|\xi_\perp|^2\widehat{u^3_{0}}=-i\eta\widehat\omega_{0}.
	% \end{equation*}
	% The first identity is controlled by $\omega_{0}$ and $\partial_z\omega_{0}$. The second identity is controlled by $\omega_{0}$ at order one and by $\partial_y\omega_{0}$ at order two; this is precisely why the factor $\mu_\nu^{1/2}$ is attached to the highest $y$ derivative. The time-integrated estimates follow from the dissipative components of the corresponding $Y^0$ norms.
	The definitions and the velocity reconstruction \eqref{eq:velocity-reconstruction-final} imply
\begin{align}
 \|u_{\neq}(t)\|_{L^2}&\leq C\ee^{-2\gamma\nu^{1/3}t}E_{4}(T),\notag\\
 \|u(t)\|_{H^2}&\leq C\bigl(E_{1}(T)+E_{3}(T)+E_{5}(T)\bigr),
 \qquad 0\leq t\leq T.
 \label{eq:physical-recovery-reorganized}
\end{align}
% Indeed, the first bound follows from the $s=0$ reconstruction and the streamwise Poincar\'e inequality; the second follows from the high-order components of the three functionals.
For the product estimates, we use the following pointwise profiles associated with $E_{4}$ and $E_{6}$:
	\begin{equation*}
		\Pi_{0}(t):=\|\Delta u^2_{\neq}(t)\|_{L^2}+\|(\partial_x^2+\partial_z^2)u^3_{\neq}(t)\|_{L^2},
		\qquad
		\Pi_1(t):=\|\nabla\Delta u^2_{\neq}(t)\|_{L^2}+\|\nabla(\partial_x^2+\partial_z^2)u^3_{\neq}(t)\|_{L^2},
	\end{equation*}
	and
	\begin{align}
		\mathcal U_{0}(t):={}&\sum_{j=2}^3\left(
		\|\partial_x^2u^j_{\neq}(t)\|_{L^2}+\|\partial_xD_z^Vu^j_{\neq}(t)\|_{L^2}\right)
		+\|\partial_x\nabla (u^2_{\neq}+\kappa u^3_{\neq})(t)\|_{L^2}
		+\nu^{2/3}\|\Delta u^3_{\neq}(t)\|_{L^2},
		\label{eq:U0-revised}
		\\
		\mathcal U_1(t):={}&\sum_{j=2}^3\left(
		\|\nabla\partial_x^2u^j_{\neq}(t)\|_{L^2}+\|\nabla\partial_xD_z^Vu^j_{\neq}(t)\|_{L^2}\right)
		+\|\nabla\partial_x\nabla (u^2_{\neq}+\kappa u^3_{\neq})(t)\|_{L^2}
		+\nu^{2/3}\|\nabla\Delta u^3_{\neq}(t)\|_{L^2}.
		\label{eq:U1-revised}
	\end{align}
	The definitions of $E_{4}$ and $E_{6}$ imply
	\begin{align}
		&\|\ee^{2\gamma\nu^{1/3}t}\Pi_{0}\|_{L_t^\infty}
		+\nu^{1/6}\|\ee^{2\gamma\nu^{1/3}t}\Pi_{0}\|_{L_t^2}
		+\nu^{1/2}\|\ee^{2\gamma\nu^{1/3}t}\Pi_1\|_{L_t^2}
		\leq C E_{4}(T),
		\label{eq:P-profiles-revised}
		\\
		&\|\ee^{3\gamma\nu^{1/3}t}\mathcal U_{0}\|_{L_t^\infty}
		+\nu^{1/6}\|\ee^{3\gamma\nu^{1/3}t}\mathcal U_{0}\|_{L_t^2}
		+\nu^{1/2}\|\ee^{3\gamma\nu^{1/3}t}\mathcal U_1\|_{L_t^2}
		\leq C E_{6}(T).
		\label{eq:U-profiles-revised}
	\end{align}

Next, we consider initialization. The long-time argument starts at physical time $t=1$, when parabolic smoothing supplies the fourth-order data. We record the interface with the local theory before stating the closure estimates.
Recall the initial size from the main theorem,
	\begin{equation*}
		\varepsilon=\|u_{\rm in}\|_{H^2(\Omega)}
		+\|u^\perp_{0,\rm in}\|_{\dot B^{-2}_{2,\infty}(\mathbb R^2)}.
	\end{equation*}
	The functionals $E_{1}$--$E_{6}$ are applied only after the fixed restart at physical time one. To keep the restart interface explicit, throughout the long-time sections we impose the following standing bounds at the restarted time origin:
	\begin{align}
		\mathcal I_{\rm lo}(0)
		&:=\|u(0)\|_{H^2}
		+\|u^\perp_{0}(0)\|_{\dot B^{-2}_{2,\infty}}
		+\|\Delta (u^2_{\neq}+\kappa u^3_{\neq})(0)\|_{L^2}+\|\Delta u^3_{\neq}(0)\|_{L^2}
		\leq C_{\rm r}\varepsilon,
		\label{eq:restarted-low-data-standing}\\
		\mathcal I_{\rm hi}(0)
		&:=\|u(0)\|_{H^4}
		+\nu^{-1}\|\partial_t u^1_{0}(0)\|_{H^2}
		\leq C_{\rm r}\frac{\varepsilon}{\nu},
		\label{eq:restarted-high-data-standing}
	\end{align}
	where $C_{\rm r}\geq1$ is a fixed absolute constant. Every initial quantity in the long-time estimates is understood under \eqref{eq:restarted-low-data-standing}--\eqref{eq:restarted-high-data-standing}. These bounds are not additional assumptions on the original physical data: they are verified from the $H^2+\dot B^{-2}_{2,\infty}$ hypothesis by the following \Cref{lem:H2-restart-initialization-final}. In particular, the unweighted fourth-order components of $E_{1}$, $E_{3}$, and $E_{5}$ are initialized from positive-time parabolic smoothing rather than imposed at physical time zero.
	
	\begin{Lem}
		\label{lem:H2-restart-initialization-final}
		Let $u_{\rm in}\in H^2(\Omega)$ be divergence free and $u^\perp_{0,\rm in}\in\dot B^{-2}_{2,\infty}(\mathbb R^2)$. Assume $\varepsilon\leq c_{\rm r}\nu$, where $c_{\rm r}>0$ is sufficiently small for the short-time theory (Subsection \ref{app:restart-proof}). Let $u$ be the short-time solution of \Cref{prop:H2-short-time-final}. Then
		\begin{equation}
			\sup_{0\le t\le2}\|u(t)\|_{H^2}^2
			+\nu\int_{0}^2\|u(t)\|_{H^3}^2\,dt\le C\varepsilon^2,
			\qquad
			\sup_{1\le t\le2}\|u(t)\|_{H^4}\le C \frac\varepsilon \nu,
			\label{eq:H2-restart-smoothing-global}
		\end{equation}
		and
		\begin{equation}
			\sup_{0\le t\le1}\|u^\perp_{0}(t)\|_{\dot B^{-2}_{2,\infty}}\le C\varepsilon.
			\label{eq:H2-restart-besov-global}
		\end{equation}
		At physical time $t=1$,
		\begin{equation}
			\|u(1)\|_{H^2}+\|u^\perp_{0}(1)\|_{\dot B^{-2}_{2,\infty}}\le C_{\rm r}\varepsilon,
			\qquad
			\|u(1)\|_{H^4}\le C_{\rm r}\frac{\varepsilon}{\nu},
			\label{eq:H2-restart-sizes-final}
		\end{equation}
		and
		\begin{equation}
			\nu^{-1}\|\partial_t u^1_{0}(1)\|_{H^2}\le C_{\rm r}\frac{\varepsilon}{\nu}.
			\label{eq:H2-restart-at-final}
		\end{equation}
		If all long-time correctors are restarted with zero initial data at physical time one and elapsed time is relabelled as $t=0$, then the standing bounds \eqref{eq:restarted-low-data-standing}--\eqref{eq:restarted-high-data-standing} hold. Moreover,
		\begin{equation}
			E_{2}(0)+E_{4}(0)+E_{6}(0)\le C_{\rm r}\varepsilon,
			\qquad
			E_{1}(0)+E_{3}(0)+E_{5}(0)\le C_{\rm r}\frac{\varepsilon}{\nu}.
			\label{eq:H2-energy-initialization-final}
		\end{equation}
	\end{Lem}
The proof of \Cref{lem:H2-restart-initialization-final} is given in Subsection~\ref{app:restart-proof}.
\begin{Rem}
\label{rem:restart-dictionary-final}
Elapsed time after physical time one is denoted by $t$. The shear and the algebraic unknown $q$ are reconstructed from the translated velocity. The correctors are solved anew with
\[
q^{(2)}(0)=q^{(3)}(0)=0,\qquad q^{(1)}(0)=q(0),\qquad
(\Delta u^3_{\neq}-2\partial_xq^{(3)})(0)=\Delta u^3_{\neq}(0).
\]
The low-order initial scale is $\varepsilon$, whereas the high-order scale is $\varepsilon/\nu$.
\end{Rem}
	
	\begin{Lem}
		\label{lem:restart-physical-time-final}
		Let $t\geq1$ denote physical time, so that the elapsed time after the restart is $t-1$. Then, for every $\lambda>0$,
		\begin{equation*}
			e^{-\lambda\nu^{1/3}(t-1)}\leq e^\lambda e^{-\lambda\nu^{1/3}t},
		\end{equation*}
		and
		\begin{equation}
			\mu_\nu(t)\leq2\mu_\nu(t-1).
			\label{eq:restart-mu-comparison}
		\end{equation}
		Consequently, every restarted unweighted $Y^m$ estimate on elapsed times $t-1\geq0$ is identical to the corresponding estimate on physical time $t\geq1$, while a factor $\mu_\nu^{1/2}$ loses at most the absolute factor $\sqrt2$ under this conversion.
	\end{Lem}
	
	\begin{proof}
		The exponential estimate follows from $0<\nu\leq1$. Since $\nu^{2/3}+\nu(t-1)\geq\nu$,
		\begin{equation*}
\begin{aligned}
\mu_\nu(t)&=\min\{(\nu^{2/3}+\nu(t-1))+\nu,1\}\leq\min\{2(\nu^{2/3}+\nu(t-1)),1\}\\&\leq2\min\{(\nu^{2/3}+\nu(t-1)),1\}=2\mu_\nu(t-1).
\end{aligned}
\end{equation*}
		The last assertion is then immediate from the definition of $Y^m$.
	\end{proof}
	
	\subsection{Closure estimates and proof of the main results}
\label{sec:global-bootstrap-final}
The following propositions are the inputs to the bootstrap. We work on a smooth restarted interval satisfying \eqref{eq:restarted-low-data-standing}--\eqref{eq:restarted-high-data-standing}. The geometric threshold is chosen so that $E_{1}(T)+\varepsilon/\nu\leq\delta_*$ implies every variable-shear smallness condition used in Section~\ref{sec:interactions-final} and Appendices~\ref{sec:nonzero-linear} and~\ref{sec:second-good}. This common condition is understood in the three nonzero-mode propositions below; the roll and streak propositions state their additional smallness assumptions explicitly.

\begin{Prop}
		\label{prop:E6-nonlinear-final}
		Assume $E_{1}(T)+\varepsilon/\nu\leq\delta_*$, with $\delta_*>0$ sufficiently small, and the restarted-data bounds \eqref{eq:restarted-low-data-standing}--\eqref{eq:restarted-high-data-standing},
		\begin{equation}
			E_{6}(T)
			\leq C\left(\|(u^2_{\neq},u^3_{\neq})(0)\|_{H^2}+\|\Delta q(0)\|_{L^2}\right)
			+C\frac{E_{2}(T)}{\nu}E_{6}(T)+C\frac{E_{4}(T)^2}{\nu}.
			\label{eq:E6-nonlinear-final}
		\end{equation}
	\end{Prop}

\begin{Prop}
		\label{prop:E3-primary-final}
		Assume \eqref{eq:restarted-low-data-standing}--\eqref{eq:restarted-high-data-standing}. There exists $C>0$ such that
		\begin{align}
			E_{4}(T)
			\leq{}&C\left(\|\Delta u^2_{\neq}(0)\|_{L^2}+\|u^3_{\neq}(0)\|_{H^2}\right)+C E_{6}(T)
			+C\bigl(E_{1}(T)+\frac{\varepsilon}{\nu}\bigr)\bigl(E_{4}(T)+E_{6}(T)\bigr)
			\nonumber\\*
			&+C\frac{E_{2}(T)}{\nu}E_{6}(T)
			+C\frac{E_{4}(T)^2}{\nu}.
			\label{eq:E3-primary-final}
		\end{align}
	\end{Prop}

\begin{Prop}
		\label{prop:E4-closed-final}
		Assume \eqref{eq:restarted-low-data-standing}--\eqref{eq:restarted-high-data-standing}. There exists $C>0$ such that
		\begin{align}
			E_{5}(T)
			\leq{}&C\|u_{\neq}(0)\|_{H^4}+C\nu^{-1}\bigl(E_{4}(T)+E_{6}(T)\bigr)
			\nonumber\\*
			&+C\left(E_{1}(T)+\frac{\varepsilon}{\nu}+\frac{E_{2}(T)}{\nu}+\frac{E_{4}(T)}{\nu}\right)E_{5}(T)
			\nonumber\\*
			&+C\frac{E_{1}(T)+\frac{\varepsilon}{\nu}}{\nu}\bigl(E_{4}(T)+E_{6}(T)\bigr)
			+C\frac{E_{3}(T)}{\nu}E_{4}(T).
			\label{eq:E4-closed-final}
		\end{align}
	\end{Prop}

\begin{Prop}
		\label{prop:ZR-closed-final}
		Assume \eqref{eq:restarted-low-data-standing}--\eqref{eq:restarted-high-data-standing}. There exists $\delta_*>0$ such that, under the bootstrap smallness
		$\varepsilon+E_{2}(T)+E_{4}(T)\leq\delta_*\nu$,
		\begin{equation}
			E_{2}(T)\leq C\left(\|u^\perp_{0}(0)\|_{H^2}+\|u^\perp_{0}(0)\|_{\dot B^{-2}_{2,\infty}}\right)+C\nu^{-1}E_{4}(T)^2.
			\label{eq:ZR-closed-final}
		\end{equation}
	\end{Prop}

\begin{Prop}
		\label{prop:ZR4-closed-final}
		Assume \eqref{eq:restarted-low-data-standing}--\eqref{eq:restarted-high-data-standing}. There exists $\delta_*>0$ such that, if $E_{2}(T)\leq\delta_*\nu$, then
		\begin{equation}
			E_{3}(T)
			\leq C\|u^\perp_{0}(0)\|_{H^4}+C\nu^{-1}\Big(E_{4}(T)^2+E_{4}(T)E_{5}(T)\Big).
			\label{eq:ZR4-closed-final}
		\end{equation}
	\end{Prop}

\begin{Prop}
		\label{prop:A-closed-final}
		Assume \eqref{eq:restarted-low-data-standing}--\eqref{eq:restarted-high-data-standing}. Suppose also that $E_{1}(T)\leq\delta_*$ and $E_{2}(T)\leq\delta_*\nu$. Then
		\begin{align}
			E_{1}(T)
			\leq C\Big(&\|u^1_{0}(0)\|_{H^4}+\nu^{-1}\|u^2_{0}(0)\|_{\dot B^{-2}_{2,\infty}}+\frac{E_{2}(T)}{\nu}
			+\frac{E_{2}(T)^2}{\nu^2}
			+\frac{E_{4}(T)^2}{\nu^{4/3}}
			+\frac{E_{4}(T)E_{5}(T)}{\nu}\Big)
			\nonumber\\*
			&+C\frac{E_{2}(T)}{\nu}\bigl(E_{1}(T)+\frac{\varepsilon}{\nu}\bigr).
			\label{eq:A-closed-final}
		\end{align}
	\end{Prop}
% The zero-mode and nonzero-mode propositions are proved in Subsections~\ref{sec:zero-closure-final} and~\ref{sec:nonzero-closure-final}, respectively. Their order of use in the bootstrap is $(E_{6},E_{4})$, then $E_{2}$, $E_{5}$, $E_{3}$, and $E_{1}$.

With the definitions above, the proof of \Cref{thm:full-couette-final} also yields the following quantitative estimates. After restarting the long-time energies at physical time $t=1$ and relabelling elapsed time as $t$, for every $T>0$,
		\begin{align}
			E_{2}(T)+E_{4}(T)+E_{6}(T)&\leq C\varepsilon,
			\nonumber\\*
			E_{1}(T)+E_{3}(T)+E_{5}(T)&\leq C\frac{\varepsilon}{\nu}.
			\label{eq:full-energy-final}
		\end{align}
		The corresponding roll estimate in physical time is the following, with all time norms taken over $[0,\infty)$:
\begin{align}
			&\sup_{t\geq0}\|u^\perp_{0}(t)\|_{\dot B^{-2}_{2,\infty}}
			+\nu^{1/2}\|u^\perp_{0}\|_{L_t^2L^2}
			\nonumber\\*
			&\quad+\|\Delta_\perp u^2_{0}\|_{Y^0}
			+\|u^3_{0}\|_{Y^1}
			+\|\mu_\nu^{1/2}\Delta_\perp u^3_{0}\|_{Y^0}
			\leq C\varepsilon,
			\label{eq:full-roll-final}
		\end{align}
		
We first close the six functionals using the propositions just stated. The local theory, initialization, continuity, and continuation arguments are proved in Appendix~\ref{app:local-H4}.
\begin{Lem}
		\label{lem:six-functional-bootstrap-final}
		Let $0<\nu\leq1$, $\varepsilon>0$, and suppose the restarted initialization \eqref{eq:H2-energy-initialization-final} and the hypotheses of \Cref{prop:E6-nonlinear-final,prop:E3-primary-final,prop:E4-closed-final,prop:ZR4-closed-final,prop:ZR-closed-final,prop:A-closed-final} hold on $[0,T]$. Let $C_{\rm boot}$ dominate the constants in these propositions and let $M\geq M_{0}(C_{\rm boot},C_{\rm r})$. Assume
		\begin{equation*}
			E_{2}+E_{4}+E_{6}\leq M\varepsilon,
			\qquad
			E_{1}+E_{3}+E_{5}\leq M\frac{\varepsilon}{\nu}.
		\end{equation*}
		There exists $\delta_{\rm b}=\delta_{\rm b}(M,C_{\rm boot},C_{\rm r})>0$ such that, if $0<\frac{\varepsilon}{\nu}\leq\delta_{\rm b}$, then
		\begin{equation}
			E_{2}+E_{4}+E_{6}\leq\frac M2\varepsilon,
			\qquad
			E_{1}+E_{3}+E_{5}\leq\frac M2\frac{\varepsilon}{\nu}.
			\label{eq:six-functional-bootstrap-improvement}
		\end{equation}
	\end{Lem}
    
    \begin{Rem}
        The order of the proof is
		\begin{equation*}
			(E_{6},E_{4})\longrightarrow E_{2}\longrightarrow E_{5}\longrightarrow E_{3}\longrightarrow E_{1}.
		\end{equation*}
    \end{Rem}
	
	\begin{proof}
		All functionals below are evaluated at the same terminal time $T$. The $E_{6}$ estimate gives
		\begin{equation}
			E_{6}\leq C\varepsilon+C M\frac{\varepsilon}{\nu} E_{6}+C M\frac{\varepsilon}{\nu} E_{4},
			\label{eq:six-functional-E6-first}
		\end{equation}
		because $E_{2}/\nu\leq M\frac{\varepsilon}{\nu}$ and $E_{4}^2/\nu\leq M\frac{\varepsilon}{\nu} E_{4}$. After absorbing the middle term, substitute \eqref{eq:six-functional-E6-first} into the $E_{4}$ estimate. Since $E_{1}+\frac{\varepsilon}{\nu}\leq(M+1)\frac{\varepsilon}{\nu}$, every coefficient multiplying $E_{4}$ or $E_{6}$ is $O_M(\frac{\varepsilon}{\nu})$. For $\delta_{\rm b}$ sufficiently small, the resulting two-by-two low-order system can be absorbed and gives
		\begin{equation}
			(E_{4}+E_{6})\leq C_1\varepsilon\leq\frac M8\varepsilon.
			\label{eq:H2-E36-improve-final}
		\end{equation}
		Here and below $M_{0}$ is enlarged so that each fixed constant $C_j$ is at most $M/8$. This first step uses only the bootstrap bounds for the high-order quantities, not any high-order improvement.
		
		The closed roll estimate now uses the already improved $E_{4}$ and yields
		\begin{equation}
			E_{2}\leq C\varepsilon+C\nu^{-1}E_{4}^2
			\leq C_2\varepsilon+C M\frac{\varepsilon}{\nu}\,E_{4}
			\leq\frac M8\varepsilon.
			\label{eq:H2-ZR-improve-final}
		\end{equation}
		This completes the low-order block before any high-order functional is improved.
		
		For $E_{5}$, the genuine restarted initial term satisfies $\|u_{\neq}(0)\|_{H^4}=O(\frac{\varepsilon}{\nu})$. The already improved low-order estimate gives $\nu^{-1}(E_{4}+E_{6})=O(\frac{\varepsilon}{\nu})$. The coefficient of $E_{5}$ is
		\begin{equation*}
			E_{1}+\frac{\varepsilon}{\nu}+\frac{E_{2}}{\nu}+\frac{E_{4}}{\nu}=O_M(\frac{\varepsilon}{\nu}),
		\end{equation*}
		while the remaining couplings satisfy
		\begin{equation*}
			\frac{E_{1}+\frac{\varepsilon}{\nu}}{\nu}(E_{4}+E_{6})=O_M(\frac{\varepsilon^{2}}{\nu^{2}}),
			\qquad
			\frac{E_{3}}{\nu}E_{4}=O_M(\frac{\varepsilon^{2}}{\nu^{2}}).
		\end{equation*}
		After absorption,
		\begin{equation}
			E_{5}\leq C_3\frac{\varepsilon}{\nu}+C_M\frac{\varepsilon^{2}}{\nu^{2}}\leq\frac M8\frac{\varepsilon}{\nu}.
			\label{eq:H2-E4-improve-final}
		\end{equation}
		The high-roll estimate is now evaluated using the improved $E_{4}$ and $E_{5}$:
		\begin{align}
			E_{3}
			\leq C\|u^\perp_{0}(0)\|_{H^4}+C\nu^{-1}\bigl(E_{4}^2+E_{4}E_{5}\bigr)
			\leq C_4\frac{\varepsilon}{\nu}+C_M\frac{\varepsilon^{2}}{\nu^{2}}\leq\frac M8\frac{\varepsilon}{\nu}.
			\label{eq:H2-ZR4-improve-final}
		\end{align}
		Finally, the restarted data give
		\begin{equation*}
			\|u^1_{0}(0)\|_{H^4}+\nu^{-1}\|u^2_{0}(0)\|_{\dot B^{-2}_{2,\infty}}=O(\frac{\varepsilon}{\nu}).
		\end{equation*}
		Using the improved $E_{2}$, $E_{4}$, and $E_{5}$ in \Cref{prop:A-closed-final}, together with the still-valid bootstrap bound for $E_{1}$ only in the absorbable last term, gives
		\begin{equation*}
			E_{1}\leq C_5(\frac{\varepsilon}{\nu}+\frac{\varepsilon^{2}}{\nu^{2}})+C_5\frac{\varepsilon}{\nu} E_{1}.
		\end{equation*}
		After one final absorption,
		\begin{equation}
			E_{1}\leq\frac M8\frac{\varepsilon}{\nu}.
			\label{eq:H2-A-improve-final}
		\end{equation}
		Adding \eqref{eq:H2-E36-improve-final} and \eqref{eq:H2-ZR-improve-final}, and separately adding \eqref{eq:H2-E4-improve-final}, \eqref{eq:H2-ZR4-improve-final}, and \eqref{eq:H2-A-improve-final}, proves \eqref{eq:six-functional-bootstrap-improvement}.
	\end{proof}
	
	The global stability theorem is proved first. The decay estimates and
	the two sufficient initial-data conditions are then deduced from the
	same energy bounds.
	
	\begin{proof}[Proof of \Cref{thm:full-couette-final}]
		Let $c_{\rm loc}$ be the smallness constant in \Cref{prop:H2-short-time-final}. For the moment assume $c_*\leq\min\{c_{\rm loc},c_{\rm r}\}$; its final size will be fixed below. Then \Cref{lem:H2-restart-initialization-final} applies. Translate physical time $t=1$ to the new origin, denote elapsed time by $t$, and use the same notation $u$ for the translated velocity. Restart all auxiliary correctors with zero initial data. Then \eqref{eq:restarted-low-data-standing}--\eqref{eq:restarted-high-data-standing} and \eqref{eq:H2-energy-initialization-final} hold with the absolute constant $C_{\rm r}$.
		
		The closure propositions are now applied to the restarted solution. Let $C_{\rm boot}$ dominate the constants in \Cref{prop:E6-nonlinear-final,prop:E3-primary-final,prop:E4-closed-final,prop:ZR4-closed-final,prop:ZR-closed-final,prop:A-closed-final}, and let $\delta_{\rm adm}$ be the minimum of all geometric, modal, interaction, and zero-mode smallness thresholds. The non-circular order of choices is
		\begin{equation*}
			\text{linear modal constants}\longrightarrow(\gamma,\lambda_*)\longrightarrow(\delta_{\rm adm},C_{\rm boot})\longrightarrow M\longrightarrow c_*.
		\end{equation*}
		Choose
		\begin{equation*}
			M\geq\max\left\{M_{0}(C_{\rm boot},C_{\rm r}),64C_{\rm boot}(1+C_{\rm boot}),16C_{\rm r}\right\},
		\end{equation*}
		and then choose $c_*>0$ sufficiently small, depending only on the already fixed constants and $M$, so that
		\begin{equation}
			(C_{\rm r}+M+1)c_*\ll\delta_{\rm adm},
			\qquad C_{\rm boot}(M+1)c_*\le\frac1{64},
			\qquad c_*\leq\delta_{\rm b}(M,C_{\rm boot},C_{\rm r}).
			\label{eq:H2-cstar-choice-final}
		\end{equation}
		On the restarted maximal interval, bootstrap
		\begin{equation}
			E_{2}+E_{4}+E_{6}\le M\varepsilon,
			\qquad
			E_{1}+E_{3}+E_{5}\le M\frac{\varepsilon}{\nu}.
			\label{eq:H2-bootstrap-final}
		\end{equation}
		The initialization follows from \eqref{eq:H2-energy-initialization-final}. The admissibility assumptions in all closure propositions follow from $E_{2}/\nu\le M\frac{\varepsilon}{\nu}$ and $E_{1}+\frac{\varepsilon}{\nu}\le(M+1)\frac{\varepsilon}{\nu}$.
		
		By \Cref{lem:six-functional-bootstrap-final} and \eqref{eq:H2-cstar-choice-final}, both inequalities in \eqref{eq:H2-bootstrap-final} improve by a factor of two. The bounds for $(E_{6},E_{4})$ are closed first, followed by $E_{2}$, $E_{5}$, $E_{3}$, and $E_{1}$.
		
		By continuity of the terminal-time functionals and the positive-time $H^4$ continuation criterion, the restarted solution exists for all elapsed times. Joined with the short-time $H^2$ solution, this proves global existence in
		\begin{equation*}
			C([0,\infty);H^2)\cap L^2_{\rm loc}([0,\infty);H^3).
		\end{equation*}
		Since the improved inequalities hold on every finite terminal interval, they also prove the quantitative estimate \eqref{eq:full-energy-final}, with a universal constant depending only on the fixed bootstrap constants.
		
		For every $\tau_{0}>0$, the short-time smoothing result \eqref{eq:H2-positive-time-H4-continuity} gives $H^4$ continuity up to and across physical time one, and the restarted $H^4$ continuation theory then yields the positive-time regularity asserted in \eqref{eq:full-positive-time-H4}.
		
		Finally, \eqref{eq:physical-recovery-reorganized} and
\eqref{eq:full-energy-final} imply
$\sup_{t\geq1}\|u(t)\|_{H^2}\leq C\varepsilon/\nu$.
The short-time estimate \eqref{eq:H2-restart-smoothing-global}
controls $0\leq t\leq1$, proving
\eqref{eq:true-Couette-stability-final}. Uniqueness follows from the
local $H^2$ theory and continuation. This proves
\Cref{thm:full-couette-final}.
\end{proof}

\begin{proof}[Proof of \Cref{thm:global-decay-final}]
The global energy bounds \eqref{eq:full-energy-final} have already
been obtained. In the following estimates, $s$ denotes elapsed time
after the restart, so physical time is $t=1+s$. All $s$-norms are
taken over $[0,\infty)$, and $E_4(\infty):=\sup_{T>0}E_4(T)$.

For the nonzero modes, \eqref{eq:physical-recovery-reorganized}
and \eqref{eq:P-profiles-revised} give
\[
 \|u_{\neq}(1+s)\|_{L^2}\leq C\varepsilon\ee^{-2\gamma\nu^{1/3}s}.
\]
Moreover, velocity reconstruction
\eqref{eq:velocity-reconstruction-final} and the streamwise spectral
gap imply
\[
 \|u_{\neq}(1+s)\|_{L^2}\leq C\Pi_0(s),\qquad
 \|\nabla u_{\neq}(1+s)\|_{L^2}\leq C(\Pi_0(s)+\Pi_1(s)).
\]
Thus \eqref{eq:P-profiles-revised} gives the two time-integrated
bounds in \eqref{eq:global-ED-spacetime} on $[1,\infty)$.
Here the contribution of $\Pi_0$ to the gradient bound costs
$\nu^{1/2}\nu^{-1/6}=\nu^{1/3}\leq1$.
Taking $\lambda_*=\gamma$, the change from $s$ to $t$ costs at most
$\ee^{\lambda_*}$. The short-time $H^2$ estimate
\eqref{eq:H2-restart-smoothing-global} then proves
\eqref{eq:full-ED-final}--\eqref{eq:global-ED-spacetime} on the full
physical half-line.

The inviscid-damping component of $X_{2\gamma}$, applied to
$\Delta u^2_{\neq}$ in $E_4$, is particularly simple:
\[
 \nabla\Delta^{-1}\partial_x\Delta u^2_{\neq}
 =\partial_x\nabla u^2_{\neq}.
\]
The nonzero streamwise frequency also gives
$\|u^2_{\neq}\|_{H^1}\leq C\|\partial_x\nabla u^2_{\neq}\|_{L^2}$.
Therefore
\[
 \|\ee^{2\gamma\nu^{1/3}s}u^2_{\neq}(1+s)\|_{L_s^2H^1}
 +\|\ee^{2\gamma\nu^{1/3}s}
 \partial_x\nabla u^2_{\neq}(1+s)\|_{L_s^2L^2}
 \leq CE_4(\infty)\leq C\varepsilon.
\]
The same time translation and short-time $H^2$ bound prove
\eqref{eq:global-inviscid-damping}.

For the roll, all unweighted $Y^m$ components on $t\geq1$ are invariant under the time translation, and \eqref{eq:restart-mu-comparison} converts the restarted $\mu_\nu^{1/2}$ component to the physical weight with only an absolute loss. On $0\leq t\leq1$, the $L_t^\infty H^2$ and $\nu^{1/2}L_t^2H^3$ bounds in \eqref{eq:H2-restart-smoothing-global} control
		\begin{equation*}
			\|\Delta_\perp u^2_{0}\|_{Y^0(0,1)}
			+\|u^3_{0}\|_{Y^1(0,1)}
			+\|\mu_\nu^{1/2}\Delta_\perp u^3_{0}\|_{Y^0(0,1)}
			+\nu^{1/2}\|u^\perp_{0}\|_{L^2(0,1;L^2)}
			\leq C\varepsilon.
		\end{equation*}
		Together with the short-time Besov bound \eqref{eq:H2-restart-besov-global} and the restarted $E_{2}$ estimate, this proves \eqref{eq:full-roll-final}, and in particular
\eqref{eq:global-roll-endpoint}, on the full physical half-line.
		
		For the pointwise roll decay, set
$F(s):=(u^\perp_{\neq}\otimes u^\perp_{\neq})_{0}(1+s)$.
The wave-stress estimate \eqref{eq:wave-roll-forcing-revised} and
$E_4(\infty)\leq C\varepsilon$ imply
\begin{equation}
 \int_0^\infty(1+\nu s)^3\|F(s)\|_{L^2}^2\,ds
 \leq C\|\ee^{4\gamma\nu^{1/3}s}F\|_{L_s^2L^2}^2
 \leq C\nu^{-1}\varepsilon^4.
 \label{eq:cubic-weight-wave-stress}
\end{equation}
The first inequality follows from
$\sup_{s\geq0}(1+\nu s)^3\ee^{-8\gamma\nu^{1/3}s}\leq C_\gamma$,
since $0<\nu\leq1$. Apply \eqref{eq:E-decay} to the restarted roll.
Its initial $L^2$ norm and propagated Besov norm are bounded by
$C\varepsilon$, so \eqref{eq:cubic-weight-wave-stress} yields
\begin{align*}
 \|u^\perp_{0}(1+s)\|_{L^2}^2
 &\leq \frac{C\varepsilon^2}{(1+\nu s)^2}
 +\frac{C\varepsilon^4}{\nu^2(1+\nu s)^3}
 \leq\frac{C\varepsilon^2}{(1+\nu s)^2}.
\end{align*}
In the last step, $\varepsilon/\nu\leq c_*$ was used.
Since $1+\nu(t-1)\geq(1+\nu t)/2$ for $t\geq1$, this gives
\eqref{eq:global-roll-decay} on $[1,\infty)$. The short-time
$H^2$ estimate supplies the remaining interval.

Finally, time translation leaves the unweighted $H^4$, dissipative,
and $\partial_tu^1_{0}$ components of $E_1$ unchanged. Its global
bound in \eqref{eq:full-energy-final} proves
\eqref{eq:full-zero-final}.
\end{proof}

\begin{proof}[Proof of \Cref{cor:homogeneous-Hminus2}]
The embedding
$\dot H^{-2}=\dot B^{-2}_{2,2}\hookrightarrow\dot B^{-2}_{2,\infty}$
gives $\varepsilon\leq C\varepsilon_H$.
Choose $c_H$ so that $Cc_H\leq c_*$ and apply
\Cref{thm:full-couette-final,thm:global-decay-final}.
\end{proof}

\begin{proof}[Proof of \Cref{cor:weighted-initial-data}]
The initial roll belongs to $L^1(\R^2)$: its integral on the unit
ball is controlled by its $L^2$ norm, and its integral outside the
ball is controlled by the first moment. Since it is divergence free,
\Cref{prop:moment} gives
\[
 \|u^\perp_{0,\rm in}\|_{\dot B^{-2}_{2,\infty}}
 \leq C\||x_\perp|u^\perp_{0,\rm in}\|_{L^1}.
\]
Thus $\varepsilon\leq C\varepsilon_m$, and $Cc_m\leq c_*$ ensures
that both main theorems apply. For $\sigma>2$, Cauchy--Schwarz yields
\[
 \||x_\perp|u^\perp_{0,\rm in}\|_{L^1}
 \leq\bigl\||x_\perp|\langle x_\perp\rangle^{-\sigma}\bigr\|_{L^2}
 \|\langle x_\perp\rangle^\sigma u^\perp_{0,\rm in}\|_{L^2}
 \leq C_\sigma\varepsilon_\sigma.
\]
Choosing $c_\sigma$ accordingly proves the last assertion.
\end{proof}

\section{Endpoint control of the zero modes}
\label{sec:zero-endpoint}
In this section, the low-frequency condition in the main results is justified first.
Fourier splitting then controls the transverse velocity $u^\perp_{0}$,
and the two-level Duhamel formula controls the growth of the streak
$u^1_{0}$.
Together with the nonlinear estimates in
Section~\ref{sec:interactions-final}, these bounds yield the
zero-mode energy estimates in Subsection~\ref{sec:zero-closure-final}.

\subsection{Endpoint space and linear lift-up}
\label{sec:prelim}
\label{sec:sharpness}
The heat characterization and projected-divergence estimate explain the choice of $\dot B^{-2}_{2,\infty}(\mathbb R^2)$ in the energy framework.

\begin{Lem}
		\label{lem:heat-char}
		It holds for every $f\in L^2(\R^2)\cap\dot B^{-2}_{2,\infty}(\R^2)$ that
		\begin{equation}
			\norm{f}{\dot B^{-2}_{2,\infty}} \sim \sup_{\tau>0}\tau\norm{\ee^{\tau\DeltaP}f}{L^2}.
			\label{eq:heat-char}
		\end{equation}
		Consequently, for every $\nu>0$,
		\begin{equation}
			\sup_{t>0}\norm{t\ee^{\nu t\DeltaP}f}{L^2} \lesssim \nu^{-1}\norm{f}{\dot B^{-2}_{2,\infty}}.
			\label{eq:heat-liftup}
		\end{equation}
	\end{Lem}
	
	\begin{proof}
		Since the heat semigroup and $\dot\Delta_j$ commute and the heat multiplier is exponentially small on the support of $\varphi(2^{-j}\cdot)$, almost orthogonality gives
		\begin{equation*}
			\norm{\ee^{\tau\DeltaP}f}{L^2}^2\lesssim\sum_{j\in\mathbb Z}\ee^{-c\tau2^{2j}}\norm{\dot\Delta_jf}{L^2}^2\leq \norm{f}{\dot B^{-2}_{2,\infty}}^2\sum_{j\in\mathbb Z}2^{4j}\ee^{-c\tau2^{2j}}.
		\end{equation*}
		Multiplying by $\tau^2$, we obtain
		\begin{equation*}
			\tau^2\norm{\ee^{\tau\DeltaP}f}{L^2}^2\lesssim \norm{f}{\dot B^{-2}_{2,\infty}}^2\sum_{j\in\mathbb Z}(\tau2^{2j})^2\ee^{-c\tau2^{2j}}\lesssim \norm{f}{\dot B^{-2}_{2,\infty}}^2,
		\end{equation*}
		where the last sum is uniformly bounded because successive values of $\tau2^{2j}$ have ratio $4$. Next, fix $j\in\mathbb Z$ and evaluate the heat semigroup at time $2^{-2j}$. On the support of $\varphi(2^{-j}\xi)$ the factor $\ee^{-2^{-2j}|\xi|^2}$ is bounded below by a positive constant depending only on the fixed annulus. Thus, we have
		\begin{equation*}
			\norm{\dot\Delta_jf}{L^2}\lesssim \norm{\dot\Delta_j\ee^{2^{-2j}\DeltaP}f}{L^2}\lesssim \norm{\ee^{2^{-2j}\DeltaP}f}{L^2}
		\end{equation*}
		which implies
		\begin{equation*}
			2^{-2j}\norm{\dot\Delta_jf}{L^2}\lesssim \sup_{\tau>0}\tau\norm{\ee^{\tau\DeltaP}f}{L^2}.
		\end{equation*}
		Hence, taking the supremum over $j$ proves \eqref{eq:heat-char}. Finally, evaluating the supremum in \eqref{eq:heat-char} at $\nu t$ and dividing by $\nu$ yields \eqref{eq:heat-liftup}.
	\end{proof}
	
	\begin{Lem}
		\label{lem:div-B}
		For every matrix field $F\in L^1(\R^2;\R^{2\times2})$,
		\begin{equation}
			\norm{\PP\nablaP\cdot F}{\dot B^{-2}_{2,\infty}} \lesssim \norm{F}{L^1}.
			\label{eq:div-B}
		\end{equation}
		The same estimate holds for a vector field $G\in L^1(\R^2;\R^2)$ with $\nablaP\cdot G$ in place of $\PP\nablaP\cdot F$.
	\end{Lem}
	
	\begin{proof}
		The two-dimensional Leray projector is a bounded Fourier multiplier of order zero on each dyadic annulus. Therefore,
		\begin{equation*}
			\norm{\dot\Delta_j\PP\nablaP\cdot F}{L^2}\lesssim 2^j\norm{\dot\Delta_jF}{L^2}.
		\end{equation*}
		The $L^1$--$L^2$ Bernstein inequality in dimension two gives
		\begin{equation*}
			\norm{\dot\Delta_jF}{L^2}\lesssim 2^j\norm{F}{L^1}
		\end{equation*}
		which implies
		\begin{equation*}
			2^{-2j}\norm{\dot\Delta_j\PP\nablaP\cdot F}{L^2}\lesssim \norm{F}{L^1}
		\end{equation*}
		for every $j$, proving \eqref{eq:div-B}. The scalar-divergence case is identical.
	\end{proof}

The heat characterization in \Cref{lem:heat-char} gives the following linear statement.
	
	\begin{Prop}
		\label{prop:linear-criticality}
		For every $f\in L^2(\R^2)$ and $\nu>0$,
		\begin{equation*}
			\sup_{t>0}\norm{t\ee^{\nu t\DeltaP}f}{L^2}\simeq \nu^{-1}\norm{f}{\dot B^{-2}_{2,\infty}}.
		\end{equation*}
		% with the convention that either side may be infinite. Therefore $\dot B^{-2}_{2,\infty}$ is the endpoint $L^2$ low-frequency space for uniform-in-time linear lift-up control.
	\end{Prop}
	
	\begin{proof}
For finite Besov norm, the assertion follows from \eqref{eq:heat-char}
by setting $\tau=\nu t$. Conversely, for every $f\in L^2$ and $j\in\mathbb Z$,
the heat multiplier is bounded below on the $j$th annulus at time $2^{-2j}$, so
\begin{equation*}
 2^{-2j}\|\dot\Delta_jf\|_{L^2}
 \leq C2^{-2j}\|\ee^{2^{-2j}\DeltaP}f\|_{L^2}
 \leq C\sup_{\tau>0}\tau\|\ee^{\tau\DeltaP}f\|_{L^2}.
\end{equation*}
Taking the supremum over $j$ shows that finiteness of the heat-side
quantity implies finiteness of the Besov norm. This also proves the
assertion when either side is infinite.
\end{proof}

\begin{Rem}
    Since $\dot H^{-2}=\dot B^{-2}_{2,2}\hookrightarrow\dot B^{-2}_{2,\infty}$, an $\dot H^{-2}$ condition also controls the linear lift-up, but it is stronger than necessary. More importantly, the actual Leray-projected divergence forcing is naturally endpoint-Besov rather than endpoint-Sobolev. Let $\phi\in C_c^\infty(\mathbb R^2)$ have nonzero mean, set
	\begin{equation*}
		F=\begin{pmatrix}\phi&0\\0&0\end{pmatrix},
		\qquad
		g:=\mathbb P_\perp\operatorname{div}_\perp F,
	\end{equation*}
	and observe from \Cref{lem:div-B} that $g\in\dot B^{-2}_{2,\infty}$. On the other hand,
	\begin{equation*}
		\widehat g(\xi)
		=i\xi_1\left(I-\frac{\xi\otimes\xi}{|\xi|^2}\right)e_1\widehat\phi(\xi),
		\qquad
		|\widehat g(\xi)|^2
		=\frac{\xi_1^2\xi_2^2}{|\xi|^2}|\widehat\phi(\xi)|^2.
	\end{equation*}
	Because $\widehat\phi(0)\neq0$, there are $r_{0},c_{0}>0$ such that $|\widehat\phi(\xi)|\geq c_{0}$ for $|\xi|\leq r_{0}$. On any fixed angular sector on which $|\xi_1|,|\xi_2|\geq c|\xi|$, one therefore has
	\begin{equation*}
		\int_{|\xi|\leq r_{0}}|\xi|^{-4}|\widehat g(\xi)|^2\,d\xi
		\geq c\int_{0}^{r_{0}}\frac{dr}{r}=\infty.
	\end{equation*}
	Thus $g\notin\dot H^{-2}$ even though it has exactly the projected divergence structure of the nonlinear roll forcing. This explains why $\dot H^{-2}$ is not the correct nonlinear endpoint.
\end{Rem}

The Besov hypothesis has a simple sufficient formulation in terms of one spatial moment.
	
	\begin{Prop}
		\label{prop:moment}
		Let $u^\perp_{0}\in L^1(\R^2;\R^2)$ be divergence free and suppose $|x_\perp|u^\perp_{0}\in L^1(\R^2)$. Then
		\begin{equation}
			\int_{\R^2}u^\perp_{0}(x_\perp)\,dx_\perp=0
			\label{eq:R-mean-zero}
		\end{equation}
		and
		\begin{equation}
			\norm{u^\perp_{0}}{\dot B^{-2}_{2,\infty}}\leq C\norm{|x_\perp|u^\perp_{0}}{L^1}.
			\label{eq:moment-B}
		\end{equation}
	\end{Prop}
	
	\begin{proof}
		Since $u^\perp_{0}\in L^1$, $\widehat u^\perp_{0}$ is continuous. The divergence-free condition gives $\xi\cdot\widehat u^\perp_{0}(\xi)=0$. Taking $\xi=(r,0)$ with $r\neq0$ yields $\widehat u^2_{0}(r,0)=0$, and letting $r\to0$ gives $\widehat u^2_{0}(0)=0$. Similarly, taking $\xi=(0,r)$ gives $\widehat u^3_{0}(0)=0$. Hence \eqref{eq:R-mean-zero} holds.
		
		Using the zero mean and $|\ee^{-i x\cdot\xi}-1|\leq|x||\xi|$,
		\begin{equation*}
			|\widehat u^\perp_{0}(\xi)|=\abs{\int_{\R^2}(\ee^{-ix_\perp\cdot\xi}-1)u^\perp_{0}(x_\perp)\,dx_\perp}\leq |\xi|\norm{|x_\perp|u^\perp_{0}}{L^1}.
		\end{equation*}
		On the support of the $j$th dyadic block, $|\xi|\simeq2^j$ and the annulus has measure $O(2^{2j})$. Hence Plancherel gives
		\begin{equation*}
			\norm{\dot\Delta_ju^\perp_{0}}{L^2}\lesssim 2^{2j}\norm{|x_\perp|u^\perp_{0}}{L^1}.
		\end{equation*}
		Multiplying by $2^{-2j}$ and taking the supremum in $j$ proves \eqref{eq:moment-B}.
	\end{proof}
	
	The estimate \eqref{eq:moment-B} is the physical-space input to
\Cref{cor:weighted-initial-data}; it requires only the first moment
of the initial roll.

\subsection{Roll propagation and Fourier splitting}
\label{sec:roll}
We now use the endpoint estimates to control the nonlinear roll. Recall that
	\begin{equation}
		\pt u^\perp_{0}-\nu\DeltaP u^\perp_{0}=-\PP\nablaP\cdot(u^\perp_{0}\otimes u^\perp_{0}+(u^\perp_{\neq}\otimes u^\perp_{\neq})_{0}),\qquad \nablaP\cdot u^\perp_{0}=0,
		\label{eq:forced-roll}
	\end{equation}
	on $[0,T]\times\R^2$.
The next theorem propagates the endpoint norm and recovers time-integrated roll control under explicit bounds on the wave stress.
	
	\begin{Thm}
		\label{thm:roll-main}
		Fix $M\geq1$. There exist constants $c_{0}=c_{0}(M)>0$ and $C=C(M)>0$ such that the following holds. Let $0<\nu\leq1$ and let $u^\perp_{0}$ solve the forced roll equation \eqref{eq:forced-roll} on $[0,T]$. Assume
		\begin{equation*}
			\|u^\perp_{0,\rm in}\|_{L^2}+\|u^\perp_{0,\rm in}\|_{\dot B^{-2}_{2,\infty}}\leq\varepsilon,
		\end{equation*}
		\begin{align}
			\int_{0}^T\|\bigl(u^\perp_{\neq}(t)\otimes u^\perp_{\neq}(t)\bigr)_{0}\|_{L^1}\,dt&\leq M\nu^{-1/3}\varepsilon^2,
			\label{eq:stress-L1-main}\\
			\int_{0}^T(1+\nu t)\|\bigl(u^\perp_{\neq}(t)\otimes u^\perp_{\neq}(t)\bigr)_{0}\|_{L^2}^2\,dt&\leq M\nu^{-1}\varepsilon^4,
			\label{eq:stress-L2-main}
		\end{align}
		and $\varepsilon\leq c_{0}\nu$. Then
		\begin{equation}
			\sup_{0\leq t\leq T}\|u^\perp_{0}(t)\|_{\dot B^{-2}_{2,\infty}}\leq C\varepsilon,
			\label{eq:roll-B-main}
		\end{equation}
		\begin{equation}
			\int_{0}^T\|u^\perp_{0}(t)\|_{L^2}^2\,dt\leq C\nu^{-1}\varepsilon^2.
			\label{eq:roll-L2-main}
		\end{equation}
		The constants are independent of $T$ and $\nu$.
	\end{Thm}
	
	% To prove \Cref{thm:roll-main}, we combine Besov propagation with Fourier splitting.

	The Duhamel formula and \Cref{lem:div-B} immediately give the first basic estimate.
	
	\begin{Lem}
		\label{lem:B-prop}
		Every smooth solution of \eqref{eq:forced-roll} satisfies
		\begin{equation}
			\sup_{0\leq t\leq T}\norm{u^\perp_{0}(t)}{\dot B^{-2}_{2,\infty}}\leq \norm{u^\perp_{0,\rm in}}{\dot B^{-2}_{2,\infty}}+C\int_{0}^T\norm{u^\perp_{0}(t)}{L^2}^2\,dt+C\int_{0}^T\norm{\bigl(u^\perp_{\neq}(t)\otimes u^\perp_{\neq}(t)\bigr)_{0}}{L^1}\,dt.
			\label{eq:B-prop}
		\end{equation}
	\end{Lem}
	
	\begin{proof}
		The Duhamel formula for \eqref{eq:forced-roll} reads
		\begin{equation*}
			u^\perp_{0}(t)=\ee^{\nu t\DeltaP}u^\perp_{0,\rm in}-\int_{0}^t\ee^{\nu(t-s)\DeltaP}\PP\nablaP\cdot(u^\perp_{0}\otimes u^\perp_{0}+(u^\perp_{\neq}\otimes u^\perp_{\neq})_{0})(s)\,ds.
		\end{equation*}
		The heat semigroup is bounded on $\dot B^{-2}_{2,\infty}$. Hence \Cref{lem:div-B} and Minkowski give
		\begin{equation*}
			\norm{u^\perp_{0}(t)}{\dot B^{-2}_{2,\infty}}\leq\norm{u^\perp_{0,\rm in}}{\dot B^{-2}_{2,\infty}}+C\int_{0}^t\bigl(\norm{u^\perp_{0}\otimes u^\perp_{0}}{L^1}+\norm{(u^\perp_{\neq}\otimes u^\perp_{\neq})_{0}}{L^1}\bigr)(s)\,ds.
		\end{equation*}
		Since $\norm{u^\perp_{0}\otimes u^\perp_{0}}{L^1}\leq\norm{u^\perp_{0}}{L^2}^2$, taking the supremum in $t$ proves \eqref{eq:B-prop}.
	\end{proof}
	
The endpoint norm bounds the low-frequency Fourier mass. Splitting the energy at a time-dependent radius then yields the missing time-integrated $L^2$ estimate.
	
	\begin{Lem}
		\label{lem:lowmass}
		Let $f\in L^2(\R^2)\cap\dot B^{-2}_{2,\infty}(\R^2)$. Then for every $r>0$,
		\begin{equation*}
			\int_{|\xi|\leq r}|\widehat f(\xi)|^2\,d\xi\leq C r^4\norm{f}{\dot B^{-2}_{2,\infty}}^2.
		\end{equation*}
	\end{Lem}
	
	\begin{proof}
		Choose $J\in\mathbb Z$ so that $2^J\leq r<2^{J+1}$. By the finite overlap of the dyadic decomposition and Plancherel,
		\begin{equation*}
			\int_{|\xi|\leq r}|\widehat f(\xi)|^2\,d\xi\lesssim\sum_{j\leq J+C_{0}}\norm{\dot\Delta_jf}{L^2}^2.
		\end{equation*}
		Using \eqref{eq:Bdef},
		\begin{equation*}
			\sum_{j\leq J+C_{0}}\norm{\dot\Delta_jf}{L^2}^2\leq\norm{f}{\dot B^{-2}_{2,\infty}}^2\sum_{j\leq J+C_{0}}2^{4j}\lesssim 2^{4J}\norm{f}{\dot B^{-2}_{2,\infty}}^2\lesssim r^4\norm{f}{\dot B^{-2}_{2,\infty}}^2.
		\end{equation*}
	\end{proof}
	
	\begin{Lem}
		\label{lem:fourier-splitting}
		Let $u^\perp_{0}$ solve \eqref{eq:forced-roll}. Then
		\begin{equation}
			\int_{0}^T\norm{u^\perp_{0}(t)}{L^2}^2\,dt\leq \frac{C}{\nu}\bigl(\norm{u^\perp_{0,\rm in}}{L^2}^2+\left(\sup_{0\leq s\leq T}\norm{u^\perp_{0}(s)}{\dot B^{-2}_{2,\infty}}\right)^2\bigr)+\frac{C}{\nu^2}\int_{0}^T(1+\nu t)\norm{\bigl(u^\perp_{\neq}(t)\otimes u^\perp_{\neq}(t)\bigr)_{0}}{L^2}^2\,dt.
			\label{eq:fourier-splitting}
		\end{equation}
	\end{Lem}
	
	\begin{proof}
		Taking the $L^2$ inner product of \eqref{eq:forced-roll} with $u^\perp_{0}$ and using $\nablaP\cdot u^\perp_{0}=0$, the quadratic self-interaction cancels:
		\begin{equation*}
			\int_{\R^2}u^\perp_{0}\cdot\PP\nablaP\cdot(u^\perp_{0}\otimes u^\perp_{0})=\int_{\R^2}u^\perp_{0}\cdot\nablaP\cdot(u^\perp_{0}\otimes u^\perp_{0})=0.
		\end{equation*}
		Since $\PP$ is self-adjoint on divergence-free test fields,
		\begin{equation*}
			\frac12\frac{d}{dt}\norm{u^\perp_{0}(t)}{L^2}^2+\nu \norm{\nablaP u^\perp_{0}(t)}{L^2}^2=\int_{\R^2}(u^\perp_{\neq}\otimes u^\perp_{\neq})_{0}:\nablaP u^\perp_{0}\,dx_\perp.
		\end{equation*}
		Young's inequality therefore yields
		\begin{equation}
			\frac{d}{dt}\norm{u^\perp_{0}(t)}{L^2}^2+\nu \norm{\nablaP u^\perp_{0}(t)}{L^2}^2\leq C\nu^{-1}\norm{\bigl(u^\perp_{\neq}(t)\otimes u^\perp_{\neq}(t)\bigr)_{0}}{L^2}^2.
			\label{eq:E-basic}
		\end{equation}
		
		Fix $r>0$. By Plancherel and \Cref{lem:lowmass},
		\begin{align}
			\norm{\nablaP u^\perp_{0}(t)}{L^2}^2&=\int_{\R^2}|\xi|^2|\widehat u^\perp_{0}(t,\xi)|^2\,d\xi
			\geq r^2\int_{|\xi|>r}|\widehat u^\perp_{0}(t,\xi)|^2\,d\xi\nonumber\\
			&\geq r^2\norm{u^\perp_{0}(t)}{L^2}^2-Cr^6\left(\sup_{0\leq s\leq T}\norm{u^\perp_{0}(s)}{\dot B^{-2}_{2,\infty}}\right)^2.
			\label{eq:D-lower}
		\end{align}
		Taking $r^2=3/(1+\nu t)$ in \eqref{eq:D-lower} and combining with \eqref{eq:E-basic},
		\begin{equation*}
			\frac{d}{dt}\norm{u^\perp_{0}(t)}{L^2}^2+\frac{3\nu}{1+\nu t}\norm{u^\perp_{0}(t)}{L^2}^2\leq \frac{C\nu \left(\sup_{0\leq s\leq T}\norm{u^\perp_{0}(s)}{\dot B^{-2}_{2,\infty}}\right)^2}{(1+\nu t)^3}+C\nu^{-1}\norm{\bigl(u^\perp_{\neq}(t)\otimes u^\perp_{\neq}(t)\bigr)_{0}}{L^2}^2.
		\end{equation*}
		Multiplying by $(1+\nu t)^3$ gives
		\begin{equation*}
			\frac{d}{dt}\bigl((1+\nu t)^3\norm{u^\perp_{0}(t)}{L^2}^2\bigr)\leq C\nu \left(\sup_{0\leq s\leq T}\norm{u^\perp_{0}(s)}{\dot B^{-2}_{2,\infty}}\right)^2+C\nu^{-1}(1+\nu t)^3\norm{\bigl(u^\perp_{\neq}(t)\otimes u^\perp_{\neq}(t)\bigr)_{0}}{L^2}^2.
		\end{equation*}
		Hence
		\begin{equation}
\begin{aligned}
\norm{u^\perp_{0}(t)}{L^2}^2&\leq \frac{\norm{u^\perp_{0,\rm in}}{L^2}^2}{(1+\nu t)^3}+\frac{C\left(\sup_{0\leq s\leq T}\norm{u^\perp_{0}(s)}{\dot B^{-2}_{2,\infty}}\right)^2\nu t}{(1+\nu t)^3}\\&\quad+\frac{C}{\nu}(1+\nu t)^{-3}\int_{0}^t(1+\nu s)^3\norm{\bigl(u^\perp_{\neq}(s)\otimes u^\perp_{\neq}(s)\bigr)_{0}}{L^2}^2\,ds.
\end{aligned}
\label{eq:E-decay}
\end{equation}
		Integrating the first two terms of \eqref{eq:E-decay} over $[0,T]$ gives $C\nu^{-1}(\norm{u^\perp_{0,\rm in}}{L^2}^2+\left(\sup_{0\leq s\leq T}\norm{u^\perp_{0}(s)}{\dot B^{-2}_{2,\infty}}\right)^2)$. For the last term, Fubini yields
		\begin{align}
			&\frac{C}{\nu}\int_{0}^T(1+\nu t)^{-3}\int_{0}^t(1+\nu s)^3\norm{\bigl(u^\perp_{\neq}(s)\otimes u^\perp_{\neq}(s)\bigr)_{0}}{L^2}^2\,ds\,dt\nonumber\\
			&\qquad=\frac{C}{\nu}\int_{0}^T(1+\nu s)^3\norm{\bigl(u^\perp_{\neq}(s)\otimes u^\perp_{\neq}(s)\bigr)_{0}}{L^2}^2\int_s^T(1+\nu t)^{-3}\,dt\,ds\nonumber\\
			&\qquad\leq\frac{C}{\nu^2}\int_{0}^T(1+\nu s)\norm{\bigl(u^\perp_{\neq}(s)\otimes u^\perp_{\neq}(s)\bigr)_{0}}{L^2}^2\,ds.\notag
		\end{align}
		This proves \eqref{eq:fourier-splitting}.
	\end{proof}
	
    We then prove \Cref{thm:roll-main}.
	
	\begin{proof}[Proof of \Cref{thm:roll-main}]

		By \Cref{lem:B-prop} and \eqref{eq:stress-L1-main},
		\begin{equation}
			\sup_{0\leq s\leq T}\norm{u^\perp_{0}(s)}{\dot B^{-2}_{2,\infty}}\leq\varepsilon+C\int_{0}^T\|u^\perp_{0}(t)\|_{L^2}^2\,dt+M\nu^{-1/3}\varepsilon^2.
			\label{eq:BT-from-Y}
		\end{equation}
		Combining \Cref{lem:fourier-splitting} with \eqref{eq:stress-L2-main} gives
		\begin{equation}
			\int_{0}^T\|u^\perp_{0}(t)\|_{L^2}^2\,dt\leq\frac{C}{\nu}(\varepsilon^2+\left(\sup_{0\leq s\leq T}\norm{u^\perp_{0}(s)}{\dot B^{-2}_{2,\infty}}\right)^2)+CM\nu^{-3}\varepsilon^4.
			\label{eq:Y-from-B}
		\end{equation}
		Let $K>1$ and suppose $\int_{0}^T\|u^\perp_{0}(t)\|_{L^2}^2\,dt\leq K\nu^{-1}\varepsilon^2$.  Since $\varepsilon\leq\nu$ implies
$\varepsilon\nu^{-1/3}\leq\varepsilon^{2/3}$,
\eqref{eq:BT-from-Y} yields
		\begin{equation*}
			\sup_{0\leq s\leq T}\norm{u^\perp_{0}(s)}{\dot B^{-2}_{2,\infty}}\leq\left(1+CK\frac{\varepsilon}{\nu}+M\varepsilon^{2/3}\right)\varepsilon.
		\end{equation*}
		After fixing $K$ and taking $c_{0}$ sufficiently small, $\sup_{0\leq s\leq T}\norm{u^\perp_{0}(s)}{\dot B^{-2}_{2,\infty}}\leq2\varepsilon$.  Substituting in \eqref{eq:Y-from-B},
		\begin{equation*}
			\int_{0}^T\|u^\perp_{0}(t)\|_{L^2}^2\,dt\leq\left(C_1+CM\frac{\varepsilon^{2}}{\nu^{2}}\right)\nu^{-1}\varepsilon^2.
		\end{equation*}
		Take $K=2C_1+2$ and then reduce $c_{0}$ so that the second term is at most $\nu^{-1}\varepsilon^2$.  The resulting bound strictly improves the bootstrap.  Continuity gives
		$\int_{0}^T\|u^\perp_{0}(t)\|_{L^2}^2\,dt\leq C\nu^{-1}\varepsilon^2$ for every $T$, which is \eqref{eq:roll-L2-main}.  Inserting this estimate in \eqref{eq:BT-from-Y} proves \eqref{eq:roll-B-main}.
	\end{proof}

\subsection{Roll-to-streak superposition}
\label{sec:streak}

The additional time integration from the roll to the streak produces a heat kernel with a time factor. We first record its endpoint bounds and then derive the exact Duhamel identity.
	
	\begin{Lem}
		\label{lem:initial-liftup}
		Let $N\geq2$. For every $f\in H^{N-2}(\R^2)\cap\dot B^{-2}_{2,\infty}(\R^2)$,
		\begin{equation}
			\sup_{t>0}\norm{t\ee^{\nu t\DeltaP}f}{H^N}\leq C_N\nu^{-1}\left(\norm{f}{\dot B^{-2}_{2,\infty}}+\norm{f}{H^{N-2}}\right).
			\label{eq:initial-liftup}
		\end{equation}
	\end{Lem}
	
	\begin{proof}
		By \Cref{lem:heat-char}, the $L^2$ component is bounded by $C\nu^{-1}\norm{f}{\dot B^{-2}_{2,\infty}}$. For the top homogeneous derivative, Plancherel and the elementary bound $s\ee^{-s}\leq C$ give
		\begin{equation*}
			\norm{|D_\perp|^N t\ee^{\nu t\DeltaP}f}{L^2}=\nu^{-1}\norm{(\nu t|D_\perp|^2)\ee^{\nu t\DeltaP}|D_\perp|^{N-2}f}{L^2}\leq C\nu^{-1}\norm{f}{\dot H^{N-2}}.
		\end{equation*}
		Since $\norm{h}{H^N}\lesssim\norm{h}{L^2}+\norm{|D_\perp|^Nh}{L^2}$, \eqref{eq:initial-liftup} follows.
	\end{proof}
	
	\begin{Lem}
		\label{lem:superposition}
		Let $N\in\mathbb N_0$ and define
		\begin{equation*}
			\mathcal S_\nu F(t):=\int_{0}^t(t-s)\ee^{\nu(t-s)\DeltaP}\PP\nablaP\cdot F(s)\,ds.
		\end{equation*}
		Then
		\begin{equation}
			\norm{\mathcal S_\nu F}{L^\infty(0,T;H^N)}\leq C_N\nu^{-1}\left(\norm{F}{L^1(0,T;L^1)}+\mathbf 1_{\{N\geq1\}}\norm{F}{L^1(0,T;H^{N-1})}\right).
			\label{eq:Sbound}
		\end{equation}
	\end{Lem}
	
	\begin{proof}
		Fix $s$ and interpret $\PP\nablaP\cdot F(s)$ in the sense of distributions. We do not invoke \Cref{lem:heat-char} here, because $F(s)\in L^1$ does not imply $\PP\nablaP\cdot F(s)\in L^2$. Instead, the Fourier multiplier of $\ee^{\nu\tau\DeltaP}\PP\nablaP\cdot$ has size at most $C|\xi|\ee^{-\nu\tau|\xi|^2}$. Since $|\widehat F(\xi)|\leq\|F(s)\|_{L^1}$, Plancherel gives, in two dimensions,
		\begin{align}
			\norm{\ee^{\nu\tau\DeltaP}\PP\nablaP\cdot F(s)}{L^2}^2
			&\leq C\|F(s)\|_{L^1}^2\int_{\R^2}|\xi|^2\ee^{-2\nu\tau|\xi|^2}\,d\xi
			\nonumber\\
			&\leq C(\nu\tau)^{-2}\|F(s)\|_{L^1}^2.
			\notag
		\end{align}
		Consequently,
		\begin{equation}
			\sup_{\tau>0}\norm{\tau\ee^{\nu\tau\DeltaP}\PP\nablaP\cdot F(s)}{L^2}\leq C\nu^{-1}\norm{F(s)}{L^1}.
			\label{eq:S-low}
		\end{equation}
		If $N\geq1$, the top derivative satisfies
		\begin{align}
			\norm{|D_\perp|^N\tau\ee^{\nu\tau\DeltaP}\PP\nablaP\cdot F(s)}{L^2}&\lesssim \norm{\tau|D_\perp|^{N+1}\ee^{\nu\tau\DeltaP}F(s)}{L^2}\nonumber\\
			&\leq C\nu^{-1}\norm{F(s)}{\dot H^{N-1}},
			\notag
		\end{align}
		again because $\nu\tau|\xi|^2\ee^{-\nu\tau|\xi|^2}$ is uniformly bounded.  For $N=0$ only \eqref{eq:S-low} is used.  Combining the estimates gives
		\begin{equation*}
			\sup_{\tau>0}\norm{\tau\ee^{\nu\tau\DeltaP}\PP\nablaP\cdot F(s)}{H^N}\leq C_N\nu^{-1}\left(\norm{F(s)}{L^1}+\mathbf 1_{\{N\geq1\}}\norm{F(s)}{H^{N-1}}\right).
		\end{equation*}
		Minkowski's inequality then implies
		\begin{align}
			\sup_{0\leq t\leq T}\norm{\mathcal S_\nu F(t)}{H^N}&\leq\int_{0}^T\sup_{\tau>0}\norm{\tau\ee^{\nu\tau\DeltaP}\PP\nablaP\cdot F(s)}{H^N}\,ds,\notag
		\end{align}
		which proves \eqref{eq:Sbound}.
	\end{proof}
	
	\begin{Lem}
		\label{lem:direct-Duhamel}
		Let $N\geq0$ and define
		\begin{equation*}
			\mathcal D_\nu G(t):=\int_{0}^t\ee^{\nu(t-s)\DeltaP}\nablaP\cdot G(s)\,ds.
		\end{equation*}
		Then
		\begin{equation*}
			\norm{\mathcal D_\nu G}{L^\infty(0,T;H^N)}\leq C\nu^{-1/2}\norm{G}{L^2(0,T;H^N)}.
		\end{equation*}
	\end{Lem}
	
	\begin{proof}
		Taking the Fourier transform in $x_\perp=(y,z)$ and using Cauchy--Schwarz in time,
		\begin{align}
			\abs{\widehat{\mathcal D_\nu G}(t,\xi)}&\leq\left(\int_{0}^t|\xi|^2\ee^{-2\nu(t-s)|\xi|^2}\,ds\right)^{1/2}\left(\int_{0}^t|\widehat G(s,\xi)|^2\,ds\right)^{1/2}\nonumber\\
			&\leq (2\nu)^{-1/2}\left(\int_{0}^T|\widehat G(s,\xi)|^2\,ds\right)^{1/2}.\notag
		\end{align}
		Multiply by $\langle\xi\rangle^N$, square, integrate in $\xi$, and use Plancherel. The resulting bound is uniform in $t$.
	\end{proof}

Substituting the roll equation into the streak equation makes the extra time integration explicit.
	
	The roll equation gives
	\begin{equation}
		u^\perp_{0}(t)=\ee^{\nu t\DeltaP}u^\perp_{0,\rm in}-\int_{0}^t\ee^{\nu(t-s)\DeltaP}\PP\nablaP\cdot \bigl(u^\perp_{0}\otimes u^\perp_{0}+(u^\perp_{\neq}\otimes u^\perp_{\neq})_{0}\bigr)(s)\,ds.
		\label{eq:R-Duhamel}
	\end{equation}
	The streak equation gives
	\begin{equation}
		u^1_{0}(t)=\ee^{\nu t\DeltaP}u^1_{0,\rm in}-\int_{0}^t\ee^{\nu(t-s)\DeltaP}u^2_{0}(s)\,ds-\int_{0}^t\ee^{\nu(t-s)\DeltaP}\nablaP\cdot \bigl(u^\perp_{0}u^1_{0}+(u^\perp_{\neq}u^1_{\neq})_{0}\bigr)(s)\,ds.
		\label{eq:a-Duhamel-first}
	\end{equation}
	Substituting the second component of \eqref{eq:R-Duhamel} into \eqref{eq:a-Duhamel-first} and using the semigroup property gives the exact identity below.
	
	\begin{Lem}
		\label{lem:superposition-identity}
		For every $t\in[0,T]$,
		\begin{align}
			u^1_{0}(t)={}&\ee^{\nu t\DeltaP}u^1_{0,\rm in}-t\ee^{\nu t\DeltaP}u^2_{0,\rm in}\nonumber\\
			&+\int_{0}^t(t-s)\ee^{\nu(t-s)\DeltaP}\left[\PP\nablaP\cdot \bigl(u^\perp_{0}\otimes u^\perp_{0}+(u^\perp_{\neq}\otimes u^\perp_{\neq})_{0}\bigr)(s)\right]_y\,ds\nonumber\\
			&-\int_{0}^t\ee^{\nu(t-s)\DeltaP}\nablaP\cdot \bigl(u^\perp_{0}u^1_{0}+(u^\perp_{\neq}u^1_{\neq})_{0}\bigr)(s)\,ds,
			\label{eq:superposition-identity}
		\end{align}
		where $[\cdot]_y$ denotes the $y$ component of a transverse vector field, corresponding to the wall-normal roll component $u^2_{0}$.
	\end{Lem}
	
	\begin{proof}
		From \eqref{eq:R-Duhamel},
		\begin{equation*}
			u^2_{0}(s)=\ee^{\nu s\DeltaP}u^2_{0,\rm in}-\int_{0}^s\ee^{\nu(s-\sigma)\DeltaP}\left[\PP\nablaP\cdot \bigl(u^\perp_{0}\otimes u^\perp_{0}+(u^\perp_{\neq}\otimes u^\perp_{\neq})_{0}\bigr)(\sigma)\right]_y\,d\sigma.
		\end{equation*}
		Hence
		\begin{equation*}
			\int_{0}^t\ee^{\nu(t-s)\DeltaP}\ee^{\nu s\DeltaP}u^2_{0,\rm in}\,ds=t\ee^{\nu t\DeltaP}u^2_{0,\rm in}.
		\end{equation*}
		For the double integral, Fubini and the semigroup property give
		\begin{align}
			&\int_{0}^t\ee^{\nu(t-s)\DeltaP}\int_{0}^s\ee^{\nu(s-\sigma)\DeltaP}\left[\PP\nablaP\cdot \bigl(u^\perp_{0}\otimes u^\perp_{0}+(u^\perp_{\neq}\otimes u^\perp_{\neq})_{0}\bigr)(\sigma)\right]_y\,d\sigma\,ds\nonumber\\
			&\qquad=\int_{0}^t\int_\sigma^t\ee^{\nu(t-\sigma)\DeltaP}\left[\PP\nablaP\cdot \bigl(u^\perp_{0}\otimes u^\perp_{0}+(u^\perp_{\neq}\otimes u^\perp_{\neq})_{0}\bigr)(\sigma)\right]_y\,ds\,d\sigma\nonumber\\
			&\qquad=\int_{0}^t(t-\sigma)\ee^{\nu(t-\sigma)\DeltaP}\left[\PP\nablaP\cdot \bigl(u^\perp_{0}\otimes u^\perp_{0}+(u^\perp_{\neq}\otimes u^\perp_{\neq})_{0}\bigr)(\sigma)\right]_y\,d\sigma,\notag
		\end{align}
		Substitution into \eqref{eq:a-Duhamel-first} proves \eqref{eq:superposition-identity}.
	\end{proof}

The estimates above will be applied to the streak energy in Subsection~\ref{sec:zero-closure-final}.

\section{Nonlinear interaction estimates}
\label{sec:interactions-final}
In this section, the nonlinear terms are estimated using the energy functionals defined
in Section~\ref{sec:full-system}.
Velocity bounds are first derived to control products.
The interactions between zero and nonzero modes, the interactions
among nonzero modes, and the corresponding pressure terms
are then estimated.

\subsection{Wave estimates}
\begin{Lem}
		\label{lem:transient-streak-revised}
		For every $0\leq t\leq T$,
		\begin{equation}
			\|u^1_{0}(t)\|_{H^2}
			\leq C\bigl(E_{1}(T)+\frac{\varepsilon}{\nu}\bigr)m_\nu(t).
			\label{eq:transient-streak-bound-revised}
		\end{equation}
		Moreover, for every $c>0$,
		\begin{equation}
			\sup_{t\geq0}m_\nu(t)^2\ee^{-c\nu^{1/3}t}
			\leq C_c\nu^{4/3}.
			\label{eq:mnu-exponential-revised}
		\end{equation}
	\end{Lem}
	
	\begin{proof}
Integrate the time derivative and retain the restarted initial size:
\begin{equation*}
 \|u^1_{0}(t)\|_{H^2}
 \leq\min\{E_1(T),\,C_{\rm r}\varepsilon+\nu tE_1(T)\}
 \leq C\bigl(E_1(T)+\varepsilon/\nu\bigr)m_\nu(t).
\end{equation*}
The extra $\varepsilon/\nu$ accounts for the initial value; no
spanwise Poincar\'e estimate enters this argument.
For the weight, $m_\nu(t)\leq\nu(1+t)$ and $s=\nu^{1/3}t$ give
\begin{equation*}
 m_\nu(t)^2\ee^{-c\nu^{1/3}t}
 \leq\nu^{4/3}(\nu^{1/3}+s)^2\ee^{-cs}
 \leq C_c\nu^{4/3}.
\end{equation*}
\end{proof}

	\begin{Lem}
		\label{lem:transient-kappa-revised}
		Assume $E_{1}(T)\leq\delta_{0}$ and recall $\kappa=V_z/V_y$. Then, for every $0\leq t\leq T$,
		\begin{align}
			\|\kappa(t)\|_{H^1}&\leq C\bigl(E_{1}(T)+\frac{\varepsilon}{\nu}\bigr)m_\nu(t),
			\notag\\
			\|\kappa(t)\|_{H^2}&\leq C\bigl(E_{1}(T)+\frac{\varepsilon}{\nu}\bigr)m_\nu(t)^{1/2},
			\label{eq:kappa-H2-transient-revised}\\
			\|\kappa(t)\|_{H^3}&\leq C E_{1}(T).
			\notag
		\end{align}
		Moreover, for every $c>0$,
		\begin{equation}
			\sup_{t\geq0}m_\nu(t)\ee^{-c\nu^{1/3}t}\leq C_c\nu^{2/3}.
			\label{eq:mnu-first-power-revised}
		\end{equation}
		Consequently, if $h$ is $x$-nonzero, then
		\begin{align}
			&\nu^{-1/2}\|\ee^{3\gamma\nu^{1/3}t}\nabla(\kappa h)\|_{L_t^2L_x^2}
			\nonumber\\
			&\qquad\leq C\bigl(E_{1}(T)+\frac{\varepsilon}{\nu}\bigr)
			\nu^{-1/6}\|\ee^{4\gamma\nu^{1/3}t}\nabla h\|_{L_t^2L_x^2},
			\label{eq:kappa-multiplier-gain-revised}
		\end{align}
		and, one exponential level lower,
		\begin{align}
			&\nu^{-1/2}\|\ee^{2\gamma\nu^{1/3}t}\nabla(\kappa h)\|_{L_t^2L_x^2}
			+\nu^{-1/2}\|\ee^{2\gamma\nu^{1/3}t}\kappa\nabla h\|_{L_t^2L_x^2}
			\nonumber\\
			&\qquad\leq C\bigl(E_{1}(T)+\frac{\varepsilon}{\nu}\bigr)
			\nu^{-1/6}\|\ee^{3\gamma\nu^{1/3}t}\nabla h\|_{L_t^2L_x^2}.
			\label{eq:kappa-multiplier-gain-primary}
		\end{align}
		Here and below the zero-order term in a product with $\nabla\kappa$ is absorbed by the $x$-Poincar\'e inequality.
	\end{Lem}
	
	\begin{proof}
By Sobolev composition and $V_y\geq1/2$, on $\R^2$ we have
\begin{equation*}
 \|\kappa\|_{H^j}\leq C\|u^1_{0}\|_{H^{j+1}},\quad j=1,2,3,
 \qquad
 \|u^1_{0}\|_{H^3}^2\leq
 \|u^1_{0}\|_{H^2}\|u^1_{0}\|_{H^4}.
\end{equation*}
Together with \Cref{lem:transient-streak-revised}, these prove
\eqref{eq:kappa-H2-transient-revised}. The one-power version of its
weight calculation proves \eqref{eq:mnu-first-power-revised}.

For the forcing multipliers, \eqref{eq:anis-prod-L2},
$H^2(\R^2)\hookrightarrow L^\infty$, and $h_{0}=0$ give
\begin{equation*}
 \|\nabla(\kappa h)\|_{L^2}+\|\kappa\nabla h\|_{L^2}
 \leq C\|\kappa\|_{H^2}\|\nabla h\|_{L^2}.
\end{equation*}
The coefficient gain is
\begin{equation*}
 \sup_{t\geq0}m_\nu(t)^{1/2}\ee^{-\gamma\nu^{1/3}t}
 \leq C_\gamma\nu^{1/3}.
\end{equation*}
Applying this between weights $3\gamma,4\gamma$ and between
$2\gamma,3\gamma$ gives the two stated multiplier estimates,
with $\nu^{-1/2}\nu^{1/3}=\nu^{-1/6}$.
\end{proof}

The pointwise profiles defined in Section~\ref{sec:full-system} connect the energy bounds to the velocity factors in nonlinear products. We first establish this recovery and then estimate the quadratic wave terms.

	\begin{Lem}
		\label{lem:E6-wave-recovery}
		Assume \eqref{eq:variable-small-assumption}. Then,
		\begin{align}
			\|u^\perp_{\neq}\|_{L^2}+\|D_z^Vu_{\neq}^\perp\|_{L^2}+\|q\|_{H^1}+\|D_z^Vq\|_{L^2}
			&\leq C\mathcal U_{0},
			\label{eq:E6-recovery-low}\\
			\|\nabla u^\perp_{\neq}\|_{L^2}+\|\nabla D_z^Vu_{\neq}^\perp\|_{L^2}+\|q\|_{H^2}+\|D_z^Vq\|_{H^1}
			&\leq C\mathcal U_1,
			\label{eq:E6-recovery-high}\\
			\|\nabla^2u^3_{\neq}\|_{L^2}
			&\leq C\bigl(\mathcal U_1+\nu^{-2/3}\mathcal U_{0}\bigr).
			\label{eq:E6-recovery-W3-second}
		\end{align}
		Moreover,
		\begin{equation}
			\|D_z^V\partial_mu^j_{\neq}\|_{L^2}
			\leq C\mathcal U_1,
			\qquad m\in\{y,z\},\quad j\in\{2,3\}.
			\label{eq:E6-D-commuted}
		\end{equation}
	\end{Lem}
	
	\begin{proof}
		Since $(u_{\neq})_{0}=0$, the Poincar\'e inequality gives, for $j=2,3$,
		\begin{equation*}
			\|u^j_{\neq}\|_{L^2}\leq\|\partial_x^2u^j_{\neq}\|_{L^2},
			\qquad
			\|D_z^Vu^j_{\neq}\|_{L^2}\leq\|\partial_xD_z^Vu^j_{\neq}\|_{L^2},
		\end{equation*}
		and the corresponding inequalities after applying one ordinary derivative. Likewise,
		\begin{equation*}
			\|q\|_{H^1}\leq C\|\partial_x\nabla q\|_{L^2},
			\qquad
			\|q\|_{H^2}\leq C\|\nabla\partial_x\nabla q\|_{L^2}.
		\end{equation*}
		These quantities occur in \eqref{eq:U0-revised}--\eqref{eq:U1-revised}.
		
		Noting $q=u^2_{\neq}+\kappa u^3_{\neq}$, we have
		\begin{equation*}
			D_z^Vq=D_z^Vu^2_{\neq}+(D_z^V\kappa)u^3_{\neq}+\kappa D_z^Vu^3_{\neq}.
		\end{equation*}
		The coefficient bounds $\|\kappa\|_{H^3}\leq C\delta_{0}$ and the two-dimensional Sobolev product estimates therefore imply
		\begin{equation*}
			\|D_z^Vq\|_{L^2}\leq C\mathcal U_{0},
			\qquad
			\|D_z^Vq\|_{H^1}\leq C\mathcal U_1.
		\end{equation*}
		This proves \eqref{eq:E6-recovery-low}--\eqref{eq:E6-recovery-high}.
		
		For \eqref{eq:E6-recovery-W3-second}, every second derivative containing $\partial_x$ is controlled by $\mathcal U_1$ through the streamwise gap. Write $\partial_z=D_z^V+\kappa\partial_y$. A second derivative containing $\partial_z$ is reduced to $\nabla D_z^Vu^3_{\neq}$ and terms containing $\kappa$ or $\nabla\kappa$, except for the pure derivative $\partial_y^2u^3_{\neq}$. The latter satisfies
		\begin{equation*}
			\|\partial_y^2u^3_{\neq}\|_{L^2}\leq\|\Delta u^3_{\neq}\|_{L^2}\leq\nu^{-2/3}\mathcal U_{0}.
		\end{equation*}
		Finally,
		\begin{equation*}
			[D_z^V,\partial_m]f=(\partial_m\kappa)\partial_yf,
			\qquad m\in\{y,z\},
		\end{equation*}
		and $\|\partial_m\kappa\|_{L^\infty}\leq C\delta_{0}$, while $\|\partial_yu^j_{\neq}\|_{L^2}+\|\partial_mD_z^Vu^j_{\neq}\|_{L^2}\leq C\mathcal U_1$. This proves \eqref{eq:E6-D-commuted} and completes the proof.
	\end{proof}

	The general streamwise--spanwise bilinear, mean-zero high-order, and tame zero--wave product estimates used in this subsection are collected in \Cref{lem:streamwise-spanwise-bilinear,lem:mean-zero-high-product,lem:tame-zero-wave-high}; see Appendix~\ref{app:anisotropic-tools}.
	
	\begin{Lem}
		\label{lem:velocity-reconstruction-final}
		For every integer $s\geq0$ and every divergence-free $x$-nonzero field $u_{\neq}$,
		\begin{equation}
			\|\nabla^s(\partial_x,\partial_z)\partial_xu_{\neq}\|_{L^2}
			\leq C_s\bigl(\|\nabla^s\Delta u^2_{\neq}\|_{L^2}+\|\nabla^s(\partial_x^2+\partial_z^2)u^3_{\neq}\|_{L^2}\bigr).
			\label{eq:velocity-reconstruction-final}
		\end{equation}
		In particular,
		\begin{align}
			\|u_{\neq}\|_{L^2}&\leq C\Pi_{0},
			&\|\nabla u_{\neq}\|_{L^2}&\leq C(\Pi_{0}+\Pi_1),
			\label{eq:velocity-low-reconstruction}\\
			\|u^2_{\neq}\|_{H^2}&\leq C\Pi_{0},
			&\|u^2_{\neq}\|_{H^3}&\leq C(\Pi_{0}+\Pi_1),
			\label{eq:W2-reconstruction}\\
			\|(\partial_x,\partial_z)u_{\neq}\|_{H^1}&\leq C(\Pi_{0}+\Pi_1).
			\label{eq:good-W-reconstruction}
		\end{align}
		The genuinely bad derivatives of the third component are recorded separately:
		\begin{equation}
			\|\partial_y^2u^3_{\neq}\|_{L^2}\leq C\nu^{-2/3}\mathcal U_{0},
			\qquad
			\|\nabla\partial_y^2u^3_{\neq}\|_{L^2}\leq C\nu^{-2/3}\mathcal U_1.
			\label{eq:bad-W3-profile-reconstruction}
		\end{equation}
		Moreover, every occurrence of $\partial_xq$, $\partial_xD_z^Vu^j_{\neq}$, and one additional derivative thereof is bounded, respectively, by $\mathcal U_{0}$ and $\mathcal U_1$.
	\end{Lem}
	
	\begin{proof}
The recovery estimate identified above uses only $\partial_xu^1_{\neq}=-\partial_yu^2_{\neq}
-\partial_zu^3_{\neq}$ and pointwise Fourier inequalities, which
hold for real $\ell$. Since $|k|\geq1$ on the $x$-nonzero subspace, the $x$-Poincar\'e inequality can be applied repeatedly. The cases $s=0,1$ of \eqref{eq:velocity-reconstruction-final}, together with $x$-Poincar\'e, give \eqref{eq:velocity-low-reconstruction}; the case $s=1$ gives precisely the $H^1$ good-derivative estimate \eqref{eq:good-W-reconstruction}. The estimate \eqref{eq:W2-reconstruction} follows from the elliptic bound $\|u^2_{\neq}\|_{H^{m+2}}\leq C(\|\Delta u^2_{\neq}\|_{H^m}+\|u^2_{\neq}\|_{L^2})$. No $H^2$ estimate for $(\partial_x,\partial_z)u_{\neq}$ is claimed at the level of $\Pi_{0}+\Pi_1$; such a bound would require one additional differentiated profile.
		
		Finally,
		\begin{equation*}
			\|\partial_y^2u^3_{\neq}\|_{L^2}\leq\|\Delta u^3_{\neq}\|_{L^2},
			\qquad
			\|\nabla\partial_y^2u^3_{\neq}\|_{L^2}\leq\|\nabla\Delta u^3_{\neq}\|_{L^2},
		\end{equation*}
		and the last terms in \eqref{eq:U0-revised}--\eqref{eq:U1-revised} prove \eqref{eq:bad-W3-profile-reconstruction}. The assertions involving $q$ and $D_z^V$ are direct consequences of the definitions of $\mathcal U_{0}$ and $\mathcal U_1$, together with the $x$-Poincar\'e inequality. Notice that no unweighted estimate of the full $H^2$ norm of $u^3_{\neq}$ is asserted; its pure $y$ derivatives are exactly the weak quantities in \eqref{eq:bad-W3-profile-reconstruction}.
	\end{proof}
	
	\begin{Lem}
		\label{lem:weak-wave3-gradient-revised}
		Define
		\begin{equation*}
			\mathcal D_{\neq}(t):=\|\Delta\partial_xu_{\neq}(t)\|_{L^2}+\|\Delta\partial_zu^3_{\neq}(t)\|_{L^2}.
		\end{equation*}
		Then,
		\begin{align}
			\mathcal D_{\neq}(t)
			&\leq C\left(\nu^{-1/3}\Pi_1(t)+\nu^{1/3}\|\nabla\Delta u^3_{\neq}(t)\|_{L^2}\right),
			\label{eq:DW-interpolation-revised}\\
			\|\nabla(u_{\neq}\cdot\nabla u^3_{\neq})(t)\|_{L^2}
			&\leq C\Pi_{0}(t)\left(\mathcal D_{\neq}(t)+\|\Delta u^3_{\neq}(t)\|_{L^2}\right).
			\label{eq:weak-wave3-point-revised}
		\end{align}
		Consequently,
		\begin{equation}
			\nu^{1/3}\|\ee^{4\gamma\nu^{1/3}t}\nabla(u_{\neq}\cdot\nabla u^3_{\neq})\|_{L_t^2L_x^2}
			\leq C\nu^{-1/2}E_{4}(T)^2.
			\label{eq:weak-wave3-gradient-aux}
		\end{equation}
	\end{Lem}
	
	\begin{proof}
The Fourier interpolation identified above proves
\eqref{eq:DW-interpolation-revised} for real $\ell$.
The differentiated transport argument gives
\begin{align*}
 \sum_{j=1,3}\|u^j_{\neq}\partial_ju^3_{\neq}\|_{H^1}
 &\leq C\Pi_0\mathcal D_{\neq},\\
 \|u^2_{\neq}\partial_yu^3_{\neq}\|_{H^1}
 &\leq C\Pi_0\|\Delta u^3_{\neq}\|_{L^2}.
\end{align*}
These use \Cref{lem:mean-zero-high-product,lem:velocity-reconstruction-final},
the continuous-spanwise versions of the product and recovery estimates
used there. Their sum is \eqref{eq:weak-wave3-point-revised}.

It remains to check the time weights in the present energy.
Its weak Laplacian component gives
\begin{equation*}
 \nu^{5/6}\|\ee^{3\gamma\nu^{1/3}t}\Delta u^3_{\neq}\|_{L_t^2L_x^2}
 +\nu^{7/6}\|\ee^{3\gamma\nu^{1/3}t}\nabla\Delta u^3_{\neq}\|_{L_t^2L_x^2}
 \leq CE_4(T).
\end{equation*}
Consequently, \eqref{eq:P-profiles-revised} and the two pointwise
estimates imply
\begin{align*}
 &\nu^{1/3}\|\ee^{4\gamma\nu^{1/3}t}
 \nabla(u_{\neq}\cdot\nabla u^3_{\neq})\|_{L_t^2L_x^2}\\
 &\quad\leq C\|\ee^{2\gamma\nu^{1/3}t}\Pi_0\|_{L_t^\infty}
 \Bigl(\|\ee^{2\gamma\nu^{1/3}t}\Pi_1\|_{L_t^2}
 +\nu^{2/3}\|\ee^{3\gamma\nu^{1/3}t}\nabla\Delta u^3_{\neq}\|_{L_t^2L_x^2}
 +\nu^{1/3}\|\ee^{3\gamma\nu^{1/3}t}\Delta u^3_{\neq}\|_{L_t^2L_x^2}\Bigr)\\
 &\quad\leq C\nu^{-1/2}E_4(T)^2.
\end{align*}
\end{proof}
	
	\begin{Lem}
		\label{lem:quadratic-wave-revised}
        It holds that
		\begin{align}
			&\|u_{\neq}\otimes u_{\neq}\|_{L^2}+\|u_{\neq}\cdot\nabla u_{\neq}\|_{L^2}
			+\|\nabla(u_{\neq}\cdot\nabla u^2_{\neq})\|_{L^2}
			+\|(\partial_x,\partial_z)(u_{\neq}\cdot\nabla u^3_{\neq})\|_{L^2}
			\nonumber\\
			&\qquad+\|\nabla^2p^{\neq\neq}\|_{L^2}+\|\nabla\partial_zp^{\neq\neq}\|_{L^2}
			\leq C\Pi_{0}(t)\Pi_1(t).
			\label{eq:quadratic-M-final}
		\end{align}
		Consequently,
		\begin{align}
			&\|\ee^{4\gamma\nu^{1/3}t}(u_{\neq}\otimes u_{\neq},u_{\neq}\cdot\nabla u_{\neq},
			\nabla(u_{\neq}\cdot\nabla u^2_{\neq}),(\partial_x,\partial_z)(u_{\neq}\cdot\nabla u^3_{\neq}),
			\nabla^2p^{\neq\neq},\nabla\partial_zp^{\neq\neq})\|_{L_t^2L_x^2}
			\nonumber\\
			&\qquad\leq C\nu^{-1/2}E_{4}(T)^2,
			\label{eq:quadratic-time-final}
		\end{align}
		and the only derivative not contained in \eqref{eq:quadratic-time-final} satisfies
		\begin{equation}
			\nu^{1/3}\|\ee^{4\gamma\nu^{1/3}t}\nabla(u_{\neq}\cdot\nabla u^3_{\neq})\|_{L_t^2L_x^2}
			\leq C\nu^{-1/2}E_{4}(T)^2.
			\label{eq:bad-wave3-weighted-revised}
		\end{equation}
		The same estimates hold for both the zero and nonzero parts of the product.
	\end{Lem}
	
	\begin{proof}
In the pointwise transport bounds identified above, take
$A_0=\Pi_{0}$ and $A_1=\Pi_1$.
The required recovery and bilinear estimates on the present domain are
\Cref{lem:velocity-reconstruction-final,lem:streamwise-spanwise-bilinear}.
The pressure corresponds to the streamwise nonzero part of $p^{(4)}$
in the cited source. Its Hessian is controlled by its Laplacian:
\begin{equation*}
 \Delta p^{\neq\neq}
 =-\bigl(\partial_j(u_{\neq}\cdot\nabla u^j_{\neq})\bigr)_{\neq},
 \qquad
 \|\nabla^2p^{\neq\neq}\|_{L^2}
 +\|\nabla\partial_zp^{\neq\neq}\|_{L^2}
 \leq C\|\Delta p^{\neq\neq}\|_{L^2}.
\end{equation*}
The transport bound also includes the $j=1$ derivative, proving \eqref{eq:quadratic-M-final}.
Using \eqref{eq:P-profiles-revised},
\begin{equation*}
 \|\ee^{4\gamma\nu^{1/3}t}\Pi_{0}\Pi_1\|_{L_t^2}
 \leq\|\ee^{2\gamma\nu^{1/3}t}\Pi_{0}\|_{L_t^\infty}
 \|\ee^{2\gamma\nu^{1/3}t}\Pi_1\|_{L_t^2}
 \leq C\nu^{-1/2}E_{4}(T)^2.
\end{equation*}
Estimate \eqref{eq:bad-wave3-weighted-revised} is
\eqref{eq:weak-wave3-gradient-aux}. Streamwise averaging and taking
the nonzero part are contractions in the $L^2$-based norms, which
proves the last assertion.
\end{proof}
	
	\begin{Lem}
		\label{lem:wave-Y2-final}
		One has
		\begin{equation}
			\|\ee^{2\gamma\nu^{1/3}t}(u_{\neq}\cdot\nabla u_{\neq})_{\neq}\|_{Y^2}
			+\|\ee^{2\gamma\nu^{1/3}t}(u_{\neq}\cdot\nabla u_{\neq})_{0}\|_{Y^2}
			\leq C E_{4}(T)E_{5}(T).
			\label{eq:wave-Y2-final}
		\end{equation}
	\end{Lem}
	
	\begin{proof}
In the $Y^2$ estimate identified above, the energies $E_3,E_4$
correspond to the present $E_{4},E_{5}$, and $Y_0^2$ corresponds to $Y^2$.
The weights $2,3$ become $2\gamma,3\gamma$.
The spatial argument uses
\Cref{lem:mean-zero-high-product,lem:velocity-reconstruction-final}.
Time integration places the $2\gamma$ weight on
the high-order factor and the unweighted energy and dissipation
bounds supplied by $E_{4}$ for the other factor.
It therefore gives $CE_{4}E_{5}$ with the same powers of $\nu$.
Both projections are bounded in $Y^2$, proving \eqref{eq:wave-Y2-final}.
\end{proof}

\begin{Lem}
		\label{lem:wave-roll-forcing-revised}
		One has
		\begin{align}
			&\|\ee^{4\gamma\nu^{1/3}t}((u^\perp_{\neq}\otimes u^\perp_{\neq})_{0},(u_{\neq}\cdot\nabla u^\perp_{\neq})_{0},\nabla_\perp (u_{\neq}\cdot\nabla u^2_{\neq})_{0},
			\partial_z(u_{\neq}\cdot\nabla u^3_{\neq})_{0})\|_{L_t^2L_x^2}
			\nonumber\\
			&\qquad
			+\nu^{1/3}\|\ee^{4\gamma\nu^{1/3}t}\nabla_\perp (u_{\neq}\cdot\nabla u^3_{\neq})_{0}\|_{L_t^2L_x^2}
			\leq C\nu^{-1/2}E_{4}(T)^2.
			\label{eq:wave-roll-forcing-revised}
		\end{align}
		Moreover,
		\begin{equation}
			\|\mu_\nu^{1/2}\partial_y(u_{\neq}\cdot\nabla u^2_{\neq})_{0}\|_{L_t^2L_x^2}
			\leq C\nu^{-1/6}E_{4}(T)^2,
			\label{eq:weighted-NW2-revised}
		\end{equation}
		and
		\begin{equation}
			\nu^{-1/2}\|\mu_\nu^{1/2}\partial_y(u_{\neq}\cdot\nabla u^3_{\neq})_{0}\|_{L_t^2L_x^2}
			\leq C\nu^{-1}E_{4}(T)^2.
			\label{eq:weighted-NW3-revised}
		\end{equation}
	\end{Lem}
	
	\begin{proof}
Streamwise averaging is an $L^2$ contraction and commutes with
transverse derivatives. Applying it to
\Cref{lem:quadratic-wave-revised} proves
\eqref{eq:wave-roll-forcing-revised}, including the weak third-component term.
For the two additional weights in this paper, observe that
\begin{equation}
 \sup_{t\geq0}\mu_\nu(t)\ee^{-8\gamma\nu^{1/3}t}
 \leq\nu^{2/3}\sup_{s\geq0}(1+s)\ee^{-8\gamma s}
 \leq C_\gamma\nu^{2/3}.
 \label{eq:mu-wave-gap-revised}
\end{equation}
Thus multiplication by $\mu_\nu^{1/2}$ costs at most
$C\nu^{1/3}$ relative to the $4\gamma$ weight.
For the second component this gives
$\nu^{1/3}\nu^{-1/2}E_4^2=\nu^{-1/6}E_4^2$.
For the third component it cancels the $\nu^{1/3}$ on the left of
\eqref{eq:bad-wave3-weighted-revised}; the additional factor
$\nu^{-1/2}$ in \eqref{eq:weighted-NW3-revised} then gives
$C\nu^{-1}E_4^2$.
\end{proof}
\subsection{Mixed interactions and nonlinear forcing}
We next estimate the mixed interactions and the nonlinear forcing
in the norms required by $E_{6}$ and $E_{4}$.

\begin{Lem}
		\label{lem:roll-coefficients-E6}
		Assume $E_{1}(T)+\frac{\varepsilon}{\nu}\leq\delta_*$. 
		Then, for every $t\in[0,T]$,
		\begin{align}
			&\|u^\perp_{0}(t)\|_{L^\infty}+\|u^\perp_{0}(t)\|_{H^1}
			+\|\bigl(u^2_{0}+\kappa u^3_{0}\bigr)(t)\|_{H^2}+\|D_z^Vu^\perp_{0}(t)\|_{H^1}
			\leq C E_{2}(T),
			\label{eq:roll-coeff-point-E6}\\
			&\|\nabla_\perp u^2_{0}(t)\|_{H^1}
			+\|\partial_zu^3_{0}(t)\|_{H^1}
			+\|\kappa(t)\partial_yu^3_{0}(t)\|_{H^1}
			\leq C E_{2}(T),
			\notag
		\end{align}
		and
		\begin{equation}
			\nu^{1/2}\|\bigl(\|\nabla_\perp u^\perp_{0}(t)\|_{H^1}+\|\nabla_\perp D_z^V u^\perp_{0}(t)\|_{H^1}\bigr)\|_{L_t^2}\leq C E_{2}(T).
			\label{eq:R1-time-complete}
		\end{equation}
	\end{Lem}
	
	\begin{proof}
		The pointwise roll recovery \eqref{eq:roll-recovery-final} and the anisotropic two-dimensional embedding
		\begin{equation*}
			\|f\|_{L^\infty(\mathbb R^2)}
			\leq C\bigl(\|f\|_{H^1(\mathbb R^2)}+\|\partial_zf\|_{H^1(\mathbb R^2)}\bigr)
		\end{equation*}
		give the bounds for $u^\perp_{0}$ itself. Since $0<\nu\leq1$, $m_\nu(t)\leq\mu_\nu(t)$, and \Cref{lem:transient-kappa-revised} gives
		\begin{align}
			\|\kappa\partial_yu^3_{0}\|_{H^1}
			&\leq C\|\kappa\|_{H^2}\|\partial_yu^3_{0}\|_{H^1}
			\leq C(E_{1}+\frac{\varepsilon}{\nu})m_\nu^{1/2}\mu_\nu^{-1/2}E_{2}
			\leq C\delta_* E_{2},
			\notag\\
			\|\kappa u^3_{0}\|_{H^2}
			&\leq C\|\kappa\|_{H^2}\|u^3_{0}\|_{H^2}
			\leq C(E_{1}+\frac{\varepsilon}{\nu})m_\nu^{1/2}\mu_\nu^{-1/2}E_{2}
			\leq C\delta_* E_{2},\notag
		\end{align}
		where in the last line we used \eqref{eq:kappa-H2-transient-revised}, \eqref{eq:roll-recovery-final}, and $m_\nu\leq\mu_\nu$.
		Consequently, $D_z^Vu^3_{0}=\partial_zu^3_{0}-\kappa\partial_yu^3_{0}$ and $u^2_{0}+\kappa u^3_{0}$ satisfy the asserted pointwise estimates. The bounds for $D_z^Vu^2_{0}$ are easier because $u^2_{0}\in H^2$ pointwise.
		
		For the spacetime estimate, use
		\begin{equation*}
			\nabla_\perp D_z^Vu^\perp_{0}
			=\nabla_\perp\partial_zu^\perp_{0}-(\nabla_\perp\kappa)\partial_yu^\perp_{0}
			-\kappa\nabla_\perp\partial_yu^\perp_{0}.
		\end{equation*}
		The first term is controlled by \eqref{eq:roll-recovery-time-final}. For the terms containing $u^2_{0}$, the unweighted estimate for $\nabla_\perp u^2_{0}$ in \eqref{eq:roll-recovery-time-final} and the uniform $H^3$ bound for $\kappa$ are sufficient. The only point requiring the weighted roll norm is the third component. By the two-dimensional Sobolev product estimate, \eqref{eq:kappa-H2-transient-revised}, and $m_\nu\leq\mu_\nu$,
		\begin{align}
			\nu^{1/2}\|(\nabla_\perp\kappa)\partial_yu^3_{0}\|_{L_t^2H^1}
			&\leq C(E_{1}+\frac{\varepsilon}{\nu})\nu^{1/2}\|m_\nu^{1/2}\partial_yu^3_{0}\|_{L_t^2H^2}
			\nonumber\\
			&\leq C(E_{1}+\frac{\varepsilon}{\nu})\nu^{1/2}\|\mu_\nu^{1/2}\partial_yu^3_{0}\|_{L_t^2H^2}
			\leq C\delta_* E_{2},
			\label{eq:kappagrad-roll-time-E6}
			\\
			\nu^{1/2}\|\kappa\nabla_\perp\partial_yu^3_{0}\|_{L_t^2H^1}
			&\leq C(E_{1}+\frac{\varepsilon}{\nu})\nu^{1/2}\|m_\nu^{1/2}\nabla_\perp\partial_yu^3_{0}\|_{L_t^2H^1}
			\nonumber\\
			&\leq C(E_{1}+\frac{\varepsilon}{\nu})\nu^{1/2}\|\mu_\nu^{1/2}\nabla_\perp\partial_yu^3_{0}\|_{L_t^2H^1}
			\leq C\delta_* E_{2}.
			\notag
		\end{align}
		Thus no unweighted highest $y$ derivative of $u^3_{0}$ is used. Combining these estimates with \eqref{eq:roll-recovery-time-final} proves \eqref{eq:R1-time-complete}.
	\end{proof}
	
	We shall also use repeatedly the two-level consequence of \eqref{eq:bilinear-two-level}. If $b=b(y,z)$ and $f_{0}=0$, then
	\begin{equation}
		\|bf\|_{L^2}\leq C\|b\|_{H^1(\mathbb R^2)}\|f\|_{L^2}
		+C\|b\|_{L^2(\mathbb R^2)}\|f\|_{H^1(\Omega)}.
		\label{eq:two-level-zero-wave-E6}
	\end{equation}
	Thus a differentiated roll factor is paired with the lower wave profile $\mathcal U_{0}$, whereas an undifferentiated roll factor is paired with the differentiated wave profile $\mathcal U_1$. All the mixed estimates below remain at the $3\gamma$ weight; no upgrade from the $2\gamma$-weighted primary profiles is used.
	
	\begin{Lem}
		\label{lem:0neq-transport-revised}
		Assume \eqref{eq:variable-small-assumption} and $E_{1}(T)+\frac{\varepsilon}{\nu}\leq\delta_*$. Then
		\begin{align}
			&\nu^{-1/2}\|\ee^{3\gamma\nu^{1/3}t}\nabla(g_{0\neq}^2+\kappa g_{0\neq}^3)\|_{L_t^2L_x^2}
			+\nu^{-1/2}\|\ee^{3\gamma\nu^{1/3}t}\partial_xg_{0\neq}^3\|_{L_t^2L_x^2}
			\nonumber\\
			&\qquad
			+\nu^{-1/6}\|\ee^{3\gamma\nu^{1/3}t}\nabla g_{0\neq}^3\|_{L_t^2L_x^2}
			\leq C\frac{E_{2}(T)}{\nu}E_{6}(T).
			\label{eq:0neq-forcing-revised}
		\end{align}
	\end{Lem}
	
	\begin{proof}
The good combination satisfies
\begin{equation*}
 g_{0\neq}^2+\kappa g_{0\neq}^3
 =u^\perp_{0}\cdot\nabla_\perp q
 -(u^\perp_{0}\cdot\nabla_\perp\kappa)u^3_{\neq}.
\end{equation*}
In the mixed product argument identified above, use
\Cref{lem:roll-coefficients-E6} for the zero-mode recovery. The continuous-domain products needed here are
\eqref{eq:anis-prod-L2} and \eqref{eq:two-level-zero-wave-E6}.
Together with \Cref{lem:E6-wave-recovery}, they give
\begin{align*}
 &\|\nabla(g_{0\neq}^2+\kappa g_{0\neq}^3)\|_{L^2}\\
 &\quad\leq CE_2\mathcal U_1
 +C\bigl(\|\nabla_\perp u^\perp_{0}\|_{H^1}
 +\|\nabla_\perp D_z^Vu^\perp_{0}\|_{H^1}\bigr)\mathcal U_0.
\end{align*}
Here a derivative on the roll uses the two-level product; one on the
wave uses $\mathcal U_1$. The terms containing $\nabla\kappa$ or
$\nabla^2\kappa$ obey the same bound by $\|\kappa\|_{H^3}\leq C\delta_*$.
For the third component, retain the identity
\begin{equation}
 g_{0\neq}^3=(u^2_{0}+\kappa u^3_{0})\partial_yu^3_{\neq}
 +u^3_{0}D_z^Vu^3_{\neq}.
 \label{eq:g03-BR-decomposition}
\end{equation}
The same product estimates yield
\begin{align*}
 \|\partial_xg_{0\neq}^3\|_{L^2}
 &\leq CE_2(\mathcal U_0+\mathcal U_1),\\
 \|\nabla g_{0\neq}^3\|_{L^2}
 &\leq C\bigl(\|\nabla_\perp u^\perp_{0}\|_{H^1}
 +\|\nabla_\perp D_z^Vu^\perp_{0}\|_{H^1}\bigr)\mathcal U_0
 +CE_2(\mathcal U_1+\nu^{-2/3}\mathcal U_0).
\end{align*}
The last term uses \eqref{eq:E6-recovery-W3-second} only when a
second $y$ derivative falls on $u^3_{\neq}$.
The time integration is fixed by
\eqref{eq:U-profiles-revised} and \eqref{eq:R1-time-complete}:
\begin{align}
 \nu^{-1/2}E_2\|\ee^{3\gamma\nu^{1/3}t}\mathcal U_1\|_{L_t^2}
 &\leq C\frac{E_2}{\nu}E_6,
 \label{eq:0neq-time-ZU1}\\
 \nu^{-1/2}\bigl\|\|\nabla_\perp u^\perp_{0}\|_{H^1}
 +\|\nabla_\perp D_z^Vu^\perp_{0}\|_{H^1}\bigr\|_{L_t^2}
 \|\ee^{3\gamma\nu^{1/3}t}\mathcal U_0\|_{L_t^\infty}
 &\leq C\frac{E_2}{\nu}E_6.
 \label{eq:0neq-time-RU0}
\end{align}
The $\mathcal U_0$ term in the streamwise derivative costs only
$C\nu^{-2/3}E_2E_6$. In the weak full gradient, the exceptional term costs
$C\nu^{-1/6}\nu^{-2/3}\nu^{-1/6}E_2E_6
=C\nu^{-1}E_2E_6$; all other terms are smaller.
This proves \eqref{eq:0neq-forcing-revised} at the $3\gamma$ weight.
\end{proof}
	
	\begin{Lem}
		\label{lem:neq0-transport-revised}
		Assume \eqref{eq:variable-small-assumption} and $E_{1}(T)+\frac{\varepsilon}{\nu}\leq\delta_*$. Then
		\begin{align}
			&\nu^{-1/2}\|\ee^{3\gamma\nu^{1/3}t}\nabla(g_{\neq0}^2+\kappa g_{\neq0}^3)\|_{L_t^2L_x^2}
			+\nu^{-1/2}\|\ee^{3\gamma\nu^{1/3}t}\partial_xg_{\neq0}^3\|_{L_t^2L_x^2}
			\nonumber\\
			&\qquad
			+\nu^{-1/6}\|\ee^{3\gamma\nu^{1/3}t}\nabla g_{\neq0}^3\|_{L_t^2L_x^2}
			\leq C\frac{E_{2}(T)}{\nu}E_{6}(T).
			\label{eq:neq0-forcing-revised}
		\end{align}
	\end{Lem}
	
	\begin{proof}
We retain the coefficient estimates in this decomposition because
the present roll norm
controls the highest $y$ derivatives with $\mu_\nu^{1/2}$.

		For every $x$-independent scalar $f$,
		\begin{equation*}
			u^\perp_{\neq}\cdot\nabla_\perp f=q\,\partial_yf+u^3_{\neq}D_z^Vf.
		\end{equation*}
		Hence
		\begin{equation}
			g_{\neq0}^j=q\,\partial_yu^j_{0}+u^3_{\neq}D_z^Vu^j_{0},
			\qquad j=2,3.
			\label{eq:gneq0-component-complete}
		\end{equation}
		The good combination has the exact cancellation
		\begin{equation}
			g_{\neq0}^2+\kappa g_{\neq0}^3
			=u^3_{\neq}D_z^Vu^2_{0}-u^2_{\neq}D_z^Vu^3_{0}.
			\label{eq:Gneq0-cancellation-revised}
		\end{equation}
		Indeed, incompressibility gives $\partial_yu^2_{0}+D_z^Vu^3_{0}+\kappa\partial_yu^3_{0}=0$, and $q=u^2_{\neq}+\kappa u^3_{\neq}$.
		
		Differentiating \eqref{eq:Gneq0-cancellation-revised} gives the exact four-term identity
		\begin{align}
			\nabla(g_{\neq0}^2+\kappa g_{\neq0}^3)
			={}&(\nabla u^3_{\neq})D_z^Vu^2_{0}+u^3_{\neq}\nabla D_z^Vu^2_{0}
			\nonumber\\
			&-(\nabla u^2_{\neq})D_z^Vu^3_{0}-u^2_{\neq}\nabla D_z^Vu^3_{0}.
			\notag
		\end{align}
		For the terms in which the derivative falls on the wave, \Cref{lem:E6-wave-recovery,lem:roll-coefficients-E6} give
		\begin{equation*}
			\|(\nabla u^3_{\neq})D_z^Vu^2_{0}\|_{L^2}+\|(\nabla u^2_{\neq})D_z^Vu^3_{0}\|_{L^2}
			\leq C E_{2}\mathcal U_1.
		\end{equation*}
		For the terms in which the derivative falls on the roll, the two-level estimate \eqref{eq:two-level-zero-wave-E6} gives
		\begin{equation*}
			\|u^3_{\neq}\nabla D_z^Vu^2_{0}\|_{L^2}+\|u^2_{\neq}\nabla D_z^Vu^3_{0}\|_{L^2}
			\leq C\bigl(\|\nabla_\perp u^\perp_{0}(t)\|_{H^1}+\|\nabla_\perp D_z^V u^\perp_{0}(t)\|_{H^1}\bigr)\mathcal U_{0}+C E_{2}\mathcal U_1.
		\end{equation*}
		Hence
		\begin{equation*}
			\|\nabla(g_{\neq0}^2+\kappa g_{\neq0}^3)\|_{L^2}
			\leq C E_{2}\mathcal U_1+C\bigl(\|\nabla_\perp u^\perp_{0}(t)\|_{H^1}+\|\nabla_\perp D_z^V u^\perp_{0}(t)\|_{H^1}\bigr)\mathcal U_{0}.
		\end{equation*}
		
		From \eqref{eq:gneq0-component-complete},
		\begin{equation*}
			\partial_xg_{\neq0}^3=(\partial_xq)\partial_yu^3_{0}+(\partial_xu^3_{\neq})D_z^Vu^3_{0}.
		\end{equation*}
		For the first product, use \eqref{eq:two-level-zero-wave-E6} with the $H^1$ roll norm in $\bigl(\|\nabla_\perp u^\perp_{0}(t)\|_{H^1}+\|\nabla_\perp D_z^V u^\perp_{0}(t)\|_{H^1}\bigr)$ and the $H^1$ wave norm in $\mathcal U_1$; for the second, use \eqref{eq:roll-coeff-point-E6}. This gives
		\begin{equation*}
			\|\partial_xg_{\neq0}^3\|_{L^2}
			\leq C\bigl(\|\nabla_\perp u^\perp_{0}(t)\|_{H^1}+\|\nabla_\perp D_z^V u^\perp_{0}(t)\|_{H^1}\bigr)\mathcal U_{0}+C E_{2}(\mathcal U_{0}+\mathcal U_1).
		\end{equation*}
		
		For the weak full gradient, differentiating \eqref{eq:gneq0-component-complete} gives
		\begin{align}
			\nabla g_{\neq0}^3
			={}&(\nabla q)\partial_yu^3_{0}+q\nabla\partial_yu^3_{0}
			\nonumber\\
			&+(\nabla u^3_{\neq})D_z^Vu^3_{0}+u^3_{\neq}\nabla D_z^Vu^3_{0}.
			\label{eq:gneq0-gradient-explicit}
		\end{align}
		The $z$-component hidden in $q\nabla\partial_yu^3_{0}$ is not treated by an unweighted highest-$y$ derivative. Indeed, since $D_z^V=\partial_z-\kappa\partial_y$,
		\begin{equation*}
			\partial_z\partial_yu^3_{0}
			=\partial_yD_z^Vu^3_{0}+(\partial_y\kappa)\partial_yu^3_{0}+\kappa\partial_y^2u^3_{0}.
		\end{equation*}
		The first piece satisfies
		\begin{equation*}
			\|q\,\partial_yD_z^Vu^3_{0}\|_{L^2}
			\leq C\bigl(\|\nabla_\perp u^\perp_{0}(t)\|_{H^1}+\|\nabla_\perp D_z^V u^\perp_{0}(t)\|_{H^1}\bigr)\mathcal U_{0}+C E_{2}\mathcal U_1
		\end{equation*}
		by \eqref{eq:two-level-zero-wave-E6}. The second piece is kept in spacetime. The already established stronger estimate \eqref{eq:kappagrad-roll-time-E6} implies
		\begin{equation*}
			\nu^{1/2}\|(\partial_y\kappa)\partial_yu^3_{0}\|_{L_t^2H^1}
			\leq C\delta_* E_{2}.
		\end{equation*}
		Together with \eqref{eq:E6-recovery-low} and the two-level product estimate, this contributes at most $C\delta_*(E_{2}/\nu)E_{6}$ to the weak forcing norm. For the third piece, \eqref{eq:kappa-H2-transient-revised}, $m_\nu\leq\mu_\nu$, and \eqref{eq:anis-prod-L2-coeff} give the pointwise bound
		\begin{equation*}
			\|\kappa q\,\partial_y^2u^3_{0}\|_{L^2}
			\leq C(E_{1}+\frac{\varepsilon}{\nu}) E_{2}\mathcal U_1,
		\end{equation*}
		which is smaller than the exceptional term below under the bootstrap smallness assumption. We now record the three remaining regular pieces in \eqref{eq:gneq0-gradient-explicit} explicitly. First, the two-level zero--wave estimate, the two derivative norms in \eqref{eq:R1-time-complete}, \eqref{eq:roll-recovery-final}, and \eqref{eq:E6-recovery-low}--\eqref{eq:E6-recovery-high} give
		\begin{align}
			\|(\nabla q)\partial_yu^3_{0}\|_{L^2}
			&\leq C\|\partial_yu^3_{0}\|_{H^1}\|\nabla q\|_{L^2}
			+C\|\partial_yu^3_{0}\|_{L^2}\|\nabla q\|_{H^1}
			\nonumber\\
			&\leq C\bigl(\|\nabla_\perp u^\perp_{0}(t)\|_{H^1}+\|\nabla_\perp D_z^V u^\perp_{0}(t)\|_{H^1}\bigr)\mathcal U_{0}+C E_{2}\mathcal U_1.
			\notag
		\end{align}
		Next, using \eqref{eq:anis-prod-L2} with the $x$-independent coefficient $D_z^Vu^3_{0}$, together with \eqref{eq:roll-coeff-point-E6},
		\begin{equation*}
			\|(\nabla u^3_{\neq})D_z^Vu^3_{0}\|_{L^2}
			\leq C E_{2}\mathcal U_1.
		\end{equation*}
		Finally, applying the two-level zero--wave estimate with $b=\nabla D_z^Vu^3_{0}$ and $f=u^3_{\neq}$, and using \eqref{eq:roll-coeff-point-E6} and the two derivative norms in \eqref{eq:R1-time-complete},
		\begin{align}
			\|u^3_{\neq}\nabla D_z^Vu^3_{0}\|_{L^2}
			&\leq C\|\nabla D_z^Vu^3_{0}\|_{H^1}\|u^3_{\neq}\|_{L^2}
			+C\|\nabla D_z^Vu^3_{0}\|_{L^2}\|u^3_{\neq}\|_{H^1}
			\nonumber\\
			&\leq C\bigl(\|\nabla_\perp u^\perp_{0}(t)\|_{H^1}+\|\nabla_\perp D_z^V u^\perp_{0}(t)\|_{H^1}\bigr)\mathcal U_{0}+C E_{2}\mathcal U_1.
			\notag
		\end{align}
		Consequently,
		\begin{align*}
			&\| (\nabla q)\partial_yu^3_{0}\|_{L^2}
			+\|(\nabla u^3_{\neq})D_z^Vu^3_{0}\|_{L^2}
			+\|u^3_{\neq}\nabla D_z^Vu^3_{0}\|_{L^2}
			\\ &\quad\leq C E_{2}\mathcal U_1+C\bigl(\|\nabla_\perp u^\perp_{0}(t)\|_{H^1}+\|\nabla_\perp D_z^V u^\perp_{0}(t)\|_{H^1}\bigr)\mathcal U_{0}.
		\end{align*}
		The only term requiring the explicit factor $\mu_\nu^{-1/2}$ is therefore $q\,\partial_y^2u^3_{0}$, which is estimated by \eqref{eq:anis-prod-L2-coeff}, \eqref{eq:E6-recovery-high}, and the weighted roll recovery:
		\begin{equation*}
			\|q\,\partial_y^2u^3_{0}\|_{L^2}
			\leq C\mu_\nu(t)^{-1/2}E_{2}\mathcal U_1(t).
		\end{equation*}
		Here the $H^1$ norms of both $q$ and $D_z^Vq$ required by \eqref{eq:anis-prod-L2-coeff} are contained in $\mathcal U_1$.
		
		The regular terms are integrated using \eqref{eq:0neq-time-ZU1}--\eqref{eq:0neq-time-RU0}; they contribute at most $C(E_{2}/\nu)E_{6}$ to the first two norms and a smaller $C\nu^{-2/3}E_{2}E_{6}$ to the weak norm. Since $\mu_\nu\geq\nu^{2/3}$,
		\begin{align}
			\nu^{-1/6}\|\ee^{3\gamma\nu^{1/3}t}q\,\partial_y^2u^3_{0}\|_{L_t^2L_x^2}
			\leq C\nu^{-1/6}\nu^{-1/3}E_{2}
			\|\ee^{3\gamma\nu^{1/3}t}\mathcal U_1\|_{L_t^2}
			\leq C\frac{E_{2}}{\nu}E_{6}.
			\notag
		\end{align}
		This proves \eqref{eq:neq0-forcing-revised}.
	\end{proof}
	
The remaining mixed forcing is the pressure generated by one roll factor and one wave factor.

\begin{Lem}
		\label{lem:mixed-pressure-revised}
		Assume \eqref{eq:variable-small-assumption} and $E_{1}(T)+\frac{\varepsilon}{\nu}\leq\delta_*$. Let $p^{0\neq}$ be defined by \eqref{eq:pR-final}. Then
		\begin{align}
			&\nu^{-1/2}\|\ee^{3\gamma\nu^{1/3}t}\nabla(\partial_yp^{0\neq}+\kappa\partial_zp^{0\neq})\|_{L_t^2L_x^2}
			+\nu^{-1/2}\|\ee^{3\gamma\nu^{1/3}t}\partial_x\partial_zp^{0\neq}\|_{L_t^2L_x^2}
			\nonumber\\
			&\qquad
			+\nu^{-1/6}\|\ee^{3\gamma\nu^{1/3}t}\nabla\partial_zp^{0\neq}\|_{L_t^2L_x^2}
			\leq C\frac{E_{2}(T)}{\nu}E_{6}(T).
			\label{eq:mixed-pressure-revised}
		\end{align}
	\end{Lem}
	
	\begin{proof}
The mixed pressure corresponds to $p^{(3)}$ in the source identified
above. We check its source using the present roll norms:

\begin{equation}
 \Delta p^{0\neq}
 =-2\sum_{\alpha,\beta\in\{2,3\}}
 \partial_\alpha u^\beta_{0}\,\partial_\beta u^\alpha_{\neq}.
 \label{eq:pR-source-SR}
\end{equation}
Using $\partial_zu^3_{0}=-\partial_yu^2_{0}$ gives
\begin{align}
 \sum_{\alpha,\beta\in\{2,3\}}
 \partial_\alpha u^\beta_{0}\,\partial_\beta u^\alpha_{\neq}
 ={}&(\partial_yu^2_{0}+\kappa\partial_yu^3_{0})\partial_yu^2_{\neq}
 +\partial_yu^3_{0}D_z^Vu^2_{\neq}\nonumber\\
 &+(D_z^Vu^2_{0})\partial_yu^3_{\neq}
 -\partial_yu^2_{0}D_z^Vu^3_{\neq}.
 \label{eq:pR-good-expansion}
\end{align}
By \Cref{lem:roll-coefficients-E6,lem:E6-wave-recovery} and
\eqref{eq:anis-prod-L2}, the first, third, and fourth terms are
bounded by $CE_2\mathcal U_1$.
For the second, \eqref{eq:two-level-zero-wave-E6} gives
\begin{equation*}
 \|\partial_yu^3_{0}D_z^Vu^2_{\neq}\|_{L^2}
 \leq C\bigl(\|\nabla_\perp u^\perp_{0}\|_{H^1}
 +\|\nabla_\perp D_z^Vu^\perp_{0}\|_{H^1}\bigr)\mathcal U_0
 +CE_2\mathcal U_1.
\end{equation*}
The two time estimates \eqref{eq:0neq-time-ZU1}--\eqref{eq:0neq-time-RU0}
therefore imply
\begin{equation}
 \nu^{-1/2}\left\|\ee^{3\gamma\nu^{1/3}t}
 \sum_{\alpha,\beta\in\{2,3\}}
 \partial_\alpha u^\beta_{0}\,\partial_\beta u^\alpha_{\neq}
 \right\|_{L_t^2L_x^2}
 \leq C\frac{E_2}{\nu}E_6.
 \label{eq:pR-source-E6-time}
\end{equation}
Order-zero multipliers control the pressure Hessian. Since this source is $x$-nonzero, the first derivative is also
bounded, uniformly for real $\ell$:
\begin{equation}
 \|\nabla\partial_yp^{0\neq}\|_{L^2}
 +\|\partial_x\partial_zp^{0\neq}\|_{L^2}
 +\|\partial_z^2p^{0\neq}\|_{L^2}
 +\|\nabla\partial_zp^{0\neq}\|_{L^2}
 +\|\partial_zp^{0\neq}\|_{L^2}
 \leq C\left\|\sum_{\alpha,\beta\in\{2,3\}}
 \partial_\alpha u^\beta_{0}\,\partial_\beta u^\alpha_{\neq}\right\|_{L^2}.
 \label{eq:pR-elliptic-from-source}
\end{equation}
The product rule for $\nabla(\kappa\partial_zp^{0\neq})$,
$\|\kappa\|_{W^{1,\infty}}\leq C\delta_*$, and
$\nu^{-1/6}\leq\nu^{-1/2}$ finish the proof.
\end{proof}

\begin{Lem}
		\label{lem:neqneq-forcing-revised}
		Assume $E_{1}(T)+\frac{\varepsilon}{\nu}\leq\delta_*$. Let $g_{\neq\neq}^j=(u_{\neq}\cdot\nabla u^j_{\neq})_{\neq}+\partial_jp^{\neq\neq}$. Then
		\begin{align}
			&\nu^{-1/2}\|\ee^{3\gamma\nu^{1/3}t}\nabla(g_{\neq\neq}^2+\kappa g_{\neq\neq}^3)\|_{L_t^2L_x^2}
			+\nu^{-1/2}\|\ee^{3\gamma\nu^{1/3}t}\partial_xg_{\neq\neq}^3\|_{L_t^2L_x^2}
			\nonumber\\
			&\qquad
			+\nu^{-1/6}\|\ee^{3\gamma\nu^{1/3}t}\nabla g_{\neq\neq}^3\|_{L_t^2L_x^2}
			\leq C\frac{E_{4}(T)^2}{\nu}.
			\label{eq:neqneq-forcing-revised}
		\end{align}
	\end{Lem}
	
	\begin{proof}
Combine \Cref{lem:quadratic-wave-revised,lem:transient-kappa-revised}
at the required forcing weights.

		By \eqref{eq:quadratic-time-final},
		\begin{equation*}
			\|\ee^{4\gamma\nu^{1/3}t}\nabla g_{\neq\neq}^2\|_{L_t^2L_x^2}
			+\|\ee^{4\gamma\nu^{1/3}t}\partial_xg_{\neq\neq}^3\|_{L_t^2L_x^2}
			\leq C\nu^{-1/2}E_{4}^2.
		\end{equation*}
		Moreover, \eqref{eq:bad-wave3-weighted-revised} and the pressure part of \eqref{eq:quadratic-time-final} give
		\begin{equation}
			\nu^{1/3}\|\ee^{4\gamma\nu^{1/3}t}\nabla g_{\neq\neq}^3\|_{L_t^2L_x^2}
			\leq C\nu^{-1/2}E_{4}^2.
			\label{eq:wave-weak-forcing-revised}
		\end{equation}
		The first two terms in \eqref{eq:neqneq-forcing-revised} now follow directly.  The last follows by multiplying \eqref{eq:wave-weak-forcing-revised} by $\nu^{-1/2}$.
		
		It remains to control the occurrence of $\kappa g_{\neq\neq}^3$ in the good combination.  By \Cref{lem:transient-kappa-revised},
		\begin{align}
			&\nu^{-1/2}\|\ee^{3\gamma\nu^{1/3}t}\nabla(\kappa g_{\neq\neq}^3)\|_{L_t^2L_x^2}
			\nonumber\\
			&\qquad\leq C(E_{1}+\frac{\varepsilon}{\nu})\nu^{-1/6}
			\|\ee^{4\gamma\nu^{1/3}t}\nabla g_{\neq\neq}^3\|_{L_t^2L_x^2}
			\leq C(E_{1}+\frac{\varepsilon}{\nu})\frac{E_{4}^2}{\nu}.\notag
		\end{align}
		Since $E_{1}+\frac{\varepsilon}{\nu}$ is chosen smaller than an absolute constant in the bootstrap, this is absorbed in the right-hand side of \eqref{eq:neqneq-forcing-revised}.
	\end{proof}
	
We now assemble the interaction classes and record the high-order products needed in the energy closure.

\begin{Lem}
		\label{lem:streak-physical-tame}
		Let $b=b(y,z)\in H^3(\mathbb R^2)$ and let $f=f(x,y,z)$ satisfy $f_{0}=0$. Then
		\begin{align}
			\|bf\|_{H^2}
			&\leq C\|b\|_{H^1}\|(\partial_x,\partial_z)f\|_{H^2}
			+C\|b\|_{H^3}\|(\partial_x,\partial_z)f\|_{L^2},
			\label{eq:streak-physical-tame-H2}\\
			\|bf\|_{H^3}
			&\leq C\|b\|_{H^1}\|(\partial_x,\partial_z)f\|_{H^3}
			+C\|b\|_{H^3}\|(\partial_x,\partial_z)f\|_{H^1}.
			\label{eq:streak-physical-tame-H3}
		\end{align}
	\end{Lem}
	
	\begin{proof}
The coefficient products identified in Appendix~\ref{app:functional-tools} give, for $m=2,3$,
\begin{align}
 \|bf\|_{H^m}
 \leq{}&C\|b\|_{H^1}
   (\|f\|_{H^m}+\|\partial_zf\|_{H^m})\nonumber\\
 &+C\|b\|_{H^3}
   (\|f\|_{H^{m-2}}+\|\partial_zf\|_{H^{m-2}}).
 \label{eq:continuous-WZ-tame}
\end{align}
Their proof uses the flat anisotropic products verified on
$\mathbb R^2$ in \Cref{lem:anisotropic-products-R2}, followed by
Leibniz expansion and Sobolev interpolation; these steps are unchanged.
Since $f_{0}=0$,
$\|f\|_{H^r}+\|\partial_zf\|_{H^r}
\leq C\|(\partial_x,\partial_z)f\|_{H^r}$.
Substitution proves both estimates.
\end{proof}

	\begin{Lem}
		\label{lem:wide-gap-zero-wave-H3}
		Let $b=b(y,z)\in H^4(\mathbb R^2)$ and $f=f(x,y,z)$. Then
		\begin{equation}
			\|bf\|_{H^3(\Omega)}
			\leq C\|b\|_{H^1(\mathbb R^2)}
			\bigl(\|f\|_{H^3(\Omega)}+\|\partial_zf\|_{H^3(\Omega)}\bigr)
			+C\|b\|_{H^4(\mathbb R^2)}
			\bigl(\|f\|_{L^2}+\|\partial_zf\|_{L^2}\bigr).
			\label{eq:wide-gap-zero-wave-H3}
		\end{equation}
	\end{Lem}
	
	\begin{proof}
		
		Let $|\alpha|\leq3$ and let $\beta\leq\alpha$ contain $\ell=|\beta|$ derivatives falling on $b$. Since $b$ is independent of $x$, $\beta$ is purely transverse. The flat product estimate proved in \eqref{eq:flat-fixed-z-product}--\eqref{eq:Hilbert-Sobolev-z}, applied to $\partial^\beta b$ and $\partial^{\alpha-\beta}f$, gives
		\begin{equation}
			\|\partial^\beta b\,\partial^{\alpha-\beta}f\|_{L^2}
			\leq C\|b\|_{H^{\ell+1}(\mathbb R^2)}\,\bigl(\|f\|_{H^{3-\ell}}+\|\partial_zf\|_{H^{3-\ell}}\bigr),
			\qquad 0\leq\ell\leq3.
			\label{eq:wide-gap-Leibniz-level}
		\end{equation}
		Here we have harmlessly enlarged the derivative order on $f$ to $3-\ell$. Put $\theta=\ell/3$. Sobolev interpolation yields
		\begin{align}
			\|b\|_{H^{\ell+1}}
			&\leq \|b\|_{H^1}^{1-\theta}\|b\|_{H^4}^{\theta},\notag\\
			\bigl(\|f\|_{H^{3-\ell}}+\|\partial_zf\|_{H^{3-\ell}}\bigr)
			&\leq \bigl(\|f\|_{H^{3}}+\|\partial_zf\|_{H^{3}}\bigr)^{1-\theta}\bigl(\|f\|_{H^{0}}+\|\partial_zf\|_{H^{0}}\bigr)^{\theta}.\notag
		\end{align}
		Thus each term in \eqref{eq:wide-gap-Leibniz-level} is at most
		\begin{equation*}
\begin{aligned}
&
			\bigl(\|b\|_{H^1}\bigl(\|f\|_{H^{3}}+\|\partial_zf\|_{H^{3}}\bigr)\bigr)^{1-\theta}
			\bigl(\|b\|_{H^4}\bigl(\|f\|_{H^{0}}+\|\partial_zf\|_{H^{0}}\bigr)\bigr)^{\theta}
			\\&\leq \|b\|_{H^1}\bigl(\|f\|_{H^{3}}+\|\partial_zf\|_{H^{3}}\bigr)+\|b\|_{H^4}\bigl(\|f\|_{H^{0}}+\|\partial_zf\|_{H^{0}}\bigr),
\end{aligned}
\end{equation*}
		by the weighted arithmetic--geometric mean inequality. Summing the finitely many Leibniz terms proves \eqref{eq:wide-gap-zero-wave-H3}.
	\end{proof}
	
	\begin{Prop}
		\label{prop:forcing-final}
		Assume the restarted-data bounds \eqref{eq:restarted-low-data-standing}--\eqref{eq:restarted-high-data-standing}, \eqref{eq:variable-small-assumption}, and $E_{1}(T)+\frac{\varepsilon}{\nu}\leq\delta_*$, where $\delta_*$ is sufficiently small. Then
		\begin{equation*}
			\mathfrak F(T)
			\leq C\frac{E_{2}(T)}{\nu}E_{6}(T)+C\frac{E_{4}(T)^2}{\nu}.
		\end{equation*}
		Moreover,
		\begin{align}
			&\|\ee^{2\gamma\nu^{1/3}t}(u^\perp_{0}\cdot\nabla_\perp u_{\neq}+u^\perp_{\neq}\cdot\nabla_\perp u^\perp_{0})\|_{Y^2}
			\nonumber\\
			&\qquad\leq C E_{2}(T)E_{5}(T)+C E_{3}(T)E_{4}(T),
			\label{eq:mixed-Y2-final}
		\end{align}
		and \eqref{eq:wave-Y2-final} controls the $\neq\cdot\neq$ interaction in $Y^2$.
	\end{Prop}
	
	\begin{proof}
		The estimate for $\mathfrak F$ is obtained by summing \Cref{lem:0neq-transport-revised,lem:neq0-transport-revised,lem:mixed-pressure-revised,lem:neqneq-forcing-revised}. We now prove the $Y^2$ estimate with a two-level tensor estimate which keeps the low wave factor at physical order zero. This is the point at which the continuous transverse low frequency matters.
		
		Extend $u^\perp_{0}$ to $\widetilde u^\perp_{0}=(0,u^2_{0},u^3_{0})$. Since $\operatorname{div}_\perp u^\perp_{0}=0$ and $\operatorname{div}u_{\neq}=0$,
		\begin{equation*}
			u^\perp_{0}\cdot\nabla_\perp u_{\neq}+u^\perp_{\neq}\cdot\nabla_\perp u^\perp_{0}
			=\operatorname{div}\bigl(\widetilde u^\perp_{0}\otimes u_{\neq}+u_{\neq}\otimes\widetilde u^\perp_{0}\bigr).
		\end{equation*}
		Therefore
		\begin{align}
			&\|\ee^{2\gamma\nu^{1/3}t}(u^\perp_{0}\cdot\nabla_\perp u_{\neq}+u^\perp_{\neq}\cdot\nabla_\perp u^\perp_{0})\|_{Y^2}
			\nonumber\\
			&\quad\leq C\|\ee^{2\gamma\nu^{1/3}t}(\widetilde u^\perp_{0}\otimes u_{\neq}+u_{\neq}\otimes\widetilde u^\perp_{0})\|_{L_t^\infty H^3}
			\nonumber\\
			&\qquad+C\nu^{1/2}\|\ee^{2\gamma\nu^{1/3}t}(\widetilde u^\perp_{0}\otimes u_{\neq}+u_{\neq}\otimes\widetilde u^\perp_{0})\|_{L_t^2H^4}.
			\label{eq:mixed-Y2-div-reduction}
		\end{align}
		
		We record the wave norms used in the two tensor estimates. The definition of $E_{5}$ and the $x$-Poincar\'e inequality give
		\begin{equation}
			\|\ee^{2\gamma\nu^{1/3}t}(u_{\neq},\partial_zu_{\neq})\|_{L_t^\infty H^3}
			+\nu^{1/2}\|\ee^{2\gamma\nu^{1/3}t}(u_{\neq},\partial_zu_{\neq})\|_{L_t^2H^4}
			\leq C E_{5}.
			\label{eq:mixed-wave-high-two-level}
		\end{equation}
		For the lower wave level, \Cref{lem:velocity-reconstruction-final} gives
		\begin{equation*}
			\|u_{\neq}\|_{L^2}+\|\partial_zu_{\neq}\|_{L^2}\leq C\Pi_{0},
			\qquad
			\|u_{\neq}\|_{H^1}+\|\partial_zu_{\neq}\|_{H^1}\leq C(\Pi_{0}+\Pi_1).
		\end{equation*}
		Indeed, $\|\partial_zu_{\neq}\|_{L^2}\leq\|\partial_x\partial_zu_{\neq}\|_{L^2}$ by the streamwise gap, and the latter is contained in \eqref{eq:velocity-reconstruction-final} with $s=0$. Hence \eqref{eq:P-profiles-revised} implies
		\begin{align}
			\|\ee^{2\gamma\nu^{1/3}t}(u_{\neq},\partial_zu_{\neq})\|_{L_t^\infty L^2}
			&\leq C E_{4},
			\label{eq:mixed-low-Linf}\\
			\nu^{1/2}\|\ee^{2\gamma\nu^{1/3}t}(u_{\neq},\partial_zu_{\neq})\|_{L_t^2H^1}
			&\leq C E_{4}.
			\label{eq:mixed-low-L2}
		\end{align}
		For completeness, the $\Pi_{0}$ contribution in \eqref{eq:mixed-low-L2} is
		\begin{equation*}
			\nu^{1/2}\|\ee^{2\gamma\nu^{1/3}t}\Pi_{0}\|_{L_t^2}
			\leq \nu^{1/3}\Bigl(\nu^{1/6}\|\ee^{2\gamma\nu^{1/3}t}\Pi_{0}\|_{L_t^2}\Bigr)
			\leq E_{4},
		\end{equation*}
		whereas the $\Pi_1$ contribution is directly bounded by $E_{4}$. Thus no spurious factor $\nu^{1/3}$ is placed in front of the entire lower wave norm.
		
		For the $L_t^\infty H^3$ tensor norm, apply \Cref{lem:wide-gap-zero-wave-H3} componentwise with $b=u^j_{0}$ and $f=u^k_{\neq}$. The roll bounds give
		\begin{equation*}
			\|u^\perp_{0}\|_{H^1}\leq C E_{2},
			\qquad
			\|u^\perp_{0}\|_{H^4}\leq E_{3}.
		\end{equation*}
		Using \eqref{eq:mixed-wave-high-two-level} and \eqref{eq:mixed-low-Linf},
		\begin{equation}
			\|\ee^{2\gamma\nu^{1/3}t}(\widetilde u^\perp_{0}\otimes u_{\neq}+u_{\neq}\otimes\widetilde u^\perp_{0})\|_{L_t^\infty H^3}
			\leq C E_{2} E_{5}+C E_{3}E_{4}.
			\label{eq:mixed-tensor-H3}
		\end{equation}
		Notice that the top roll derivatives in this estimate are paired only with $\|u_{\neq}\|_{L^2}+\|\partial_zu_{\neq}\|_{L^2}$, which is controlled pointwise by $\Pi_{0}$; no pointwise $H^1$ norm of the full $u^3_{\neq}$ is used.
		
		For the dissipative tensor norm, apply \Cref{lem:tame-zero-wave-high} with $m=4$. It gives, componentwise,
		\begin{align}
			\|u^j_{0}u^k_{\neq}\|_{H^4}
			\leq C\|u^j_{0}\|_{H^1}
			\bigl(\|u^k_{\neq}\|_{H^4}+\|\partial_zu^k_{\neq}\|_{H^4}\bigr)
			+C\|u^j_{0}\|_{H^4}
			\bigl(\|u^k_{\neq}\|_{H^1}+\|\partial_zu^k_{\neq}\|_{H^1}\bigr).\notag
		\end{align}
		Multiplying by $\nu^{1/2}\ee^{2\gamma\nu^{1/3}t}$, integrating in time, and using \eqref{eq:mixed-wave-high-two-level} and \eqref{eq:mixed-low-L2}, we obtain
		\begin{equation}
			\nu^{1/2}\|\ee^{2\gamma\nu^{1/3}t}(\widetilde u^\perp_{0}\otimes u_{\neq}+u_{\neq}\otimes\widetilde u^\perp_{0})\|_{L_t^2H^4}
			\leq C E_{2} E_{5}+C E_{3}E_{4}.
			\label{eq:mixed-tensor-H4}
		\end{equation}
		Combining \eqref{eq:mixed-Y2-div-reduction}, \eqref{eq:mixed-tensor-H3}, and \eqref{eq:mixed-tensor-H4} proves \eqref{eq:mixed-Y2-final}. This proof uses the high roll norm only against genuinely lower wave profiles and therefore avoids the false implication $E_{4}+E_{6}\Rightarrow \|u_{\neq}\|_{L_t^\infty H^1}$. Finally, \Cref{lem:wave-Y2-final} gives the $\neq\cdot\neq$ contribution.
	\end{proof}

	\begin{Lem}
		\label{lem:primary-forcing-2gamma}
		Assume \eqref{eq:variable-small-assumption} and $E_{1}(T)+\frac{\varepsilon}{\nu}\leq\delta_*$, with $\delta_*$ sufficiently small. Then
		\begin{align}
			\mathfrak F_{\rm prim}(T):={}&
			\nu^{-1/2}\|\ee^{2\gamma\nu^{1/3}t}\nabla g^2\|_{L_t^2L_x^2}
			+\nu^{-1/2}\|\ee^{2\gamma\nu^{1/3}t}(\partial_x,\partial_z)g^3\|_{L_t^2L_x^2}
			\nonumber\\
			&\leq C\frac{E_{2}(T)}{\nu}E_{6}(T)+C\frac{E_{4}(T)^2}{\nu}.
			\label{eq:primary-forcing-2gamma}
		\end{align}
	\end{Lem}
	
	\begin{proof}
		Write
		\begin{equation*}
			K_{0}^{\perp}:=\frac{E_{2}(T)}{\nu}E_{6}(T),
			\qquad
			K_{\neq}:=\frac{E_{4}(T)^2}{\nu}.
		\end{equation*}
		We estimate the two components required by the coupled primary system separately.
		
		\emph{Step 1: the strong derivative of $g^2$.}
		Expanding the combined forcing,
		\begin{equation}
			\nabla g^2=\nabla (g^2+\kappa g^3)-\nabla(\kappa g^3).
			\label{eq:g2-from-G-primary}
		\end{equation}
		By \Cref{prop:forcing-final}, the complete $E_{6}$ forcing estimate gives
		\begin{equation}
			\nu^{-1/2}\|\ee^{2\gamma\nu^{1/3}t}\nabla (g^2+\kappa g^3)\|_{L_t^2L_x^2}
			\leq
			\nu^{-1/2}\|\ee^{3\gamma\nu^{1/3}t}\nabla (g^2+\kappa g^3)\|_{L_t^2L_x^2}
			\leq C(K_{0}^{\perp}+K_{\neq}).
			\label{eq:g2-G-primary}
		\end{equation}
		By the lower-weight multiplier estimate \eqref{eq:kappa-multiplier-gain-primary} and the weak third-component term in \eqref{eq:Fforcing-final},
		\begin{align}
			\nu^{-1/2}\|\ee^{2\gamma\nu^{1/3}t}\nabla(\kappa g^3)\|_{L_t^2L_x^2}
			&\leq C(E_{1}+\frac{\varepsilon}{\nu})\nu^{-1/6}
			\|\ee^{3\gamma\nu^{1/3}t}\nabla g^3\|_{L_t^2L_x^2}
			\nonumber\\
			&\leq C\delta_*(K_{0}^{\perp}+K_{\neq}).
			\label{eq:g2-kappa-primary}
		\end{align}
		Combining \eqref{eq:g2-from-G-primary}--\eqref{eq:g2-kappa-primary} yields
		\begin{equation}
			\nu^{-1/2}\|\ee^{2\gamma\nu^{1/3}t}\nabla g^2\|_{L_t^2L_x^2}
			\leq C(K_{0}^{\perp}+K_{\neq}).
			\label{eq:g2-primary-final-bound}
		\end{equation}
		
		\emph{Step 2: the streamwise derivative of $g^3$.}
		The definition of $\mathfrak F$ immediately gives
		\begin{equation}
			\nu^{-1/2}\|\ee^{2\gamma\nu^{1/3}t}\partial_xg^3\|_{L_t^2L_x^2}
			\leq
			\nu^{-1/2}\|\ee^{3\gamma\nu^{1/3}t}\partial_xg^3\|_{L_t^2L_x^2}
			\leq C(K_{0}^{\perp}+K_{\neq}).
			\label{eq:g3-x-primary}
		\end{equation}
		
		\emph{Step 3: a strong $D_z^V$ derivative of the mixed transport terms.}
		For the $0\cdot\neq$ interaction, use \eqref{eq:g03-BR-decomposition} and the fact that $D_z^V$ is a derivation:
		\begin{align}
			D_z^Vg_{0\neq}^3
			={}&(D_z^V\bigl(u^2_{0}+\kappa u^3_{0}\bigr))\partial_yu^3_{\neq}+\bigl(u^2_{0}+\kappa u^3_{0}\bigr)D_z^V\partial_yu^3_{\neq}
			\nonumber\\
			&+(D_z^Vu^3_{0})D_z^Vu^3_{\neq}+u^3_{0}(D_z^V)^2u^3_{\neq}.
			\label{eq:Dz-g03-primary-expansion}
		\end{align}
		We first record the coefficient bound
		\begin{equation}
			\|D_z^V\bigl(u^2_{0}+\kappa u^3_{0}\bigr)\|_{H^1}
			\leq C E_{2}.
			\label{eq:DzBR-primary}
		\end{equation}
		By the product rule,
		\begin{equation*}
			D_z^V\bigl(u^2_{0}+\kappa u^3_{0}\bigr)=D_z^Vu^2_{0}+(D_z^V\kappa)u^3_{0}+\kappa D_z^Vu^3_{0},
		\end{equation*}
		and \Cref{lem:roll-coefficients-E6}, $\|\kappa\|_{H^3}\leq C\delta_*$, and the two-dimensional $H^2\cdot H^1\to H^1$ product estimate prove \eqref{eq:DzBR-primary}. Moreover,
		\begin{equation*}
			\|(D_z^V)^2u^3_{\neq}\|_{L^2}
			\leq C\|\nabla D_z^Vu^3_{\neq}\|_{L^2}
			\leq C\mathcal U_1,
		\end{equation*}
		while \eqref{eq:E6-D-commuted} gives $\|D_z^V\partial_yu^3_{\neq}\|_{L^2}\leq C\mathcal U_1$. Applying \eqref{eq:anis-prod-L2} to the first and third products in \eqref{eq:Dz-g03-primary-expansion} gives
		\begin{align}
			\|(D_z^V\bigl(u^2_{0}+\kappa u^3_{0}\bigr))\partial_yu^3_{\neq}\|_{L^2}
			+\|(D_z^Vu^3_{0})D_z^Vu^3_{\neq}\|_{L^2}
			\leq C E_{2}\bigl(\mathcal U_{0}+\mathcal U_1\bigr).\notag
		\end{align}
		For the second product, $\bigl(u^2_{0}+\kappa u^3_{0}\bigr)\in H^2(\mathbb R^2)\hookrightarrow L^\infty(\mathbb R^2)$ and \eqref{eq:E6-D-commuted} give
		\begin{equation*}
			\|\bigl(u^2_{0}+\kappa u^3_{0}\bigr)D_z^V\partial_yu^3_{\neq}\|_{L^2}
			\leq C\|\bigl(u^2_{0}+\kappa u^3_{0}\bigr)\|_{H^2}\|D_z^V\partial_yu^3_{\neq}\|_{L^2}
			\leq C E_{2}\mathcal U_1.
		\end{equation*}
		For the fourth product we use the separate pointwise roll recovery $\|u^3_{0}\|_{L^\infty}\leq C E_{2}$ from \Cref{lem:roll-coefficients-E6}:
		\begin{equation*}
			\|u^3_{0}(D_z^V)^2u^3_{\neq}\|_{L^2}
			\leq \|u^3_{0}\|_{L^\infty}\|(D_z^V)^2u^3_{\neq}\|_{L^2}
			\leq C E_{2}\mathcal U_1.
		\end{equation*}
		Consequently,
		\begin{equation*}
			\|D_z^Vg_{0\neq}^3\|_{L^2}
			\leq C E_{2}\bigl(\mathcal U_{0}+\mathcal U_1\bigr).
		\end{equation*}
		Consequently,
		\begin{equation}
			\nu^{-1/2}\|\ee^{3\gamma\nu^{1/3}t}D_z^Vg_{0\neq}^3\|_{L_t^2L_x^2}
			\leq C K_{0}^{\perp},
			\label{eq:Dz-g03-primary-time}
		\end{equation}
		because the $\mathcal U_1$ term uses $\nu^{1/2}\|\ee^{3\gamma\nu^{1/3}t}\mathcal U_1\|_{L_t^2}\lesssim E_{6}$, whereas the $\mathcal U_{0}$ term is smaller by \eqref{eq:U-profiles-revised}.
		
		For the $\neq\cdot0$ interaction, start from \eqref{eq:gneq0-component-complete}:
		\begin{align}
			D_z^Vg_{\neq0}^3
			=(D_z^Vq)\partial_yu^3_{0}+qD_z^V\partial_yu^3_{0}
			+(D_z^Vu^3_{\neq})D_z^Vu^3_{0}+u^3_{\neq}(D_z^V)^2u^3_{0}.
			\notag
		\end{align}
		The first product is estimated by the two-level inequality \eqref{eq:two-level-zero-wave-E6}:
		\begin{align}
			\|(D_z^Vq)\partial_yu^3_{0}\|_{L^2}
			&\leq C\|\partial_yu^3_{0}\|_{H^1}\|D_z^Vq\|_{L^2}
			+C\|\partial_yu^3_{0}\|_{L^2}\|D_z^Vq\|_{H^1}
			\nonumber\\
			&\leq C\mu_\nu^{-1/2}E_{2}\mathcal U_{0}+C E_{2}\mathcal U_1.
			\label{eq:Dz-gneq0-term1-primary}
		\end{align}
		Next,
		\begin{equation*}
			D_z^V\partial_yu^3_{0}=\partial_yD_z^Vu^3_{0}+(\partial_y\kappa)\partial_yu^3_{0}.
		\end{equation*}
		The first part is estimated pointwise by \eqref{eq:anis-prod-L2}:
		\begin{equation}
			\|q\,\partial_yD_z^Vu^3_{0}\|_{L^2}
			\leq C\bigl(\|\nabla_\perp u^\perp_{0}(t)\|_{H^1}+\|\nabla_\perp D_z^V u^\perp_{0}(t)\|_{H^1}\bigr)\mathcal U_{0}.
			\label{eq:Dz-gneq0-term2a-primary}
		\end{equation}
		The coefficient in the second part is treated in spacetime rather than by an unweighted pointwise estimate. This is the componentwise version of the stronger estimate \eqref{eq:kappagrad-roll-time-E6}; we record the short calculation here because it is used directly in the primary forcing norm. Since $H^2(\mathbb R^2)$ acts as a multiplier on $H^1(\mathbb R^2)$ and $H^1\cdot H^2\to H^1$ in two dimensions,
		\begin{align}
			\nu^{1/2}\|(\partial_y\kappa)\partial_yu^3_{0}\|_{L_t^2H^1}
			&\leq C\nu^{1/2}
			\left(\int_{0}^T\|\kappa(t)\|_{H^2}^2\|\partial_yu^3_{0}(t)\|_{H^2}^2\,dt\right)^{1/2}
			\nonumber\\
			&\leq C(E_{1}+\frac{\varepsilon}{\nu})\nu^{1/2}\|m_\nu^{1/2}\partial_yu^3_{0}\|_{L_t^2H^2}
			\nonumber\\
			&\leq C(E_{1}+\frac{\varepsilon}{\nu})\nu^{1/2}\|\mu_\nu^{1/2}\partial_yu^3_{0}\|_{L_t^2H^2}
			\leq C\delta_* E_{2}.
			\notag
		\end{align}
		Therefore, by \eqref{eq:anis-prod-L2} and \eqref{eq:E6-recovery-low},
		\begin{align}
			&\nu^{-1/2}\|\ee^{3\gamma\nu^{1/3}t}q(\partial_y\kappa)\partial_yu^3_{0}\|_{L_t^2L_x^2}
			\nonumber\\
			&\qquad\leq C\nu^{-1/2}\|(\partial_y\kappa)\partial_yu^3_{0}\|_{L_t^2H^1}
			\|\ee^{3\gamma\nu^{1/3}t}\mathcal U_{0}\|_{L_t^\infty}
			\leq C\delta_* K_{0}^{\perp}.
			\label{eq:Dz-gneq0-term2b-primary-time}
		\end{align}
		Furthermore,
		\begin{equation*}
			\|(D_z^V)^2u^3_{0}\|_{H^1}
			\leq C\|\nabla_\perp D_z^Vu^3_{0}\|_{H^1}
			\leq C\bigl(\|\nabla_\perp u^\perp_{0}(t)\|_{H^1}+\|\nabla_\perp D_z^V u^\perp_{0}(t)\|_{H^1}\bigr),
		\end{equation*}
		because $(D_z^V)^2u^3_{0}=\partial_zD_z^Vu^3_{0}-\kappa\partial_yD_z^Vu^3_{0}$ and $\|\kappa\|_{H^2}$ is small. Therefore \eqref{eq:anis-prod-L2}, \eqref{eq:E6-recovery-low}--\eqref{eq:E6-recovery-high}, and \Cref{lem:roll-coefficients-E6} imply
		\begin{align}
			\|(D_z^Vu^3_{\neq})D_z^Vu^3_{0}\|_{L^2}
			&\leq C E_{2}(\mathcal U_{0}+\mathcal U_1),
			\label{eq:Dz-gneq0-term3-primary}\\
			\|u^3_{\neq}(D_z^V)^2u^3_{0}\|_{L^2}
			&\leq C\bigl(\|\nabla_\perp u^\perp_{0}(t)\|_{H^1}+\|\nabla_\perp D_z^V u^\perp_{0}(t)\|_{H^1}\bigr)\mathcal U_{0}.
			\label{eq:Dz-gneq0-term4-primary}
		\end{align}
		Combining the pointwise terms \eqref{eq:Dz-gneq0-term1-primary}, \eqref{eq:Dz-gneq0-term2a-primary}, \eqref{eq:Dz-gneq0-term3-primary}, and \eqref{eq:Dz-gneq0-term4-primary}, while keeping \eqref{eq:Dz-gneq0-term2b-primary-time} separate, gives
		\begin{align}
			\|D_z^Vg_{\neq0}^3-q(\partial_y\kappa)\partial_yu^3_{0}\|_{L^2}
			\leq{}&C E_{2}\mathcal U_1+C\bigl(\|\nabla_\perp u^\perp_{0}(t)\|_{H^1}+\|\nabla_\perp D_z^V u^\perp_{0}(t)\|_{H^1}\bigr)\mathcal U_{0}
			+C\mu_\nu^{-1/2}E_{2}\mathcal U_{0}.
			\notag
		\end{align}
		The three terms on the right are all integrable with the strong prefactor. Indeed,
		\begin{align}
			\nu^{-1/2}E_{2}\|\ee^{3\gamma\nu^{1/3}t}\mathcal U_1\|_{L_t^2}
			&\leq C K_{0}^{\perp},
			\notag\\
			\nu^{-1/2}\|\bigl(\|\nabla_\perp u^\perp_{0}(t)\|_{H^1}+\|\nabla_\perp D_z^V u^\perp_{0}(t)\|_{H^1}\bigr)\|_{L_t^2}
			\|\ee^{3\gamma\nu^{1/3}t}\mathcal U_{0}\|_{L_t^\infty}
			&\leq C K_{0}^{\perp},
			\notag\\
			\nu^{-1/2}E_{2}\|\mu_\nu^{-1/2}\|_{L_t^\infty}
			\|\ee^{3\gamma\nu^{1/3}t}\mathcal U_{0}\|_{L_t^2}
			&\leq C\nu^{-1/2}\nu^{-1/3}\nu^{-1/6}E_{2}E_{6}
			=C K_{0}^{\perp}.
			\notag
		\end{align}
		Together with the separately estimated term \eqref{eq:Dz-gneq0-term2b-primary-time}, this yields
		\begin{equation}
			\nu^{-1/2}\|\ee^{3\gamma\nu^{1/3}t}D_z^Vg_{\neq0}^3\|_{L_t^2L_x^2}
			\leq C K_{0}^{\perp}.
			\label{eq:Dz-gneq0-primary-time}
		\end{equation}
		
		To recover the physical $z$ derivative of either transport contribution $h\in\{g_{0\neq}^3,g_{\neq0}^3\}$, use
		\begin{equation*}
			\partial_zh=D_z^Vh+\kappa\partial_yh.
		\end{equation*}
		By \eqref{eq:kappa-multiplier-gain-primary} and the weak estimates in \Cref{lem:0neq-transport-revised,lem:neq0-transport-revised},
		\begin{equation*}
			\nu^{-1/2}\|\ee^{2\gamma\nu^{1/3}t}\kappa\partial_yh\|_{L_t^2L_x^2}
			\leq C(E_{1}+\frac{\varepsilon}{\nu})\nu^{-1/6}\|\ee^{3\gamma\nu^{1/3}t}\nabla h\|_{L_t^2L_x^2}
			\leq C\delta_* K_{0}^{\perp}.
		\end{equation*}
		Together with \eqref{eq:Dz-g03-primary-time} and \eqref{eq:Dz-gneq0-primary-time}, this proves
		\begin{equation}
			\nu^{-1/2}\|\ee^{2\gamma\nu^{1/3}t}\partial_z(g_{0\neq}^3+g_{\neq0}^3)\|_{L_t^2L_x^2}
			\leq C K_{0}^{\perp}.
			\label{eq:mixed-transport-z-primary}
		\end{equation}
		
		\emph{Step 4: wave--wave and mixed-pressure contributions to the physical $z$ derivative.}
		By \eqref{eq:quadratic-time-final},
		\begin{equation}
			\nu^{-1/2}\|\ee^{2\gamma\nu^{1/3}t}\partial_zg_{\neq\neq}^3\|_{L_t^2L_x^2}
			\leq C K_{\neq},
			\label{eq:wavewave-primary-strong}
		\end{equation}
		because $g_{\neq\neq}^3=(u_{\neq}\cdot\nabla u^3_{\neq})_{\neq}+\partial_zp^{\neq\neq}$ and \eqref{eq:quadratic-time-final} controls both $\partial_z(u_{\neq}\cdot\nabla u^3_{\neq})$ and $\partial_z^2p^{\neq\neq}$. For the mixed pressure, $\Delta p^{0\neq}=-2\sum_{\alpha,\beta\in\{2,3\}}\partial_\alpha u^\beta_{0}\,\partial_\beta u^\alpha_{\neq}$ and \eqref{eq:pR-elliptic-from-source} give
		\begin{equation*}
			\|\partial_z^2p^{0\neq}\|_{L^2}
			\leq C\|\sum_{\alpha,\beta\in\{2,3\}}\partial_\alpha u^\beta_{0}\,\partial_\beta u^\alpha_{\neq}\|_{L^2}.
		\end{equation*}
		Hence \eqref{eq:pR-source-E6-time} implies
		\begin{equation}
			\nu^{-1/2}\|\ee^{2\gamma\nu^{1/3}t}\partial_z^2p^{0\neq}\|_{L_t^2L_x^2}
			\leq C K_{0}^{\perp}.
			\label{eq:pR-primary-strong}
		\end{equation}
		The streamwise derivative of all four interaction classes, including the pressure terms, has already been controlled in \eqref{eq:g3-x-primary}. Summing \eqref{eq:g3-x-primary}, \eqref{eq:mixed-transport-z-primary}, \eqref{eq:wavewave-primary-strong}, and \eqref{eq:pR-primary-strong} gives
		\begin{equation}
			\nu^{-1/2}\|\ee^{2\gamma\nu^{1/3}t}(\partial_x,\partial_z)g^3\|_{L_t^2L_x^2}
			\leq C(K_{0}^{\perp}+K_{\neq}).
			\label{eq:g3-primary-final-bound}
		\end{equation}
		Finally, \eqref{eq:g2-primary-final-bound} and \eqref{eq:g3-primary-final-bound} prove \eqref{eq:primary-forcing-2gamma}.
	\end{proof}

\subsection{Streak--wave interactions}
The high-order wave estimate requires the following frequency splitting and tame product bounds. They retain the transient gain in the streak coefficient.
The elementary interpolation between the transient factor and the enhanced-dissipation weight is recorded in \Cref{lem:mnu-interpolation-revised}; see Appendix~\ref{app:interpolation-tools}.
	
	\begin{Lem}
		\label{lem:streak-spanwise-profile}
One has
		\begin{align}
			&\|\ee^{2\gamma\nu^{1/3}t}\|(\partial_x,\partial_z)u^3_{\neq}(t)\|_{H^0}\|_{L_t^\infty}
			+\nu^{1/2}\|\ee^{2\gamma\nu^{1/3}t}\|(\partial_x,\partial_z)u^3_{\neq}(t)\|_{H^1}\|_{L_t^2}
			\leq C E_{4}(T),
			\label{eq:spanwise-low-profile}\\
			&\|\ee^{2\gamma\nu^{1/3}t}m_\nu(t)\|(\partial_x,\partial_z)u^3_{\neq}(t)\|_{H^2}\|_{L_t^\infty}
			+\nu^{1/2}\|\ee^{2\gamma\nu^{1/3}t}m_\nu(t)\|(\partial_x,\partial_z)u^3_{\neq}(t)\|_{H^3}\|_{L_t^2}
			\nonumber\\
			&\hspace{35mm}\leq C\nu E_{5}(T)+C E_{4}(T)+C E_{6}(T).
			\label{eq:spanwise-high-profile}
		\end{align}
	\end{Lem}
	
	\begin{proof}
		Fix $k\in2\pi\mathbb Z\setminus\{0\}$ and $(\eta,\ell)\in\mathbb R^2$.
		Thus $1\leq\sqrt{k^2+\ell^2}\leq\langle k,\eta,\ell\rangle$. We first prove the scalar multiplier inequality
		\begin{equation}
\begin{aligned}
m_\nu(t)\,\sqrt{k^2+\ell^2}\langle k,\eta,\ell\rangle^r
			&\leq C_\gamma\Bigl(
			\nu\sqrt{k^2+\ell^2}\langle k,\eta,\ell\rangle^{r+1}
			+(k^2+\ell^2)\langle k,\eta,\ell\rangle^{r-2}
			\Bigr.\\&\quad+C_\gamma\Bigl.\nu^{2/3}\ee^{\gamma \nu^{1/3}t}\langle k,\eta,\ell\rangle^r
			\Bigr),
			\qquad r=2,3.
\end{aligned}
\label{eq:spanwise-frequency-splitting}
\end{equation}
		If $\langle k,\eta,\ell\rangle\geq m_\nu/(2\nu)$, the first term on the right-hand side dominates the left-hand side. Suppose instead that $\langle k,\eta,\ell\rangle<m_\nu/(2\nu)$ and use $0<\sqrt{k^2+\ell^2}/\langle k,\eta,\ell\rangle\leq1$. Young's inequality gives
		\begin{equation*}
			m_\nu\sqrt{k^2+\ell^2}
			\leq \frac12\frac{k^2+\ell^2}{\langle k,\eta,\ell\rangle^2}+C m_\nu^2\langle k,\eta,\ell\rangle^2.
		\end{equation*}
		After multiplication by $\langle k,\eta,\ell\rangle^r$, the first term becomes $\frac12(k^2+\ell^2)\langle k,\eta,\ell\rangle^{r-2}$. Moreover,
		\begin{equation*}
			m_\nu^2\langle k,\eta,\ell\rangle^2\leq C m_\nu^4\nu^{-2}.
		\end{equation*}
		The definition of $m_\nu$ implies
		\begin{equation*}
			\sup_{t\geq0}m_\nu(t)^4\ee^{-\gamma\nu^{1/3}t}
			\leq C_\gamma\nu^{8/3}.
		\end{equation*}
		Indeed, for $t\leq\nu^{-1}$ one has $m_\nu(t)\leq C\nu^{2/3}(1+\nu^{1/3}t)$, while for $t\geq\nu^{-1}$ the exponential is smaller than $\exp(-\gamma\nu^{-2/3})$. Hence
		\begin{equation*}
			m_\nu^2\langle k,\eta,\ell\rangle^2\leq C_\gamma\nu^{2/3}\ee^{\gamma \nu^{1/3}t},
		\end{equation*}
		which proves \eqref{eq:spanwise-frequency-splitting}.
		
		For $r=2$, multiply \eqref{eq:spanwise-frequency-splitting} by $|\widehat{u^3_{\neq}}|$. We use
		\begin{equation*}
			(A+B+C)^2\leq3(A^2+B^2+C^2)
		\end{equation*}
		and then integrate in $(k,\eta,\ell)$. Since $k\neq0$, the inhomogeneous multiplier satisfies
		\begin{equation}
			\langle k,\eta,\ell\rangle\simeq (k^2+\eta^2+\ell^2)^{1/2},
			\label{eq:Lambda-equivalence-nonzero}
		\end{equation}
		with constants independent of the continuous spanwise frequency $\ell$. Hence Plancherel identifies the three terms on the right-hand side with, respectively,
		\begin{equation*}
			\nu\|(\partial_x,\partial_z)u^3_{\neq}\|_{H^3},
			\qquad
			\|-(\partial_x^2+\partial_z^2)u^3_{\neq}\|_{L^2},
			\qquad
			\nu^{2/3}\ee^{\gamma\nu^{1/3}t}\|u^3_{\neq}\|_{H^2}.
		\end{equation*}
		Indeed, the multipliers $\sqrt{k^2+\ell^2}\langle k,\eta,\ell\rangle^3$, $(k^2+\ell^2)$, and $\langle k,\eta,\ell\rangle^2$ are precisely the Fourier representatives of these three quantities up to the equivalence \eqref{eq:Lambda-equivalence-nonzero}. Since $(u^3_{\neq})_{0}=0$, $\|u^3_{\neq}\|_{H^2}\leq C\|\Delta u^3_{\neq}\|_{L^2}$. Multiplication by $\ee^{2\gamma\nu^{1/3}t}$ and taking the time supremum therefore gives the first component of \eqref{eq:spanwise-high-profile} from the definitions of $E_{5}$, $E_{4}$, and $E_{6}$.
		
		For $r=3$, the same Plancherel argument identifies the three right-hand multipliers with
		\begin{equation}
			\nu\|(\partial_x,\partial_z)u^3_{\neq}\|_{H^4},
			\qquad
			\|-\nabla(\partial_x^2+\partial_z^2)u^3_{\neq}\|_{L^2},
			\qquad
			\nu^{2/3}\ee^{\gamma\nu^{1/3}t}\|u^3_{\neq}\|_{H^3}.
			\label{eq:spanwise-r3-identification}
		\end{equation}
		Here the middle multiplier is $(k^2+\ell^2)\langle k,\eta,\ell\rangle$, and the last one is $\langle k,\eta,\ell\rangle^3$. The last term in \eqref{eq:spanwise-r3-identification} is bounded by $C\nu^{2/3}\ee^{\gamma\nu^{1/3}t}\|\nabla\Delta u^3_{\neq}\|_{L^2}$ on the streamwise nonzero subspace. Multiplying by $\nu^{1/2}\ee^{2\gamma\nu^{1/3}t}$ and integrating in time proves the second component of \eqref{eq:spanwise-high-profile}. Finally, since $\sqrt{k^2+\ell^2}\leq(k^2+\ell^2)$ and $\sqrt{k^2+\ell^2}\langle k,\eta,\ell\rangle\leq(k^2+\ell^2)\langle k,\eta,\ell\rangle$, \eqref{eq:spanwise-low-profile} follows immediately from \eqref{eq:P-profiles-revised}.
	\end{proof}
	
	\begin{Lem}
		\label{lem:streak-streamwise-profile}
		One has
		\begin{align}
			&\|\ee^{3\gamma\nu^{1/3}t}\partial_xu_{\neq}\|_{L_t^\infty L^2}
			+\nu^{1/2}\|\ee^{3\gamma\nu^{1/3}t}\partial_xu_{\neq}\|_{L_t^2H^1}
			\leq C\bigl(E_{4}(T)+E_{6}(T)\bigr),
			\label{eq:streak-low-x-profile}\\
			&\|\ee^{2\gamma\nu^{1/3}t}m_\nu(t)\partial_xu_{\neq}\|_{L_t^\infty H^2}
			+\nu^{1/2}\|\ee^{2\gamma\nu^{1/3}t}m_\nu(t)\partial_xu_{\neq}\|_{L_t^2H^3}
			\nonumber\\
			&\hspace{35mm}\leq C\nu E_{5}(T)+C\bigl(E_{4}(T)+E_{6}(T)\bigr).
			\label{eq:streak-high-x-profile}
		\end{align}
	\end{Lem}
	
	\begin{proof}
		We first prove the low-order estimate component by component. For $j=2,3$, the definitions of $\mathcal U_{0},\mathcal U_1$, the streamwise Poincar\'e inequality, and \eqref{eq:U-profiles-revised} give
		\begin{align}
			&\|\ee^{3\gamma\nu^{1/3}t}\partial_xu^j_{\neq}\|_{L_t^\infty L^2}
			+\nu^{1/2}\|\ee^{3\gamma\nu^{1/3}t}\partial_xu^j_{\neq}\|_{L_t^2H^1}
			\leq C E_{6}.
			\label{eq:streak-low-x-W23}
		\end{align}
		For the streamwise component, \Cref{cor:w1-recovery} at the weight $3\gamma$ yields
		\begin{equation*}
			\|u^1_{\neq}\|_{X_{3\gamma}}^2+\|\partial_xu^1_{\neq}\|_{X_{3\gamma}}^2\leq C E_{6}^2.
		\end{equation*}
		Hence the $L_t^\infty L^2$ part of $\partial_xu^1_{\neq}$ is bounded by $CE_{6}$, while the dissipative component of $X_{3\gamma}$ gives
		\begin{equation*}
			\nu^{1/2}\|\ee^{3\gamma\nu^{1/3}t}\nabla\partial_xu^1_{\neq}\|_{L_t^2L^2}\leq C E_{6}.
		\end{equation*}
		Since $(\partial_xu^1_{\neq})_{0}=0$, streamwise Poincar\'e also gives
		\begin{equation*}
			\|\partial_xu^1_{\neq}\|_{L^2}\leq\|\partial_x^2u^1_{\neq}\|_{L^2}\leq\|\nabla\partial_xu^1_{\neq}\|_{L^2},
		\end{equation*}
		so the full $L_t^2H^1$ part is controlled by $CE_{6}$. Combining this with \eqref{eq:streak-low-x-W23} proves \eqref{eq:streak-low-x-profile} (with the harmless stated right-hand side $C(E_{4}+E_{6})$).
		
		For the high-order estimate, apply \eqref{eq:mnu-interpolation-revised} with $F=\partial_xu_{\neq}$ and $r=2,3$. The one-higher-derivative terms are exactly the two components of $\nu E_{5}$. The lower-order terms are controlled by \eqref{eq:streak-low-x-profile}; the factor $\nu(1+t)^3$ is absorbed by the gap between the $2\gamma$ and $3\gamma$ exponential weights through \eqref{eq:weight-gap-polynomial-revised}.
	\end{proof}
	
	\begin{Lem}
		\label{lem:streak-high-product-revised}
		It holds
		\begin{align}
			&\|\ee^{2\gamma\nu^{1/3}t}
			(u^1_{0}\partial_xu_{\neq}+u^\perp_{\neq}\cdot\nabla_\perp u^1_{0}\,e_1)\|_{Y^2}
			\nonumber\\
			&\qquad\leq C\nu\bigl(E_{1}(T)+\frac{\varepsilon}{\nu}\bigr)E_{5}(T)
			+C\bigl(E_{1}(T)+\frac{\varepsilon}{\nu}\bigr)\bigl(E_{4}(T)+E_{6}(T)\bigr).
			\label{eq:streak-high-product-revised}
		\end{align}
	\end{Lem}
	
	\begin{proof}
For the streak decomposition identified above, the high and low
wave norms are supplied by
\Cref{lem:streak-spanwise-profile,lem:streak-streamwise-profile}.

For the transport by $u^1_{0}$, the required spatial estimates are
\begin{align*}
 \|u^1_{0}\partial_xu_{\neq}\|_{H^2}
 &\leq C\|u^1_{0}\|_{H^2}\|\partial_xu_{\neq}\|_{H^2}
 +C\|u^1_{0}\|_{H^4}\|\partial_xu_{\neq}\|_{L^2},\\
 \|u^1_{0}\partial_xu_{\neq}\|_{H^3}
 &\leq C\|u^1_{0}\|_{H^2}\|\partial_xu_{\neq}\|_{H^3}
 +C\|u^1_{0}\|_{H^4}\|\partial_xu_{\neq}\|_{H^1}.
\end{align*}
To verify these on $\R^2$, assign zero or one derivative of the
coefficient to its $H^2$ norm, and two or three to its $H^4$ norm.
The mixed terms use $H^1_{y,z}\hookrightarrow L^4_{y,z}$,
also for $L_x^2$-valued functions; the remaining terms use
$H^2_{y,z}\hookrightarrow L^\infty_{y,z}$.
Since $\|u^1_{0}\|_{H^2}\leq
C(E_1+\varepsilon/\nu)m_\nu$ and $\|u^1_{0}\|_{H^4}\leq E_1$,
\Cref{lem:streak-streamwise-profile} gives
\begin{equation}
 \|\ee^{2\gamma\nu^{1/3}t}u^1_{0}\partial_xu_{\neq}\|_{Y^2}
 \leq C\nu(E_1+\varepsilon/\nu)E_5
 +C(E_1+\varepsilon/\nu)(E_4+E_6).
 \label{eq:aWx-Y2-direct}
\end{equation}

For the other interaction, write
$u^\perp_{\neq}\cdot\nabla_\perp u^1_{0}
=u^2_{\neq}\partial_yu^1_{0}+u^3_{\neq}\partial_zu^1_{0}$.
The Sobolev product estimate with the coefficient in
$L_t^\infty H^3$ gives
\begin{equation*}
 \|u^2_{\neq}\partial_yu^1_{0}\|_{H^r}
 \leq C\|u^1_{0}\|_{H^4}\|u^2_{\neq}\|_{H^r},\qquad r=2,3.
\end{equation*}
Together with $\|\ee^{2\gamma\nu^{1/3}t}u^2_{\neq}\|_{Y^2}\leq CE_4$,
this gives
\begin{equation}
 \|\ee^{2\gamma\nu^{1/3}t}u^2_{\neq}\partial_yu^1_{0}\|_{Y^2}
 \leq CE_1E_4.
 \label{eq:ayW2-Y2-direct}
\end{equation}
For $u^3_{\neq}\partial_zu^1_{0}$,
\Cref{lem:streak-physical-tame} gives, for $r=2,3$,
\begin{align*}
 \|u^3_{\neq}\partial_zu^1_{0}\|_{H^r}
 \leq{}&C(E_1+\varepsilon/\nu)m_\nu
 \|(\partial_x,\partial_z)u^3_{\neq}\|_{H^r}\\
 &+CE_1\|(\partial_x,\partial_z)u^3_{\neq}\|_{H^{r-2}}.
\end{align*}
Apply \Cref{lem:streak-spanwise-profile} separately in
$L_t^\infty H^2$ and $\nu^{1/2}L_t^2H^3$. This gives
\begin{equation}
 \|\ee^{2\gamma\nu^{1/3}t}u^3_{\neq}\partial_zu^1_{0}\|_{Y^2}
 \leq C\nu(E_1+\varepsilon/\nu)E_5
 +C(E_1+\varepsilon/\nu)(E_4+E_6).
 \label{eq:azW3-Y2-direct}
\end{equation}
The weak pure-$y$ derivatives are handled by the frequency splitting
in that profile lemma. Summing the three estimates proves the claim.
\end{proof}

The $E_{4}$ estimate uses the constant Couette system with the streak contribution treated as forcing. The following identities and bounds provide this final interaction input.
With the pressures $p^C$ and $p^{\rm st}$ fixed in Section~\ref{sec:full-system}, the equations for the primary variables are
	\begin{align}
		L_{0}\Delta u^2_{\neq}&=-\Delta\bigl(\partial_yp^{\rm st}+u^1_{0}\partial_xu^2_{\neq}+g^2\bigr),
		\label{eq:primary-W2-final}\\
		L_{0}u^3_{\neq}-2\partial_x\partial_z\Delta^{-1}u^2_{\neq}
		&=-\partial_zp^{\rm st}-u^1_{0}\partial_xu^3_{\neq}-g^3.
		\label{eq:primary-W3-final}
	\end{align}
	The cancellation in \eqref{eq:primary-W2-final} follows from
	$[\Delta,y\partial_x]u^2_{\neq}=2\partial_x\partial_yu^2_{\neq}$ and
	$\Delta\partial_yp^C=-2\partial_x\partial_yu^2_{\neq}$.
	
	\begin{Lem}
		\label{lem:streak-primary-revised}
		Let
		\begin{align}
			\mathfrak S_{\rm st}(T):={}&\nu^{-1/2}\|\ee^{2\gamma\nu^{1/3}t}\nabla(\partial_yp^{\rm st}+u^1_{0}\partial_xu^2_{\neq})\|_{L_t^2L_x^2}
			\nonumber\\
			&+\nu^{-1/2}\|\ee^{2\gamma\nu^{1/3}t}(\partial_x,\partial_z)(\partial_zp^{\rm st}+u^1_{0}\partial_xu^3_{\neq})\|_{L_t^2L_x^2}.
			\notag
		\end{align}
		Then
		\begin{equation}
			\mathfrak S_{\rm st}(T)
			\leq C\bigl(E_{1}(T)+\frac{\varepsilon}{\nu}\bigr)
			\left[E_{6}(T)+\bigl(E_{4}(T)E_{6}(T)\bigr)^{1/2}\right].
			\label{eq:Sa-bound-revised}
		\end{equation}
	\end{Lem}
	
	\begin{proof}
In the streak-pressure argument identified above, $p^{\rm st}$
corresponds to $p^{(1)}$, since
\begin{equation*}
 \Delta p^{\rm st}=-2\bigl(
 \partial_yu^1_{0}\partial_xu^2_{\neq}
 +\partial_zu^1_{0}\partial_xu^3_{\neq}\bigr).
\end{equation*}
The continuous-domain product \eqref{eq:anis-prod-L2} and the
order-zero pressure multipliers give
\begin{equation*}
 \|\nabla\partial_yp^{\rm st}\|_{L^2}
 +\|(\partial_x,\partial_z)\partial_zp^{\rm st}\|_{L^2}
 \leq C\|u^1_{0}\|_{H^2}\mathcal U_0.
\end{equation*}
The transport terms in the same argument use
\begin{align*}
 \|\nabla\partial_xu^2_{\neq}\|_{L^2}^2
 &\leq\|\partial_x^2u^2_{\neq}\|_{L^2}\|\Delta u^2_{\neq}\|_{L^2},\\
 \|(\partial_x,\partial_z)\partial_xu^3_{\neq}\|_{L^2}^2
 &\leq\|\partial_x^2u^3_{\neq}\|_{L^2}
 \|-(\partial_x^2+\partial_z^2)u^3_{\neq}\|_{L^2}.
\end{align*}
Both are bounded by $C\mathcal U_0\Pi_0$. If a derivative falls on
the coefficient, \eqref{eq:anis-prod-L2} instead gives
$C\|u^1_{0}\|_{H^2}\mathcal U_0$.
Thus the sum of the two spatial forcing norms defining
$\mathfrak S_{\rm st}$ is bounded by
\begin{equation}
 C\|u^1_{0}(t)\|_{H^2}
 \bigl(\mathcal U_0(t)+(\mathcal U_0(t)\Pi_0(t))^{1/2}\bigr).
 \label{eq:Sa-point-revised}
\end{equation}
This retains the lower-order term needed in our norms.

For the time integration, \Cref{lem:transient-streak-revised},
\eqref{eq:mnu-first-power-revised}, and \eqref{eq:U-profiles-revised}
bound the first term by
\begin{equation*}
 C(E_1+\varepsilon/\nu)\nu^{-1/2}\nu^{2/3}\nu^{-1/6}E_6
 =C(E_1+\varepsilon/\nu)E_6.
\end{equation*}
For the square of the mixed term, use
\eqref{eq:mnu-exponential-revised} and Cauchy--Schwarz in time:
\begin{align*}
 &\nu^{-1}\int_0^T\ee^{4\gamma\nu^{1/3}t}
 m_\nu(t)^2\mathcal U_0(t)\Pi_0(t)\,dt\\
 &\quad\leq C\nu^{-1}\nu^{4/3}
 \|\ee^{3\gamma\nu^{1/3}t}\mathcal U_0\|_{L_t^2}
 \|\ee^{2\gamma\nu^{1/3}t}\Pi_0\|_{L_t^2}
 \leq CE_4E_6.
\end{align*}
Taking the square root and restoring $E_1+\varepsilon/\nu$ proves
\eqref{eq:Sa-bound-revised}.
\end{proof}

\section{Energy estimates}
\label{sec:energy-estimates}
We will prove the six closure propositions of Section~\ref{sec:full-system} using the estimates of Section~\ref{sec:zero-endpoint} and the nonlinear bounds of Section~\ref{sec:interactions-final}.

\subsection{Zero-mode energy estimates}
\label{sec:zero-closure-final}
We first control the roll energies $E_{2}$ and $E_{3}$ and then the streak energy $E_{1}$.

% The velocity and vorticity equations provide the dissipative terms in $E_{2}$; Besov propagation and Fourier splitting supply its low-frequency components.
	
Since $u_{\neq}$ is divergence free and every streamwise derivative has zero streamwise average,
	\begin{equation*}
		(u_{\neq}\cdot\nabla u^\perp_{\neq})_{0}=\operatorname{div}_\perp (u^\perp_{\neq}\otimes u^\perp_{\neq})_{0}.
	\end{equation*}
	Recall the roll and roll-vorticity equations:
	\begin{align}
		\partial_tu^\perp_{0}-\nu\Delta_\perp u^\perp_{0}+u^\perp_{0}\cdot\nabla_\perp u^\perp_{0}+\nabla_\perp p_{0}
		&=-(u_{\neq}\cdot\nabla u^\perp_{\neq})_{0},
		\label{eq:roll-NW-revised}\\
		\partial_t\omega_{0}-\nu\Delta_\perp\omega_{0}+u^\perp_{0}\cdot\nabla_\perp\omega_{0}
		&=\bigl(-\partial_y(u_{\neq}\cdot\nabla u^3_{\neq})_{0}+\partial_z(u_{\neq}\cdot\nabla u^2_{\neq})_{0}\bigr).
		\label{eq:roll-vorticity-final}
	\end{align}

\begin{proof}[Proof of \Cref{prop:ZR-closed-final}]
		We divide the proof into four steps.
		
		\emph{Step 1: the basic velocity and vorticity energies.}
		Taking the $L^2$ scalar product of \eqref{eq:roll-NW-revised} with $u^\perp_{0}$ and using $\operatorname{div}_\perp u^\perp_{0}=0$ gives
		\begin{equation*}
			\frac12\frac{d}{dt}\|u^\perp_{0}\|_{L^2}^2+\nu\|\nabla_\perp u^\perp_{0}\|_{L^2}^2
			=\int_{\mathbb R^2}(u^\perp_{\neq}\otimes u^\perp_{\neq})_{0}:\nabla_\perp u^\perp_{0}.
		\end{equation*}
		Thus
		\begin{equation}
			\|u^\perp_{0}\|_{Y^0}
			\leq C\|u^\perp_{0}(0)\|_{L^2}+C\nu^{-1/2}\|(u^\perp_{\neq}\otimes u^\perp_{\neq})_{0}\|_{L_t^2L_x^2}.
			\label{eq:R-Y0-roll-revised}
		\end{equation}
		For \eqref{eq:roll-vorticity-final}, integration by parts gives
		\begin{equation}
			\langle\bigl(-\partial_y(u_{\neq}\cdot\nabla u^3_{\neq})_{0}+\partial_z(u_{\neq}\cdot\nabla u^2_{\neq})_{0}\bigr),\omega_{0}\rangle
			=\langle (u_{\neq}\cdot\nabla u^3_{\neq})_{0},\partial_y\omega_{0}\rangle
			-\langle (u_{\neq}\cdot\nabla u^2_{\neq})_{0},\partial_z\omega_{0}\rangle.
		\end{equation}
		The transport term cancels, and therefore
		\begin{equation}
			\|\omega_{0}\|_{Y^0}
			\leq C\|\omega_{0}(0)\|_{L^2}
			+C\nu^{-1/2}\|(u_{\neq}\cdot\nabla u^\perp_{\neq})_{0}\|_{L_t^2L_x^2}.
			\label{eq:omega-Y0-roll-revised}
		\end{equation}
		By \Cref{lem:wave-roll-forcing-revised}, the right-hand sides of \eqref{eq:R-Y0-roll-revised}--\eqref{eq:omega-Y0-roll-revised} are bounded by $C\|u^\perp_{0}(0)\|_{H^2}+C\nu^{-1}E_{4}^2$.
		
		\emph{Step 2: the unweighted $z$ derivative.}
		
		Note
		\begin{equation*}
			\partial_t\partial_z\omega_{0}-\nu\Delta_\perp \partial_z\omega_{0}+u^\perp_{0}\cdot\nabla_\perp \partial_z\omega_{0}
			=-(\partial_zu^\perp_{0})\cdot\nabla_\perp\omega_{0}+\operatorname{div}_\perp \bigl(-\partial_z(u_{\neq}\cdot\nabla u^3_{\neq})_{0},\,\partial_z(u_{\neq}\cdot\nabla u^2_{\neq})_{0}\bigr).
		\end{equation*}
		By $\operatorname{div}_\perp(\partial_zu^\perp_{0})=0$, we have
		\begin{equation*}
			\int(\partial_zu^\perp_{0}\cdot\nabla_\perp\omega_{0})\partial_z\omega_{0}
			=-\int(\partial_zu^\perp_{0}\cdot\nabla_\perp \partial_z\omega_{0})\omega_{0}.
		\end{equation*}
		The map $\omega_{0}\mapsto\partial_zu^\perp_{0}$ is a matrix of order-zero Riesz transforms.  Hence the two-dimensional Gagliardo--Nirenberg inequality gives
		\begin{align}
			\left|\int(\partial_zu^\perp_{0}\cdot\nabla_\perp \partial_z\omega_{0})\omega_{0}\right|
			&\leq C\|\omega_{0}\|_{L^4}^2\|\nabla_\perp \partial_z\omega_{0}\|_{L^2}
			\nonumber\\
			&\leq C\|\omega_{0}\|_{L^2}\|\nabla_\perp\omega_{0}\|_{L^2}\|\nabla_\perp \partial_z\omega_{0}\|_{L^2}
			\nonumber\\
			&\leq \frac{\nu}{8}\|\nabla_\perp \partial_z\omega_{0}\|_{L^2}^2
			+C\nu^{-1}\|\omega_{0}\|_{L^2}^2\|\nabla_\perp\omega_{0}\|_{L^2}^2.
			\notag
		\end{align}
		Under $\|\omega_{0}\|_{L_t^\infty L^2}\leq E_{2}\leq\delta_*\nu$, the last term is bounded by $C\delta_*^2\nu\|\nabla_\perp\omega_{0}\|_{L^2}^2$ and is controlled, after integration, by the already established basic vorticity dissipation; it is perturbative once $\delta_*$ is sufficiently small.  Integrating the source by parts and applying Young's inequality yields
		\begin{equation*}
			\|\partial_z\omega_{0}\|_{Y^0}
			\leq C\|\partial_z\omega_{0}(0)\|_{L^2}+C\delta_*\|\omega_{0}\|_{Y^0}
			+C\nu^{-1/2}\|\bigl(-\partial_z(u_{\neq}\cdot\nabla u^3_{\neq})_{0},\,\partial_z(u_{\neq}\cdot\nabla u^2_{\neq})_{0}\bigr)\|_{L_t^2L_x^2}.
		\end{equation*}
		By \eqref{eq:wave-roll-forcing-revised}, we have $\nu^{-1/2}\|\bigl(-\partial_z(u_{\neq}\cdot\nabla u^3_{\neq})_{0},\,\partial_z(u_{\neq}\cdot\nabla u^2_{\neq})_{0}\bigr)\|_{L_t^2L_x^2}\leq C\nu^{-1}E_{4}^2$.
		
		\emph{Step 3: the weighted $y$ derivative.}
		
		Multiplying the differentiated equation by $\mu_\nu \partial_y\omega_{0}$ gives the exact identity
		\begin{align}
			\frac12\frac{d}{dt}\bigl(\mu_\nu\|\partial_y\omega_{0}\|_{L^2}^2\bigr)
			+\nu\mu_\nu\|\nabla_\perp \partial_y\omega_{0}\|_{L^2}^2
			={}&\frac12\mu_\nu'\|\partial_y\omega_{0}\|_{L^2}^2
			-\mu_\nu\int(\partial_yu^\perp_{0}\cdot\nabla_\perp\omega_{0})\partial_y\omega_{0}
			\nonumber\\
			&+\mu_\nu\langle\operatorname{div}_\perp \bigl(-\partial_y(u_{\neq}\cdot\nabla u^3_{\neq})_{0},\,\partial_y(u_{\neq}\cdot\nabla u^2_{\neq})_{0}\bigr),\partial_y\omega_{0}\rangle.
			\notag
		\end{align}
		The positive weight derivative is retained. By $0\leq\mu_\nu'\leq\nu$, we have
		\begin{equation*}
			\int_{0}^T\mu_\nu'\|\partial_y\omega_{0}\|_{L^2}^2\,dt
			\leq\nu\int_{0}^T\|\nabla_\perp\omega_{0}\|_{L^2}^2\,dt,
		\end{equation*}
		which is controlled by \eqref{eq:omega-Y0-roll-revised}.  As in Step 2,
		\begin{align}
			\mu_\nu\left|\int(\partial_yu^\perp_{0}\cdot\nabla_\perp\omega_{0})\partial_y\omega_{0}\right|
			&=\mu_\nu\left|\int(\partial_yu^\perp_{0}\cdot\nabla_\perp \partial_y\omega_{0})\omega_{0}\right|
			\nonumber\\
			&\leq\frac{\nu\mu_\nu}{8}\|\nabla_\perp \partial_y\omega_{0}\|_{L^2}^2
			+C\nu^{-1}\mu_\nu\|\omega_{0}\|_{L^2}^2\|\nabla_\perp\omega_{0}\|_{L^2}^2.\notag
		\end{align}
		The last term is controlled by the already established basic vorticity dissipation and is perturbative for sufficiently small $\delta_*$. The source satisfies
		\begin{equation*}
\begin{aligned}
&\mu_\nu|\langle\operatorname{div}_\perp \bigl(-\partial_y(u_{\neq}\cdot\nabla u^3_{\neq})_{0},\,\partial_y(u_{\neq}\cdot\nabla u^2_{\neq})_{0}\bigr),\partial_y\omega_{0}\rangle|
			\\&\leq\frac{\nu\mu_\nu}{8}\|\nabla_\perp \partial_y\omega_{0}\|_{L^2}^2
			+C\nu^{-1}\mu_\nu\|\bigl(-\partial_y(u_{\neq}\cdot\nabla u^3_{\neq})_{0},\,\partial_y(u_{\neq}\cdot\nabla u^2_{\neq})_{0}\bigr)\|_{L^2}^2.
\end{aligned}
\end{equation*}
		Consequently,
		\begin{equation*}
			\|\mu_\nu^{1/2}\partial_y\omega_{0}\|_{Y^0}
			\leq C\nu^{1/3}\|\partial_y\omega_{0}(0)\|_{L^2}+C\|\omega_{0}\|_{Y^0}
			+C\nu^{-1/2}\|\mu_\nu^{1/2}\bigl(-\partial_y(u_{\neq}\cdot\nabla u^3_{\neq})_{0},\,\partial_y(u_{\neq}\cdot\nabla u^2_{\neq})_{0}\bigr)\|_{L_t^2L_x^2}.
		\end{equation*}
		The initial factor is $\mu_\nu(0)^{1/2}=\nu^{1/3}$.  For the two components of $\bigl(-\partial_y(u_{\neq}\cdot\nabla u^3_{\neq})_{0},\,\partial_y(u_{\neq}\cdot\nabla u^2_{\neq})_{0}\bigr)$, \eqref{eq:weighted-NW2-revised} gives
		\begin{equation*}
			\nu^{-1/2}\|\mu_\nu^{1/2}\partial_y(u_{\neq}\cdot\nabla u^2_{\neq})_{0}\|_{L_t^2L_x^2}
			\leq C\nu^{-2/3}E_{4}^2
			\leq C\nu^{-1}E_{4}^2,
		\end{equation*}
		where we used $0<\nu\leq1$, while \eqref{eq:weighted-NW3-revised} gives the same required $C\nu^{-1}E_{4}^2$ bound for the first component. Hence
		\begin{equation*}
			\nu^{-1/2}\|\mu_\nu^{1/2}\bigl(-\partial_y(u_{\neq}\cdot\nabla u^3_{\neq})_{0},\,\partial_y(u_{\neq}\cdot\nabla u^2_{\neq})_{0}\bigr)\|_{L_t^2L_x^2}
			\leq C\nu^{-1}E_{4}^2.
		\end{equation*}
		Combining Steps 1--3 gives
		\begin{align}
			&\|u^\perp_{0}\|_{Y^0}+\|\omega_{0}\|_{Y^0}+\|\partial_z\omega_{0}\|_{Y^0}
			+\|\mu_\nu^{1/2}\partial_y\omega_{0}\|_{Y^0}
			\leq C\|u^\perp_{0}(0)\|_{H^2}+C\nu^{-1}E_{4}^2.
			\label{eq:roll-high-closed-revised}
		\end{align}
		
		\emph{Step 4: the critical low-frequency pieces.}

		We first record explicitly the $L_t^1L^1$ bound for the wave stress. Jensen's inequality in $x$ and Fubini's theorem give, for every $t$,
		\begin{align}
			\|\bigl(u^\perp_{\neq}(t)\otimes u^\perp_{\neq}(t)\bigr)_{0}\|_{L^1(\mathbb R^2)}
			&\leq C\int_{\mathbb R^2}\int_{\mathbb T}|u_{\neq}(t,x,y,z)|^2\,dx\,dy\,dz
			\nonumber\\
			&\leq C\|u_{\neq}(t)\|_{L^2}^2
			\leq C\Pi_{0}(t)^2.
			\notag
		\end{align}
		Hence \eqref{eq:P-profiles-revised} yields
		\begin{equation}
			\int_{0}^T\|\bigl(u^\perp_{\neq}(t)\otimes u^\perp_{\neq}(t)\bigr)_{0}\|_{L^1}\,dt
			\leq C\int_{0}^T\Pi_{0}(t)^2\,dt
			\leq C\nu^{-1/3}E_{4}(T)^2.
			\label{eq:T-L1-time-explicit}
		\end{equation}
		The Besov propagation inequality \eqref{eq:B-prop}, together with \eqref{eq:T-L1-time-explicit}, therefore gives
		\begin{equation}
			\sup_{0\leq s\leq T}\norm{u^\perp_{0}(s)}{\dot B^{-2}_{2,\infty}}\leq\|u^\perp_{0}(0)\|_{\dot B^{-2}_{2,\infty}}+C\int_{0}^T\|u^\perp_{0}(t)\|_{L^2}^2\,dt+C\nu^{-1/3}E_{4}^2.
			\label{eq:BT-final-roll-revised}
		\end{equation}
		Furthermore, \eqref{eq:quadratic-time-final} and the exponential weight give
		\begin{equation*}
			\int_{0}^T(1+\nu t)\|\bigl(u^\perp_{\neq}(t)\otimes u^\perp_{\neq}(t)\bigr)_{0}\|_{L^2}^2\,dt
			\leq C\nu^{-1}E_{4}^4.
		\end{equation*}
		Indeed, $(1+\nu t)\ee^{-8\gamma\nu^{1/3}t}$ is uniformly bounded.  Hence \Cref{lem:fourier-splitting} yields
		\begin{equation}
			\int_{0}^T\|u^\perp_{0}(t)\|_{L^2}^2\,dt\leq\frac{C}{\nu}(\|u^\perp_{0}(0)\|_{L^2}^2+\left(\sup_{0\leq s\leq T}\norm{u^\perp_{0}(s)}{\dot B^{-2}_{2,\infty}}\right)^2)+C\nu^{-3}E_{4}^4.
			\label{eq:YT-final-roll-revised}
		\end{equation}

		Then $\left(\|u^\perp_{0}(0)\|_{H^2}+\|u^\perp_{0}(0)\|_{\dot B^{-2}_{2,\infty}}+\nu^{-1}E_{4}(T)^2\right)/\nu\leq C\delta_*$ under the assumptions of the proposition, and
		$\nu^{-1/3}E_{4}^2\leq \left(\|u^\perp_{0}(0)\|_{H^2}+\|u^\perp_{0}(0)\|_{\dot B^{-2}_{2,\infty}}+\nu^{-1}E_{4}(T)^2\right)$.  We close a bootstrap on the continuous, nondecreasing quantity $\int_{0}^T\|u^\perp_{0}(t)\|_{L^2}^2\,dt$.  Suppose
		\begin{equation}
			\int_{0}^T\|u^\perp_{0}(t)\|_{L^2}^2\,dt\leq K\nu^{-1}\left(\|u^\perp_{0}(0)\|_{H^2}+\|u^\perp_{0}(0)\|_{\dot B^{-2}_{2,\infty}}+\nu^{-1}E_{4}(T)^2\right)^2.
			\label{eq:YT-pair-bootstrap-revised}
		\end{equation}
		Equation \eqref{eq:BT-final-roll-revised} gives
		\begin{equation}
\begin{aligned}
\sup_{0\leq s\leq T}\norm{u^\perp_{0}(s)}{\dot B^{-2}_{2,\infty}}&\leq\left(1+CK\frac{\left(\|u^\perp_{0}(0)\|_{H^2}+\|u^\perp_{0}(0)\|_{\dot B^{-2}_{2,\infty}}+\nu^{-1}E_{4}(T)^2\right)}{\nu}\right)\\&\quad\times\left(\|u^\perp_{0}(0)\|_{H^2}+\|u^\perp_{0}(0)\|_{\dot B^{-2}_{2,\infty}}+\nu^{-1}E_{4}(T)^2\right)\\&\leq2\left(\|u^\perp_{0}(0)\|_{H^2}+\|u^\perp_{0}(0)\|_{\dot B^{-2}_{2,\infty}}+\nu^{-1}E_{4}(T)^2\right)
\end{aligned}
\label{eq:BT-pair-improve-revised}
\end{equation}
		provided $\delta_*$ is sufficiently small after $K$ is fixed.  Substitution in \eqref{eq:YT-final-roll-revised} yields
		\begin{equation*}
\begin{aligned}
\int_{0}^T\|u^\perp_{0}(t)\|_{L^2}^2\,dt
			&\leq C\nu^{-1}\Bigl(\|u^\perp_{0}(0)\|_{L^2}^2+4\left(\|u^\perp_{0}(0)\|_{H^2}+\|u^\perp_{0}(0)\|_{\dot B^{-2}_{2,\infty}}+\nu^{-1}E_{4}(T)^2\right)^2\Bigr)
			\\&\quad+C\nu^{-1}(\nu^{-1}E_{4}^2)^2
			\\&\leq C_1\nu^{-1}\left(\|u^\perp_{0}(0)\|_{H^2}+\|u^\perp_{0}(0)\|_{\dot B^{-2}_{2,\infty}}+\nu^{-1}E_{4}(T)^2\right)^2.
\end{aligned}
\end{equation*}
		Choose $K=2C_1+2$ and then reduce $\delta_*$.  This strictly improves \eqref{eq:YT-pair-bootstrap-revised}; continuity therefore gives the improved estimate for every $T$.  Together with \eqref{eq:BT-pair-improve-revised},
		\begin{equation}
			\sup_{0\leq s\leq T}\norm{u^\perp_{0}(s)}{\dot B^{-2}_{2,\infty}}+\nu^{1/2}\left(\int_{0}^T\|u^\perp_{0}(t)\|_{L^2}^2\,dt\right)^{1/2}
			\leq C\left(\|u^\perp_{0}(0)\|_{H^2}+\|u^\perp_{0}(0)\|_{\dot B^{-2}_{2,\infty}}\right)+C\nu^{-1}E_{4}^2.
			\label{eq:roll-low-closed-revised}
		\end{equation}
		
		Combining \eqref{eq:roll-high-closed-revised} and \eqref{eq:roll-low-closed-revised} proves \eqref{eq:ZR-closed-final}.
	\end{proof}

\begin{Lem}
		\label{lem:2d-high-roll-commutator}
		Let $u^\perp_{0}\in H^5(\mathbb R^2;\mathbb R^2)$ be divergence free. Then
		\begin{equation}
			\sum_{1\leq|\alpha|\leq4}
			\left|\left\langle[\partial^\alpha,u^\perp_{0}\cdot\nabla_\perp]u^\perp_{0},\partial^\alpha u^\perp_{0}\right\rangle\right|
			\leq C\|\nabla_\perp u^\perp_{0}\|_{L^2}\|u^\perp_{0}\|_{H^4}\|\nabla_\perp u^\perp_{0}\|_{H^4}.
			\label{eq:2d-high-roll-commutator}
		\end{equation}
	\end{Lem}
	
	\begin{Rem}
		\label{rem:high-roll-H5-ae}
		In the application, $u^\perp_{0}\in L^\infty(0,T;H^4)\cap L^2(0,T;H^5)$. Thus the hypothesis $u^\perp_{0}(t)\in H^5$ in \Cref{lem:2d-high-roll-commutator} is available for almost every time and only supplies the viscous factor $\|\nabla_\perp u^\perp_{0}\|_{H^4}$. The estimate is applied first to smooth approximants and then integrated in time; it imposes no $H^5$ condition on the initial data.
	\end{Rem}
	
	\begin{proof}
After the divergence-free transport term cancels, H\"older's inequality
bounds each Leibniz term with $m=|\alpha|$ and $r=|\beta|$ by
\begin{equation*}
 C\|\nabla_\perp^r u^\perp_{0}\|_{L^4}
   \|\nabla_\perp^{m-r+1}u^\perp_{0}\|_{L^4}
   \|\nabla_\perp^m u^\perp_{0}\|_{L^2},
 \qquad 1\leq r\leq m\leq4.
\end{equation*}
The two-dimensional homogeneous Gagliardo--Nirenberg inequality,
applied to $\nabla_\perp u^\perp_{0}$, gives
\begin{equation*}
 \|\nabla_\perp^r u^\perp_{0}\|_{L^4}
 \leq C_m\|\nabla_\perp u^\perp_{0}\|_{L^2}^{1-(r-1/2)/m}
             \|\nabla_\perp^{m+1}u^\perp_{0}\|_{L^2}^{(r-1/2)/m}.
\end{equation*}
The exponents for $r$ and $m-r+1$ add to one. Hence
\begin{align*}
 \|\nabla_\perp^r u^\perp_{0}\|_{L^4}
   \|\nabla_\perp^{m-r+1}u^\perp_{0}\|_{L^4}
   \|\nabla_\perp^m u^\perp_{0}\|_{L^2}
 &\leq C_m\|\nabla_\perp u^\perp_{0}\|_{L^2}
                 \|\nabla_\perp^{m+1}u^\perp_{0}\|_{L^2}
                 \|\nabla_\perp^m u^\perp_{0}\|_{L^2}\\
 &\leq C\|\nabla_\perp u^\perp_{0}\|_{L^2}
                \|u^\perp_{0}\|_{H^4}
                \|\nabla_\perp u^\perp_{0}\|_{H^4}.
\end{align*}
Summing over $\alpha,\beta$ proves \eqref{eq:2d-high-roll-commutator}.
\end{proof}
	
	\begin{proof}[Proof of \Cref{prop:ZR4-closed-final}]
		Apply $\partial^\alpha$ with $1\leq|\alpha|\leq4$ to \eqref{eq:roll-NW-revised}, take the $L^2$ product with $\partial^\alpha u^\perp_{0}$, and sum.  The undifferentiated $L^2$ component is kept from the already proved basic roll estimate contained in $E_{2}$.  For $|\alpha|\geq1$, choose an index $j$ with $\alpha_j\geq1$ and integrate one derivative in
		\begin{equation}
			\langle\partial^\alpha (u_{\neq}\cdot\nabla u^\perp_{\neq})_{0},\partial^\alpha u^\perp_{0}\rangle
			=-\langle\partial^{\alpha-e_j}(u_{\neq}\cdot\nabla u^\perp_{\neq})_{0},\partial_j\partial^\alpha u^\perp_{0}\rangle.
			\label{eq:ZR4-force-shift}
		\end{equation}
		The principal transport term cancels.  By \Cref{lem:2d-high-roll-commutator}, \eqref{eq:ZR4-force-shift}, and Young's inequality,
		\begin{align}
			\frac{d}{dt}\sum_{1\leq|\alpha|\leq4}\|\partial^\alpha u^\perp_{0}\|_{L^2}^2
			+\nu\sum_{1\leq|\alpha|\leq4}\|\nabla_\perp\partial^\alpha u^\perp_{0}\|_{L^2}^2
			\leq{}&C\nu^{-1}\|\nabla_\perp u^\perp_{0}\|_{L^2}^2\|u^\perp_{0}\|_{H^4}^2
			\nonumber\\
			&+C\nu^{-1}\|(u_{\neq}\cdot\nabla u^\perp_{\neq})_{0}\|_{H^3}^2.
			\notag
		\end{align}
		The undifferentiated component is estimated separately by \eqref{eq:R-Y0-roll-revised}. Thus the wave forcing requires only three derivatives after one derivative is shifted to the tested roll factor.  The basic roll energy gives
		\begin{equation*}
			\nu^{-1}\int_{0}^T\|\nabla_\perp u^\perp_{0}\|_{L^2}^2dt
			\leq C\frac{E_{2}(T)^2}{\nu^2}
			\leq C\delta_*^2.
		\end{equation*}
		Integrating the homogeneous differentiated inequality, applying Gronwall, and then adding \eqref{eq:R-Y0-roll-revised} for the undifferentiated component yield, after reducing $\delta_*$,
		\begin{equation}
			\begin{aligned}
 E_{3}(T)\leq{}&C\|u^\perp_{0}(0)\|_{H^4}
 +C\nu^{-1/2}\|(u^\perp_{\neq}\otimes u^\perp_{\neq})_{0}\|_{L_t^2L^2}
 \\
 &+C\nu^{-1/2}\|(u_{\neq}\cdot\nabla u^\perp_{\neq})_{0}\|_{L_t^2H^3}.
\end{aligned}
			\label{eq:ZR4-before-wave}
		\end{equation}
		The stress term retained from \eqref{eq:R-Y0-roll-revised} satisfies, by \eqref{eq:wave-roll-forcing-revised},
\begin{equation*}
 \nu^{-1/2}\|(u^\perp_{\neq}\otimes u^\perp_{\neq})_{0}\|_{L_t^2L^2}
 \leq C\nu^{-1}E_4(T)^2.
\end{equation*}
At this point one must keep the transverse low frequency of the zero mode separate.  Since $(u_{\neq}\cdot\nabla u^\perp_{\neq})_{0}$ is streamwise independent, the $Y^2$ estimate does not by itself recover its full inhomogeneous $L_t^2H^3$ norm.  Instead we use
		\begin{equation}
			\|f\|_{H^3(\mathbb R^2)}
			\leq C\Big(\|f\|_{L^2(\mathbb R^2)}+\|\nabla_\perp f\|_{H^2(\mathbb R^2)}\Big).
			\label{eq:H3-zero-split}
		\end{equation}
		The zeroth-order spacetime component is supplied by \eqref{eq:wave-roll-forcing-revised}, while the differentiated component is supplied by \eqref{eq:wave-Y2-final}.  More precisely,
		\begin{align}
			\nu^{1/2}\|(u_{\neq}\cdot\nabla u^\perp_{\neq})_{0}\|_{L_t^2L^2}
			&\leq C E_{4}(T)^2,
			\label{eq:NW-L2-lowfreq-ZR4}\\
			\nu^{1/2}\|\nabla_\perp (u_{\neq}\cdot\nabla u^\perp_{\neq})_{0}\|_{L_t^2H^2}
			&\leq C E_{4}(T)E_{5}(T).
			\label{eq:NW-H3-highfreq-ZR4}
		\end{align}
		Here the exponential weights in \eqref{eq:wave-roll-forcing-revised} and \eqref{eq:wave-Y2-final} are at least one and may simply be discarded.  Combining \eqref{eq:H3-zero-split}--\eqref{eq:NW-H3-highfreq-ZR4} gives
		\begin{equation*}
			\nu^{1/2}\|(u_{\neq}\cdot\nabla u^\perp_{\neq})_{0}\|_{L_t^2H^3}
			\leq C\Big(E_{4}(T)^2+E_{4}(T)E_{5}(T)\Big).
		\end{equation*}
		Equivalently, the $|\alpha|=1$ part of \eqref{eq:ZR4-force-shift} is the only place where the undifferentiated $(u_{\neq}\cdot\nabla u^\perp_{\neq})_{0}$ appears and is controlled by \eqref{eq:NW-L2-lowfreq-ZR4}; for $2\leq|\alpha|\leq4$, at least one derivative remains on $(u_{\neq}\cdot\nabla u^\perp_{\neq})_{0}$ and \eqref{eq:NW-H3-highfreq-ZR4} applies.  Substitution into \eqref{eq:ZR4-before-wave} proves \eqref{eq:ZR4-closed-final}.
	\end{proof}

The roll estimates and the two-level Duhamel identity control the streak amplitude. A homogeneous energy estimate then supplies the fourth derivatives and the time derivative in $E_{1}$. Recall the streak equation:
	\begin{equation}
		\partial_t u^1_{0}-\nu\Delta_\perp u^1_{0}=-u^2_{0}-\bigl(u^\perp_{0}\cdot\nabla_\perp u^1_{0}+(u_{\neq}\cdot\nabla u^1_{\neq})_{0}\bigr).
		\label{eq:a-Na-revised}
	\end{equation}
	
	\begin{Lem}
		\label{lem:Na-Y2-revised}
		Assume $E_{1}(T)\leq\delta_*$ and $E_{2}(T)\leq\delta_*\nu$.  Then
		\begin{equation}
			\|\bigl(u^\perp_{0}\cdot\nabla_\perp u^1_{0}+(u_{\neq}\cdot\nabla u^1_{\neq})_{0}\bigr)\|_{Y^2}
			\leq C E_{2}(T)\bigl(E_{1}(T)+\frac{\varepsilon}{\nu}\bigr)+C E_{4}(T)E_{5}(T).
			\label{eq:Na-Y2-revised}
		\end{equation}
		Moreover,
		\begin{equation}
			\|\bigl(u^\perp_{0}\cdot\nabla_\perp u^1_{0}+(u_{\neq}\cdot\nabla u^1_{\neq})_{0}\bigr)(t)\|_{H^2}
			\leq C E_{2}(T)\bigl(E_{1}(T)+\frac{\varepsilon}{\nu}\bigr)+C E_{4}(T)E_{5}(T).
			\label{eq:Na-point-revised}
		\end{equation}
	\end{Lem}
	
	\begin{proof}
		The wave term is controlled by \Cref{lem:wave-Y2-final}. We estimate the roll--streak transport. Applying \Cref{lem:2d-asymmetric-tame} with $g=\partial_y u^1_{0}$ or $g=\partial_z u^1_{0}$ gives
		\begin{align}
			\|u^2_{0}\partial_y u^1_{0}\|_{H^2}
			&\leq C\|u^2_{0}\|_{H^2}\|u^1_{0}\|_{H^4},
			\label{eq:R2ay-H2-revised}\\
			\|u^3_{0}\partial_z u^1_{0}\|_{H^2}
			&\leq C\|u^3_{0}\|_{H^1}\|u^1_{0}\|_{H^4}
			+C\|u^3_{0}\|_{H^2}\|u^1_{0}\|_{H^3}.
			\label{eq:R3az-H2-revised}
		\end{align}
		By interpolation, \Cref{lem:transient-streak-revised}, and the definition of $E_{1}$,
		\begin{equation}
			\|u^1_{0}(t)\|_{H^3}
			\leq C\|u^1_{0}(t)\|_{H^2}^{1/2}\|u^1_{0}(t)\|_{H^4}^{1/2}
			\leq C(E_{1}+\frac{\varepsilon}{\nu})m_\nu(t)^{1/2}.
			\label{eq:a-H3-transient-revised}
		\end{equation}
		Since $m_\nu(t)\leq\mu_\nu(t)$, the second term in \eqref{eq:R3az-H2-revised} is bounded by $CE_{2}(E_{1}+\frac{\varepsilon}{\nu})$.  Equations \eqref{eq:R2ay-H2-revised}--\eqref{eq:R3az-H2-revised} prove the $L_t^\infty H^2$ estimate.
		
		For the dissipative part, the $H^3$ estimate in \Cref{lem:2d-asymmetric-tame} gives
		\begin{align}
			\|u^2_{0}\partial_y u^1_{0}\|_{H^3}
			&\leq C\|u^2_{0}\|_{H^3}\|u^1_{0}\|_{H^4},
			\nonumber\\
			\|u^3_{0}\partial_z u^1_{0}\|_{H^3}
			&\leq C\|u^3_{0}\|_{H^2}\|u^1_{0}\|_{H^4}
			+C\|u^3_{0}\|_{H^3}\|u^1_{0}\|_{H^3}.
			\notag
		\end{align}
		For the full inhomogeneous time norms, combine
\eqref{eq:roll-recovery-time-final} with the separate
$\nu^{1/2}\|u^\perp_0\|_{L_t^2L^2}$ component of $E_2$ in
\eqref{eq:E2-definition}. Since $\mu_\nu\leq1$, this gives
\begin{align*}
 \nu^{1/2}\|u^2_0\|_{L_t^2H^3}
 &\leq C\nu^{1/2}\bigl(\|u^2_0\|_{L_t^2L^2}
 +\|\nabla_\perp u^2_0\|_{L_t^2H^2}\bigr)\leq CE_2,\\
 \nu^{1/2}\|u^3_0\|_{L_t^2H^2}
 &\leq C\nu^{1/2}\bigl(\|u^3_0\|_{L_t^2L^2}
 +\|\nabla_\perp u^3_0\|_{L_t^2H^1}\bigr)\leq CE_2,\\
 \nu^{1/2}\|\mu_\nu^{1/2}u^3_0\|_{L_t^2H^3}
 &\leq C\nu^{1/2}\bigl(\|u^3_0\|_{L_t^2L^2}
 +\|\mu_\nu^{1/2}\nabla_\perp u^3_0\|_{L_t^2H^2}\bigr)\leq CE_2.
\end{align*}
Use the last bound together with \eqref{eq:a-H3-transient-revised}
for the last product; $m_\nu\leq\mu_\nu$ cancels the weight.  Thus
		\begin{equation*}
			\nu^{1/2}\|u^\perp_{0}\cdot\nabla_\perp u^1_{0}\|_{L_t^2H^3}
			\leq CE_{2}(E_{1}+\frac{\varepsilon}{\nu}).
		\end{equation*}
		Together with the wave estimate this proves \eqref{eq:Na-Y2-revised}.  The pointwise estimate \eqref{eq:Na-point-revised} follows from the already established $L_t^\infty H^2$ bounds.
	\end{proof}
	
	\begin{proof}[Proof of \Cref{prop:A-closed-final}]
		We estimate separately the low-frequency amplitude, the top homogeneous energy, and the time derivative.
		
		\emph{Step 1: the $L_t^\infty L^2$ amplitude.}
		Apply the exact superposition identity \eqref{eq:superposition-identity} with the roll stress and streak flux displayed there. The heat term is bounded by $C\|u^1_{0}(0)\|_{L^2}$. Using directly $u^2_{0}(0)$, the $L^2$ heat estimate \eqref{eq:heat-liftup} gives
		\begin{equation*}
			\sup_{t>0}\|t\ee^{\nu t\Delta_\perp}u^2_{0}(0)\|_{L^2}
			\leq C\nu^{-1}\|u^2_{0}(0)\|_{\dot B^{-2}_{2,\infty}}.
		\end{equation*}
		The $N=0$ case of \Cref{lem:superposition} uses only the $L_t^1L^1$ norm and gives
		\begin{align}
			\|\mathcal S_\nu(u^\perp_{0}\otimes u^\perp_{0}+(u^\perp_{\neq}\otimes u^\perp_{\neq})_{0})\|_{L_t^\infty L^2}
			&\leq C\nu^{-1}\left(\|u^\perp_{0}\|_{L_t^2L^2}^2+\|(u^\perp_{\neq}\otimes u^\perp_{\neq})_{0}\|_{L_t^1L^1}\right)
			\nonumber\\
			&\leq C\frac{(E_{2})^2}{\nu^2}+C\nu^{-1}\bigl(\nu^{-1/3}E_{4}^2\bigr)
			\nonumber\\
			&=C\frac{(E_{2})^2}{\nu^2}+C\frac{E_{4}^2}{\nu^{4/3}},
			\notag
		\end{align}
		where the wave-stress estimate is exactly \eqref{eq:T-L1-time-explicit}.
		Furthermore,
		\begin{equation*}
			\|u^\perp_{0}u^1_{0}\|_{L_t^2L^2}
			\leq C\|u^1_{0}\|_{L_t^\infty H^2}\|u^\perp_{0}\|_{L_t^2L^2}
			\leq C\nu^{-1/2}E_{2}E_{1},
		\end{equation*}
		and the quadratic wave estimate gives $\|(u^\perp_{\neq}u^1_{\neq})_{0}\|_{L_t^2L^2}\leq C\nu^{-1/2}E_{4}^2$.  Hence \Cref{lem:direct-Duhamel} yields
		\begin{equation*}
			\|\mathcal D_\nu(u^\perp_{0}u^1_{0}+(u^\perp_{\neq}u^1_{\neq})_{0})\|_{L_t^\infty L^2}
			\leq C\frac{E_{2}}{\nu}E_{1}+C\frac{E_{4}^2}{\nu}.
		\end{equation*}
		The last term is bounded by $CE_{4}^2/\nu^{4/3}$ because $\nu^{-1}\leq\nu^{-4/3}$ for $0<\nu\leq1$.  Consequently,
		\begin{equation}
			\|u^1_{0}\|_{L_t^\infty L^2}
			\leq C\left(\|u^1_{0}(0)\|_{L^2}+\nu^{-1}\|u^2_{0}(0)\|_{\dot B^{-2}_{2,\infty}}+\frac{(E_{2})^2}{\nu^2}+\frac{E_{4}^2}{\nu^{4/3}}+\frac{E_{2}}{\nu}E_{1}\right).
			\label{eq:a-L2-closed-revised}
		\end{equation}
		
		\emph{Step 2: the top homogeneous streak energy.}
		Apply $\Delta_\perp^2$ to \eqref{eq:a-Na-revised}, take the $L^2$ scalar product with $\Delta_\perp^2u^1_{0}$, and integrate one derivative from each forcing term.  We obtain
		\begin{align}
			&\|\Delta_\perp^2u^1_{0}\|_{L_t^\infty L^2}
			+\nu^{1/2}\|\nabla_\perp\Delta_\perp^2u^1_{0}\|_{L_t^2L^2}
			\nonumber\\
			&\leq C\|u^1_{0}(0)\|_{H^4}
			+C\nu^{-1/2}\|\nabla_\perp\Delta_\perp u^2_{0}\|_{L_t^2L^2}
			+C\nu^{-1/2}\|\nabla_\perp\Delta_\perp\bigl(u^\perp_{0}\cdot\nabla_\perp u^1_{0}+(u_{\neq}\cdot\nabla u^1_{\neq})_{0}\bigr)\|_{L_t^2L^2}.
			\label{eq:a-top-energy-revised}
		\end{align}
		Indeed,
		\begin{equation*}
			\big|\langle\Delta_\perp^2\bigl(u^\perp_{0}\cdot\nabla_\perp u^1_{0}+(u_{\neq}\cdot\nabla u^1_{\neq})_{0}\bigr),\Delta_\perp^2u^1_{0}\rangle\big|
			=\big|\langle\nabla_\perp\Delta_\perp\bigl(u^\perp_{0}\cdot\nabla_\perp u^1_{0}+(u_{\neq}\cdot\nabla u^1_{\neq})_{0}\bigr),\nabla_\perp\Delta_\perp^2u^1_{0}\rangle\big|,
		\end{equation*}
		so no inhomogeneous $L_t^2H^3$ norm of the zero mode $\bigl(u^\perp_{0}\cdot\nabla_\perp u^1_{0}+(u_{\neq}\cdot\nabla u^1_{\neq})_{0}\bigr)$ is required.  The first forcing norm in \eqref{eq:a-top-energy-revised} is at most $CE_{2}/\nu$ after the prefactor, by \eqref{eq:roll-recovery-time-final}.  The dissipative part of \eqref{eq:Na-Y2-revised} gives
		\begin{equation*}
			\nu^{1/2}\|\nabla_\perp\bigl(u^\perp_{0}\cdot\nabla_\perp u^1_{0}+(u_{\neq}\cdot\nabla u^1_{\neq})_{0}\bigr)\|_{L_t^2H^2}
			\leq C\left[E_{2}(E_{1}+\frac{\varepsilon}{\nu})+E_{4}E_{5}\right],
		\end{equation*}
		and hence the second forcing norm in \eqref{eq:a-top-energy-revised} is at most
		\begin{equation*}
			C\nu^{-1}\left[E_{2}(E_{1}+\frac{\varepsilon}{\nu})+E_{4}E_{5}\right].
		\end{equation*}
		Therefore
		\begin{align}
			&\|\Delta_\perp^2u^1_{0}\|_{L_t^\infty L^2}
			+\nu^{1/2}\|\nabla_\perp\Delta_\perp^2u^1_{0}\|_{L_t^2L^2}
			\nonumber\\
			&\leq C\|u^1_{0}(0)\|_{H^4}+C\frac{E_{2}}{\nu}
			+C\frac{E_{2}}{\nu}(E_{1}+\frac{\varepsilon}{\nu})+C\frac{E_{4}E_{5}}{\nu}.
			\label{eq:a-top-closed-revised}
		\end{align}
		The elliptic inequality $\|u^1_{0}\|_{H^4}\leq C(\|u^1_{0}\|_{L^2}+\|\Delta_\perp^2u^1_{0}\|_{L^2})$ combines \eqref{eq:a-L2-closed-revised} and \eqref{eq:a-top-closed-revised}.
		
		\emph{Step 3: the time derivative.}
		Equation \eqref{eq:a-Na-revised} and \eqref{eq:Na-point-revised} give
		\begin{align}
			\nu^{-1}\|\partial_t u^1_{0}\|_{L_t^\infty H^2}
			&\leq C\|u^1_{0}\|_{L_t^\infty H^4}
			+C\frac{E_{2}}{\nu}
			+C\frac{E_{2}}{\nu}(E_{1}+\frac{\varepsilon}{\nu})
			+C\frac{E_{4}E_{5}}{\nu}.
			\label{eq:at-closed-revised}
		\end{align}
		Substituting the bounds from Steps 1--2 into \eqref{eq:at-closed-revised}, and recalling the definition \eqref{eq:A-final}, proves \eqref{eq:A-closed-final}.
	\end{proof}

\subsection{Nonzero-mode energy estimates}
	\label{sec:nonzero-closure-final}
For the nonzero modes, we estimate $E_{6}$, $E_{4}$, and $E_{5}$ using the good-unknown, coupled, and high-order linear bounds, respectively.
	
Apply the transverse recovery estimate to the exact nonlinear system, using the complete forcing bound from Section~\ref{sec:interactions-final}.
	
	\begin{proof}[Proof of \Cref{prop:E6-nonlinear-final}]
		Apply \Cref{prop:complementary-transverse-closure} to \eqref{eq:nonzero-transverse-final}. Its initial terms are exactly bounded by
		\[
		C\left(\|(u^2_{\neq},u^3_{\neq})(0)\|_{H^2}^2+\|\Delta q(0)\|_{L^2}^2\right).
		\] If the three nonnegative terms defining $\mathfrak F(T)$ in \eqref{eq:Fforcing-final} are denoted by $F_1,F_2,F_3$, then the forcing contribution in \eqref{eq:complementary-transverse-final} is $F_1^2+F_2^2+F_3^2\leq\mathfrak F(T)^2$. Hence
		\begin{equation*}
			E_{6}(T)^2\leq C\left(\|(u^2_{\neq},u^3_{\neq})(0)\|_{H^2}^2+\|\Delta q(0)\|_{L^2}^2\right)
			+C\left(\frac{E_{2}}{\nu}E_{6}+\frac{E_{4}^2}{\nu}\right)^2.
		\end{equation*}
		This proves \eqref{eq:E6-nonlinear-final}.
	\end{proof}
	
The coupled Couette estimate and the primary forcing bounds from Section~\ref{sec:interactions-final} give the $E_{4}$ proposition.
	
	\begin{proof}[Proof of \Cref{prop:E3-primary-final}]
		Apply \Cref{prop:coupled-Couette-final} to \eqref{eq:primary-W2-final}--\eqref{eq:primary-W3-final}, with $q=\Delta u^2_{\neq}$ and $h=u^3_{\neq}$. The initial contribution is bounded by $C(\|\Delta u^2_{\neq}(0)\|_{L^2}+\|u^3_{\neq}(0)\|_{H^2})$ by \eqref{eq:coupled-estimate-final}. By \Cref{lem:streak-primary-revised}, the two streak-primary forcing terms contribute
		\begin{align}
			C(E_{1}+\frac{\varepsilon}{\nu})\left[E_{6}+(E_{4}E_{6})^{1/2}\right]
			&\leq C(E_{1}+\frac{\varepsilon}{\nu})(E_{4}+E_{6}),\notag
		\end{align}
		where we used $2(E_{4}E_{6})^{1/2}\leq E_{4}+E_{6}$.
		The concrete nonlinear terms require the strong $2\gamma$-weighted primary forcing norm appearing in \Cref{prop:coupled-Couette-final}.  By \Cref{lem:primary-forcing-2gamma},
		\begin{equation*}
			\nu^{-1/2}\|\ee^{2\gamma\nu^{1/3}t}\nabla g^2\|_{L_t^2L_x^2}
			+\nu^{-1/2}\|\ee^{2\gamma\nu^{1/3}t}(\partial_x,\partial_z)g^3\|_{L_t^2L_x^2}
			\leq C\frac{E_{2}}{\nu}E_{6}+C\frac{E_{4}^2}{\nu}.
		\end{equation*}
		Thus the nonlinear transport and pressure terms contribute exactly the last two terms on the right-hand side of \eqref{eq:E3-primary-final}.
		This proves the estimate for the first two components of $E_{4}$. The final term $\nu^{2/3}\|\Delta u^3_{\neq}\|_{X_{3\gamma}}$ is bounded directly by $E_{6}$, by definition \eqref{eq:E6-final}. Combining the estimates proves \eqref{eq:E3-primary-final}.
	\end{proof}

The linear estimate in \Cref{lem:triangular-high-revised} and the high-order products in Section~\ref{sec:interactions-final} now give the fourth-order bound.
	
	\begin{proof}[Proof of \Cref{prop:E4-closed-final}]
		The full Leray projector does not commute with the Couette transport. Since $\nabla\cdot u_{\neq}=0$,
		\begin{equation*}
			\nabla\cdot(y\partial_xu_{\neq})=\partial_xu^2_{\neq},
		\end{equation*}
		and therefore
		\begin{equation*}
			\mathbb P(y\partial_xu_{\neq})
			=y\partial_xu_{\neq}-\nabla\Delta^{-1}\partial_xu^2_{\neq}.
		\end{equation*}
		Applying $\mathbb P$ to the nonzero equation gives the exact identity
		\begin{equation}
			L_{0}u_{\neq}
			=\nabla\Delta^{-1}\partial_xu^2_{\neq}
			-\mathbb P(u^2_{\neq}e_1)
			-\mathbb P (u\cdot\nabla u)_{\neq}.
			\label{eq:Leray-nonzero-revised}
		\end{equation}
		Apply \Cref{lem:triangular-high-revised} to each component. The lift-up forcing and the Leray commutator satisfy
		\begin{align}
			\|\ee^{2\gamma\nu^{1/3}t}\mathbb P(u^2_{\neq}e_1)\|_{Y^2}
			&\leq C E_{4}(T),
			\label{eq:liftup-Y2-complete}\\
			\|\ee^{2\gamma\nu^{1/3}t}\nabla\Delta^{-1}\partial_xu^2_{\neq}\|_{Y^2}
			&\leq C E_{4}(T).
			\label{eq:Leray-commutator-Y2}
		\end{align}
		Indeed, $\mathbb P$ is bounded on every $L^2$-based Sobolev space used here, while the multiplier of $\nabla\Delta^{-1}\partial_x$ is
		\begin{equation*}
			\frac{\xi k}{k^2+\eta^2+\ell^2},
		\end{equation*}
		whose modulus is at most one. Hence both estimates reduce to the corresponding $Y^2$ bound for $u^2_{\neq}$, which follows from $\|u^2_{\neq}\|_{H^2}\lesssim\|\Delta u^2_{\neq}\|_{L^2}$ and $\|\nabla u^2_{\neq}\|_{H^2}\lesssim\|\nabla\Delta u^2_{\neq}\|_{L^2}$. Furthermore, \eqref{eq:velocity-reconstruction-final} with $s=1$ and \eqref{eq:P-profiles-revised} give
		\begin{align}
			&\nu^{-1/2}\|\ee^{2\gamma\nu^{1/3}t}
			\nabla\partial_x(\partial_x,\partial_z)u_{\neq}\|_{L_t^2L_x^2}
			\nonumber\\
			&\qquad\leq C\nu^{-1/2}
			\|\ee^{2\gamma\nu^{1/3}t}\Pi_1\|_{L_t^2}
			\leq C\nu^{-1}E_{4}(T).
			\notag
		\end{align}
		Thus \Cref{lem:triangular-high-revised}, together with \eqref{eq:liftup-Y2-complete} and \eqref{eq:Leray-commutator-Y2}, yields
		\begin{equation}
			E_{5}
			\leq C\|u_{\neq}(0)\|_{H^4}+C\nu^{-1}(E_{4}+E_{6})
			+C\nu^{-1}\|\ee^{2\gamma\nu^{1/3}t}(u\cdot\nabla u)_{\neq}\|_{Y^2}.
			\label{eq:E4-pre-final}
		\end{equation}
		The additional commutator term in \eqref{eq:Leray-nonzero-revised} therefore contributes only another $C\nu^{-1}E_{4}$, already contained in the displayed linear term; in particular it does not alter the statement of \Cref{prop:E4-closed-final}.
		Decompose
		\begin{align}
			(u\cdot\nabla u)_{\neq}
			={}&(u_{\neq}\cdot\nabla u_{\neq})_{\neq}
			+u^\perp_{0}\cdot\nabla_\perp u_{\neq}
			+u^\perp_{\neq}\cdot\nabla_\perp u^\perp_{0}
			\nonumber\\
			&+u^1_{0}\partial_xu_{\neq}+u^\perp_{\neq}\cdot\nabla_\perp u^1_{0}\,e_1.
			\notag
		\end{align}
		For clarity, we record the complete interaction ledger before applying the external factor $\nu^{-1}$ from \eqref{eq:E4-pre-final}:
		\begin{align}
			\|\ee^{2\gamma\nu^{1/3}t}(u_{\neq}\cdot\nabla u_{\neq})_{\neq}\|_{Y^2}
			&\leq C E_{4}E_{5},
			\label{eq:E4-ledger-wave-wave}\\
			\|\ee^{2\gamma\nu^{1/3}t}(u^\perp_{0}\cdot\nabla_\perp u_{\neq}+u^\perp_{\neq}\cdot\nabla_\perp u^\perp_{0})\|_{Y^2}
			&\leq C E_{2}E_{5}+C E_{3}E_{4},
			\label{eq:mixed-Y2-sharp-revised}\\
			\|\ee^{2\gamma\nu^{1/3}t}(u^1_{0}\partial_xu_{\neq}+u^\perp_{\neq}\cdot\nabla_\perp u^1_{0}\,e_1)\|_{Y^2}
			&\leq C\nu(E_{1}+\frac{\varepsilon}{\nu})E_{5}+C(E_{1}+\frac{\varepsilon}{\nu})(E_{4}+E_{6}).
			\label{eq:E4-ledger-streak-wave}
		\end{align}
		The first line is \Cref{lem:wave-Y2-final}, the second is the last assertion of \Cref{prop:forcing-final}, and the third is \Cref{lem:streak-high-product-revised}. In particular,
		\[
		\nu^{-1}\bigl[\nu(E_{1}+\frac{\varepsilon}{\nu})E_{5}\bigr]=(E_{1}+\frac{\varepsilon}{\nu})E_{5},
		\]
		so the high-order streak interaction has a perturbative coefficient uniformly as $\nu\to0$; no term of the form $(E_{1}/\nu)E_{5}$ remains. Substitution of \eqref{eq:E4-ledger-wave-wave}--\eqref{eq:E4-ledger-streak-wave} into \eqref{eq:E4-pre-final} gives \eqref{eq:E4-closed-final}. \end{proof}

\appendix

\section{Linear estimates}
\label{sec:nonzero-linear}
We recall the linear estimates from
\cite[Sections~3.2, 4, and~9.2]{WeiZhang2021} in our notation.
\Cref{prop:constant-scalar,prop:variable-scalar,prop:coupled-Couette-final}
correspond to Propositions~4.1, 4.4, and~4.3 there, respectively.
\Cref{lem:coordinate-equivalence} follows from the proofs of
Lemmas~4.5--4.6, and \Cref{lem:triangular-high-revised} is the
differentiated scalar estimate derived in Section~9.2,
pp.~2457--2458. We omit the proofs and explain the domain and
normalization adjustments in \Cref{rem:linear-continuous-spanwise}.

The scalar estimate controls the four components of $X_\sigma$.

\begin{Prop}
		\label{prop:constant-scalar}
		There exist absolute constants $\gamma_{0}>0$ and $C>0$ such that the following holds. Let $w$ solve
		\begin{equation*}
			L_{0}w=\px f_1+f_2+\operatorname{div}f_3,
			\qquad
			w_{0}=(f_1)_{0}=(f_2)_{0}=(f_3)_{0}=0,
		\end{equation*}
		on $[0,T]$. Then, for $0\leq\sigma\leq\gamma_{0}$,
		\begin{align}
			\norm{w}{X_\sigma([0,T])}^2
			\leq C\Big(&\norm{w(0)}{L^2}^2
			+\norm{\ee^{\sigma\nu^{1/3}t}\nabla f_1}{L^2(0,T;L^2)}^2
			\nonumber\\
			&+\nu^{-1/3}\norm{\ee^{\sigma\nu^{1/3}t}f_2}{L^2(0,T;L^2)}^2
			+\nu^{-1}\norm{\ee^{\sigma\nu^{1/3}t}f_3}{L^2(0,T;L^2)}^2\Big).
			\label{eq:constant-scalar-est}
		\end{align}
		The constant is independent of $T$, $\nu$, and the continuous spanwise frequency.
	\end{Prop}

For the map $\Psi_t$ defined in Section~\ref{sec:full-system}, the following composition lemma includes the negative-order norm needed for inviscid damping.
	
	\begin{Lem}
		\label{lem:coordinate-equivalence}
		There exists $\delta_*>0$ such that, if
		\begin{equation}
			\norm{\nablaP u^1_{0}(t)}{L^\infty}+\norm{\nablaP^2u^1_{0}(t)}{L^\infty}\leq\delta_*
			\label{eq:coord-small}
		\end{equation}
		for a fixed $t$, then $\Psi_t$ is a $C^1$ diffeomorphism. If $w_{0}=0$, then
		\begin{equation*}
			\norm{(w\circ\Psi_t^{-1})}{L^2}\simeq\norm{w}{L^2},
			\qquad
			\norm{\nabla (w\circ\Psi_t^{-1})}{L^2}\simeq\norm{\nabla w}{L^2},
		\end{equation*}
		\begin{equation}
			\norm{\nabla\Delta^{-1}\pa_X(w\circ\Psi_t^{-1})}{L^2}
			\simeq
			\norm{\nabla\Delta^{-1}\px w}{L^2}.
			\label{eq:Hminus1-equiv}
		\end{equation}
		The constants are uniform as long as the left-hand side of \eqref{eq:coord-small} is bounded by $\delta_*$.
	\end{Lem}

	The next result assumes uniform smallness of the streak in $H^4$ and of its time derivative in $\nu H^2$.
	
	\begin{Prop}
		\label{prop:variable-scalar}
		There exist absolute constants $\delta_{0},\gamma_{0}>0$ and $C>0$ such that the following holds. Assume
		\begin{equation}
			\sup_{0\leq \tau\leq T}\bigl(\|u^1_{0}(\tau)\|_{H^4(\R^2)}+\nu^{-1}\|\partial_tu^1_{0}(\tau)\|_{H^2(\R^2)}\bigr)\leq\delta_{0}.
			\label{eq:variable-small-assumption}
		\end{equation}
		Let $w$ solve
		\begin{equation}
			L_Vw=\px f_1+f_2+\operatorname{div}f_3,
			\qquad
			w_{0}=(f_1)_{0}=(f_2)_{0}=(f_3)_{0}=0,
			\label{eq:variable-scalar}
		\end{equation}
		on $[0,T]$. Then, for $0\leq\sigma\leq\gamma_{0}$,
		\begin{align}
			\norm{w}{X_\sigma([0,T])}^2
			\leq C\Big(&\norm{w(0)}{L^2}^2
			+\norm{\ee^{\sigma\nu^{1/3}t}\nabla f_1}{L^2(0,T;L^2)}^2
			\nonumber\\
			&+\nu^{-1/3}\norm{\ee^{\sigma\nu^{1/3}t}f_2}{L^2(0,T;L^2)}^2
			+\nu^{-1}\norm{\ee^{\sigma\nu^{1/3}t}f_3}{L^2(0,T;L^2)}^2\Big).
			\label{eq:variable-scalar-est}
		\end{align}
		The constants are independent of $T$ and $\nu$.
	\end{Prop}

The primary energy $E_{4}$ uses the following coupled estimate. Fix
$\sigma_{\rm c}>0$ below the admissible weight threshold in
\Cref{prop:constant-scalar}.

\begin{Prop}
		\label{prop:coupled-Couette-final}
		Let $q_{0}=h_{0}=0$ and suppose
		\begin{equation}
			L_{0}q=\Delta f,
			\qquad
			L_{0}h-2\partial_x\partial_z\Delta^{-2}q=r.
			\label{eq:coupled-system-final}
		\end{equation}
		Then, for $0\leq\sigma\leq\sigma_{\rm c}$,
		\begin{align}
			\|q\|_{X_\sigma}^2+\|-(\partial_x^2+\partial_z^2)h\|_{X_\sigma}^2
			\leq C\Big(&\|q(0)\|_{L^2}^2+\|h(0)\|_{H^2}^2
			\nonumber\\
			&+\nu^{-1}\|\ee^{\sigma\nu^{1/3}t}\nabla f\|_{L_t^2L_x^2}^2
			+\nu^{-1}\|\ee^{\sigma\nu^{1/3}t}(\partial_x,\partial_z)r\|_{L_t^2L_x^2}^2\Big).
			\label{eq:coupled-estimate-final}
		\end{align}
		The constant is uniform in the continuous spanwise Fourier frequency.
	\end{Prop}

We also need a differentiated scalar estimate for the high-order energy. Let $\gamma_{\rm C}>0$ be an admissible exponential-weight constant in \Cref{prop:constant-scalar}, and set
	\begin{equation*}
		\sigma_{\rm h}:=\min\{\sigma_{\rm c},\gamma_{\rm C}\}.
	\end{equation*}
	\begin{Lem}
		\label{lem:triangular-high-revised}
		Let $F$ be $x$-nonzero and solve
		\begin{equation*}
			L_{0}F=H.
		\end{equation*}
		Then, for $0\leq\sigma\leq\sigma_{\rm h}$,
		\begin{align}
			&\|\ee^{\sigma\nu^{1/3}t}(\partial_x,\partial_z)F\|_{L_t^\infty H^3}
			+\nu^{1/2}\|\ee^{\sigma\nu^{1/3}t}\nabla(\partial_x,\partial_z)F\|_{L_t^2H^3}
			\nonumber\\
			&\leq C\|F(0)\|_{H^4}
			+C\nu^{-1}\|\ee^{\sigma\nu^{1/3}t}H\|_{Y^2}
			+C\nu^{-1/2}\|\ee^{\sigma\nu^{1/3}t}
			\nabla\partial_x(\partial_x,\partial_z)F\|_{L_t^2L_x^2}.
			\label{eq:triangular-high-revised}
		\end{align}
	\end{Lem}

\begin{Rem}
\label{rem:linear-continuous-spanwise}
The estimates above extend to $\T_x\times\R_y\times\R_z$ from
\cite[Lemmas~4.2, 4.5--4.6 and Propositions~4.3--4.4]{WeiZhang2021}.
The constant-shear bounds are uniform in $\ell\in\R$ and use only
the streamwise gap $|k|\geq1$. Thus Plancherel applies with integration
in $\ell$ instead of summation, after the time supremum is estimated
mode by mode. For variable shear, the coordinate and coefficient bounds required
there follow from \eqref{eq:coord-small} and
\eqref{eq:variable-small-assumption} and remain valid on $\R^2$.
This also gives the negative-order norm equivalence and the
commutator bounds used in the differentiated estimates. Finally, $X_\sigma$ corresponds to $X_a$ with $a=\sigma$.
Fixed Fourier normalization factors are absorbed into the constants
and admissible weight thresholds. The time translation from
$[1,1+T]$ to $[0,T]$ changes the exponential weights only by the
bounded factor $\ee^{\sigma\nu^{1/3}}$.
\end{Rem}

\section{Good unknowns and transverse velocity recovery}
\label{sec:second-good}
The good-unknown estimates used below follow from
\cite[Sections~3.3, 4, and~9.3]{WeiZhang2021}.
In particular, \Cref{prop:good-operator,prop:good-derivative-scalar}
correspond to Propositions~4.9 and~4.7, respectively.
The identity in \Cref{lem:good-unknown-algebra} is given by
equations~(3.8)--(3.9), with the additional pressure source $h$
contributing $-(\partial_y+\kappa\partial_z)\Delta^{-1}h$.
The decomposition in \Cref{lem:rho-decomposition} corresponds to
equations~(3.10)--(3.11), and its coefficient bounds follow from
Lemma~4.8. The correctors $q^{(2)}$ and $q^{(3)}$ correspond to
$W^{2,2}$ and $W^{2,3}$ in equation~(9.4).
The estimates \eqref{eq:q3-good-est},
\eqref{eq:q2-px2-est}--\eqref{eq:q2-full-gradient-est}, and
\eqref{eq:q2-rho-q3-est} in \Cref{lem:singular-correctors}
correspond to Lemmas~9.1, 9.2, and~9.3, respectively. The remaining forced estimates are extracted from the proof of
Proposition~9.4, before the nonlinear forcing is substituted
in Step~5.
Specifically, \Cref{lem:complementary-r} corresponds to
equations~(9.11)--(9.12).
The estimate in \Cref{lem:f21-regular} follows from
the decomposition~(9.10) and the bounds~(9.16)--(9.17),
followed by time integration.
\Cref{prop:second-good-unknown} corresponds to
equations~(9.13)--(9.18), while
\Cref{prop:complementary-transverse-closure} collects
equations~(9.19)--(9.22). Here $W^2$ and $W^{2,1}$ are denoted by $q$ and $q^{(1)}$,
and the auxiliary squared energy $E_6$ of the cited work
is denoted by $\mathcal E_{\perp,\sigma}$.
The estimates are stated for the admissible weight $\sigma$,
with the starting time shifted to zero and the initial-data
terms retained explicitly.
The passage to $\T_x\times\R_y\times\R_z$ is addressed in
\Cref{rem:linear-continuous-spanwise}, so the borrowed proofs
are omitted. Finally, \Cref{cor:w1-recovery} follows from incompressibility
and the Poincar\'e inequality.

The estimates in this appendix allow any nonnegative weight below their stated linear thresholds; $\sigma$ denotes such a weight locally. All $X_\sigma$ norms are taken on $[0,T]$, and their arguments are $x$-nonzero. We use repeatedly the elementary consequences of \eqref{eq:Xgamma-def},
	\begin{equation}
		\nu^{1/3}\norm{\ee^{\sigma\nu^{1/3}t}f}{L_t^2L_x^2}^2
		+\nu\norm{\ee^{\sigma\nu^{1/3}t}\nabla f}{L_t^2L_x^2}^2
		\leq \norm{f}{X_\sigma}^2,
		\label{eq:Xgamma-basic-parts}
	\end{equation}
	and the Poincar\'e inequalities
	\begin{equation}
		\norm{f}{H^m(\Omega)}\leq \norm{\px f}{H^m(\Omega)},
		\qquad
		\norm{f}{X_\sigma}\leq \norm{\px f}{X_\sigma}.
		\label{eq:x-Poincare-X}
	\end{equation}

	The pressure operator $\mathscr L_V$ defined in Section~\ref{sec:full-system} satisfies the following elliptic estimate.
	
	\begin{Prop}
		\label{prop:good-operator}
		There exist absolute constants $\delta_1,\gamma_1>0$ and $C>0$ such that the following holds. Assume
		\begin{equation*}
			\sup_{0\leq \tau\leq T}\bigl(\|u^1_{0}(\tau)\|_{H^4(\R^2)}+\nu^{-1}\|\partial_tu^1_{0}(\tau)\|_{H^2(\R^2)}\bigr)\leq\delta_1.
		\end{equation*}
		Let $q$ solve
		\begin{equation*}
			\mathscr L_Vq=f,
			\qquad q_{0}=f_{0}=0,
		\end{equation*}
		on $[0,T]$. Then, for $0\leq\sigma\leq\gamma_1$,
		\begin{equation}
			\norm{\Delta q}{X_\sigma([0,T])}^2
			\leq C\left(\norm{\Delta q(0)}{L^2}^2
			+\nu^{-1}\norm{\ee^{\sigma\nu^{1/3}t}\nabla f}{L^2(0,T;L^2)}^2\right).
			\label{eq:good-operator-est}
		\end{equation}
		The constant is independent of $T$, $\nu$, and the continuous spanwise frequency.
	\end{Prop}

	Combining the transverse equations gives the following identity,
corresponding to equations~(3.8)--(3.9) in the cited work with an
additional pressure source $h$.
	
	\begin{Lem}
		\label{lem:good-unknown-algebra}
		Suppose $w^2,w^3,p$ satisfy
		\begin{equation}
			L_Vw^2+\py p=f^2,
			\qquad
			L_Vw^3+\pz p=f^3,
			\label{eq:transverse-system}
		\end{equation}
		\begin{equation}
			\Delta p=-2\px(V_yw^2+V_zw^3)+h.
			\label{eq:pressure-good}
		\end{equation}
		Let $q:=w^2+\kappa w^3$. Then
		\begin{align}
			\mathscr L_Vq={}&f^2+\kappa f^3-(\py+\kappa\pz)\Delta^{-1}h
			\nonumber\\
			&+(\pt-\nu\DeltaP)\kappa\,w^3
			-2\nu\nablaP\kappa\cdot\nablaP w^3.
			\label{eq:good-unknown-equation}
		\end{align}
	\end{Lem}

	The next result controls the good derivatives $\partial_x$ and $D_z^V$.
	
	\begin{Prop}
		\label{prop:good-derivative-scalar}
		There exist absolute constants $\delta_2,\gamma_2>0$ such that the following holds. Assume \eqref{eq:variable-small-assumption} with $\delta_{0}$ replaced by $\delta_2$. Let $w$ solve
		\begin{equation}
			L_Vw=F_1+F_2+F_3,
			\qquad w_{0}=(F_1)_{0}=(F_2)_{0}=(F_3)_{0}=0.
			\label{eq:good-derivative-scalar-equation}
		\end{equation}
		Then, for $0\leq\sigma\leq\gamma_2$,
		\begin{align}
			\bigl(\norm{\px^2w}{X_\sigma}^2+\norm{\px D_z^Vw}{X_\sigma}^2\bigr)
			\leq C\Big(&\norm{w(0)}{H^2}^2
			+\norm{\ee^{\sigma\nu^{1/3}t}\Delta F_1}{L_t^2L_x^2}^2
			\nonumber\\
			&+\nu^{-1/3}\norm{\ee^{\sigma\nu^{1/3}t}\px^2F_2}{L_t^2L_x^2}^2
			+\nu^{-1/3}\norm{\ee^{\sigma\nu^{1/3}t}\px D_z^VF_2}{L_t^2L_x^2}^2
			\nonumber\\
			&+\nu^{-1}\norm{\ee^{\sigma\nu^{1/3}t}\px F_3}{L_t^2L_x^2}^2\Big).
			\label{eq:good-derivative-scalar-estimate}
		\end{align}
		The constant is independent of $T$, $\nu$, and the continuous spanwise frequency.
	\end{Prop}

We now use the correctors from Section~\ref{sec:full-system} to
control the normal viscous term.
	
	For a general transverse field $(w^2,w^3)$, we use
$\mathcal E_{\perp,\sigma}$ with $u^j_{\neq}$ replaced by $w^j$
in its definition in Section~\ref{sec:full-system}.
We now specialize to the transverse system
	\begin{equation}
		L_Vw^2+\py p=-g^2,
		\qquad
		L_Vw^3+\pz p=-g^3,
		\label{eq:transverse-forced-system}
	\end{equation}
	\begin{equation}
		\Delta p=-2\px(V_yw^2+V_zw^3).
		\label{eq:transverse-forced-pressure}
	\end{equation}
	Set
	\begin{equation*}
		q=w^2+\kappa w^3.
	\end{equation*}
	Then \Cref{lem:good-unknown-algebra} gives
	\begin{equation}
		\mathscr L_Vq+(g^2+\kappa g^3)
		=(\pt-\nu\DeltaP)\kappa\,w^3
		-2\nu\nablaP\kappa\cdot\nablaP w^3.
		\label{eq:q-equation-second}
	\end{equation}
	
	The following decomposition and bounds follow from equations~(3.10)--(3.11)
and Lemma~4.8 of the cited work.
	
	\begin{Lem}
		\label{lem:rho-decomposition}
		With the coefficients defined in \eqref{eq:rho-definitions},
		\begin{equation}
			\nablaP\kappa=\rho_1\nablaP V+\rho_2(-\kappa,1),
			\label{eq:rho-vector-decomposition}
		\end{equation}
		and hence
		\begin{equation}
			\nablaP\kappa\cdot\nablaP f
			=\rho_1\nablaP V\cdot\nablaP f+\rho_2D_z^Vf.
			\label{eq:rho-scalar-decomposition}
		\end{equation}
		If \eqref{eq:variable-small-assumption} holds with $\delta_{0}$ replaced by $\delta_2$, then
		\begin{align}
			\norm{\kappa}{H^3}+\norm{\rho_1}{H^2}+\norm{\rho_2}{H^2}
			&\leq C\delta_2,
			\label{eq:rho-coefficient-small}
			\\
			\norm{(\pt-\nu\DeltaP)\kappa}{H^1}
			+\norm{\pt\rho_1+\nu\DeltaP\rho_1}{L^2}
			&\leq C\nu\delta_2.
			\label{eq:rho-time-small}
		\end{align}
	\end{Lem}

	For the corrector estimates, the quantities $W^{2,2},W^{2,3}$ in
the cited work correspond to $q^{(2)},q^{(3)}$, and its auxiliary
squared energy $E_6$ corresponds to $\mathcal E_{\perp,\sigma}$.
The bounds hold for each admissible weight $\sigma$, with the
coefficient parameter replaced by $\delta_2$.
The correctors have zero initial data, so $q^{(1)}(0)=q(0)$ and
$(\Delta w^3-2\partial_xq^{(3)})(0)=\Delta w^3(0)$.
	
	\begin{Lem}
		\label{lem:singular-correctors}
		Assume \eqref{eq:variable-small-assumption} with $\delta_{0}$ replaced by $\delta_2$. Then
		\begin{align}
			\bigl(\norm{\px^2q^{(3)}}{X_\sigma}^2+\norm{\px D_z^Vq^{(3)}}{X_\sigma}^2\bigr)
			&\leq C\nu^{-4/3}\mathcal E_{\perp,\sigma}(T),
			\label{eq:q3-good-est}
			\\
			\norm{\px^2q^{(2)}}{X_\sigma}^2
			&\leq C\delta_2^2\nu^{-4/3}\mathcal E_{\perp,\sigma}(T),
			\label{eq:q2-px2-est}
			\\
			\norm{\px\nabla q^{(2)}}{X_\sigma}^2
			&\leq C\delta_2^2\nu^{-2}\mathcal E_{\perp,\sigma}(T),
			\label{eq:q2-full-gradient-est}
			\\
			\norm{\px(q^{(2)}-\rho_1q^{(3)})}{X_\sigma}^2
			&\leq C\delta_2^2\nu^{-2/3}\mathcal E_{\perp,\sigma}(T).
			\label{eq:q2-rho-q3-est}
		\end{align}
	\end{Lem}

	Equations~(9.11)--(9.12) of the cited work give the complementary
component estimate, with its initial norm retained below.
	
	\begin{Lem}
		\label{lem:complementary-r}
Then
		\begin{equation}
			L_V(\Delta w^3-2\px q^{(3)})=-\Delta g^3+2\pz(V_y\px q)-(\DeltaP V)\px w^3.
			\label{eq:r-complementary-equation}
		\end{equation}
		Consequently,
		\begin{align}
			\norm{(\Delta w^3-2\px q^{(3)})}{X_\sigma}^2
			\leq C\Big(&\norm{\Delta w^3(0)}{L^2}^2
			+\nu^{-2/3}\mathcal E_{\perp,\sigma}(T)
			\nonumber\\
			&+\nu^{-1}\norm{\ee^{\sigma\nu^{1/3}t}\nabla g^3}{L_t^2L_x^2}^2\Big).
			\label{eq:r-complementary-est}
		\end{align}
	\end{Lem}

	For the regular equation, define
	\begin{equation*}
		f_{2,1}:=\mathscr L_Vq^{(2)}+2\rho_1\nablaP V\cdot\nablaP w^3.
	\end{equation*}
	Using \eqref{eq:rho-scalar-decomposition}, \eqref{eq:q-equation-second}, and \eqref{eq:q1-second-good}, we obtain
	\begin{equation}
		\mathscr L_Vq^{(1)}+(g^2+\kappa g^3)
		=(\pt-\nu\DeltaP)\kappa\,w^3
		-2\nu\rho_2D_z^Vw^3
		-\nu f_{2,1}.
		\label{eq:q1-regular-equation}
	\end{equation}
	
	Equations~(9.16)--(9.17) in the same proof, combined with the
corrector and complementary-component estimates, give the following bound.

\begin{Lem}
		\label{lem:f21-regular}
		Under the hypotheses above,
		\begin{align}
			\nu\norm{\ee^{\sigma\nu^{1/3}t}\nabla f_{2,1}}{L_t^2L_x^2}^2
			\leq C\delta_2^2\Big(&\mathcal E_{\perp,\sigma}(T)
			+\nu^{2/3}\norm{\Delta w^3(0)}{L^2}^2
			\nonumber\\
			&+\nu^{-1/3}\norm{\ee^{\sigma\nu^{1/3}t}\nabla g^3}{L_t^2L_x^2}^2\Big).
			\label{eq:f21-regular-est}
		\end{align}
	\end{Lem}

	The next estimate follows from equations~(9.13)--(9.18) there,
with the initial norms kept explicit.

\begin{Prop}
		\label{prop:second-good-unknown}
		Assume \eqref{eq:variable-small-assumption} with $\delta_{0}$ replaced by $\delta_2$. Then
		\begin{align}
			\norm{\Delta q^{(1)}}{X_\sigma}^2
			\leq C\Big(&\norm{\Delta q(0)}{L^2}^2
			+\nu^{-1}\norm{\ee^{\sigma\nu^{1/3}t}\nabla (g^2+\kappa g^3)}{L_t^2L_x^2}^2
			\nonumber\\
			&+\nu^{-1/3}\norm{\ee^{\sigma\nu^{1/3}t}\nabla g^3}{L_t^2L_x^2}^2
			+\delta_2^2\nu^{2/3}\norm{\Delta w^3(0)}{L^2}^2
			\nonumber\\
			&+\delta_2^2\mathcal E_{\perp,\sigma}(T)\Big).
			\label{eq:second-good-unknown-est}
		\end{align}
	\end{Prop}

Equations~(9.19)--(9.22) in the cited proof give the transverse
closure below for the general forcing $(g^2,g^3)$.

	\begin{Prop}
		\label{prop:complementary-transverse-closure}
		There exist absolute constants $\delta_3,\gamma_3>0$ and $C>0$ such that the following holds. Assume \eqref{eq:variable-small-assumption} with $\delta_{0}$ replaced by $\delta_3$, and let $(w^2,w^3,p)$ solve \eqref{eq:transverse-forced-system}--\eqref{eq:transverse-forced-pressure}. Then, for $0\leq\sigma\leq\gamma_3$,
		\begin{align}
			&\mathcal E_{\perp,\sigma}(T)
			+\nu^{4/3}\norm{\Delta w^3}{X_\sigma}^2
			\nonumber\\
			&\quad\leq C\Big(
			\norm{(w^2,w^3)(0)}{H^2}^2
			+\norm{\Delta q(0)}{L^2}^2
			\nonumber\\
			&\qquad\qquad
			+\nu^{-1}\norm{\ee^{\sigma\nu^{1/3}t}\nabla (g^2+\kappa g^3)}{L_t^2L_x^2}^2
			+\nu^{-1}\norm{\ee^{\sigma\nu^{1/3}t}\px g^3}{L_t^2L_x^2}^2
			\nonumber\\
			&\qquad\qquad
			+\nu^{-1/3}\norm{\ee^{\sigma\nu^{1/3}t}\nabla g^3}{L_t^2L_x^2}^2
			\Big).
			\label{eq:complementary-transverse-final}
		\end{align}
		No spanwise spectral gap is used; the corresponding modal constants are uniform in $\ell\in\mathbb R$ and remain uniform after integration over the continuous $z$-Fourier variable.
	\end{Prop}

	Incompressibility gives
\begin{equation*}
 \partial_xw^1=-\partial_yq-D_z^Vw^3+(\partial_y\kappa)w^3.
\end{equation*}
Together with \eqref{eq:x-Poincare-X} and
\Cref{lem:anisotropic-products-R2}, this yields the following recovery estimate.
	
	\begin{Cor}
		\label{cor:w1-recovery}
		Assume the hypotheses of \Cref{prop:complementary-transverse-closure} and let $w=(w^1,w^2,w^3)$ be divergence free. Then
		\begin{equation*}
			\norm{w^1}{X_\sigma}^2+\norm{\px w^1}{X_\sigma}^2
			\leq C\mathcal E_{\perp,\sigma}(T).
		\end{equation*}
		Consequently, the right-hand side of \eqref{eq:complementary-transverse-final} controls all three components of the $x$-nonzero velocity in the anisotropic norms required in the nonlinear estimates of Section~\ref{sec:interactions-final}.
	\end{Cor}

	The domain and normalization adjustments in \Cref{rem:linear-continuous-spanwise} apply throughout this appendix.

\section{Auxiliary averaging, product, and interpolation estimates}
	\label{app:functional-tools}
\label{app:averaging-tools}
\label{app:anisotropic-tools}
\label{app:interpolation-tools}
We collect the averaging, product, and interpolation estimates
used in the nonlinear argument.

	The averaging convention in \eqref{eq:mode-convention} yields the following product bounds.
	
	\begin{Lem}
		\label{lem:average-products}
		Let $N\geq2$. Then
		\begin{align}
			\norm{(fg)_{0}}{L^1(\R^2)}&\leq \norm{f}{L^2(\Omega)}\norm{g}{L^2(\Omega)},
			\label{eq:avg-L1}\\
			\norm{(fg)_{0}}{L^2(\R^2)}&\leq C\norm{f}{H^1(\Omega)}\norm{g}{H^1(\Omega)},
			\label{eq:avg-L2}\\
			\norm{(fg)_{0}}{H^N(\R^2)}&\leq C_N\norm{f}{H^N(\Omega)}\norm{g}{H^N(\Omega)},
			\label{eq:avg-HN}
		\end{align}
		whenever the right-hand sides are finite.
	\end{Lem}
	
	\begin{proof}
The first estimate follows from Fubini and Cauchy--Schwarz.
For the second, Jensen's inequality and $H^1(\Omega)\hookrightarrow L^4(\Omega)$
give $\|(fg)_{0}\|_{L^2}\leq\|fg\|_{L^2}\leq C\|f\|_{H^1}\|g\|_{H^1}$.
The last follows because averaging commutes with transverse derivatives
and $H^N(\Omega)$ is an algebra for $N\geq2$.
\end{proof}
	
	\begin{Lem}
		\label{lem:2d-asymmetric-tame}
		For sufficiently regular scalar or vector-valued functions $f,g$ on $\mathbb R^2$,
		\begin{align}
			\|fg\|_{H^2(\mathbb R^2)}
			&\leq C\Big(
			\|f\|_{H^1(\mathbb R^2)}\|g\|_{H^3(\mathbb R^2)}
			+\|f\|_{H^2(\mathbb R^2)}\|g\|_{H^2(\mathbb R^2)}
			\Big),
			\notag\\
			\|fg\|_{H^3(\mathbb R^2)}
			&\leq C\Big(
			\|f\|_{H^2(\mathbb R^2)}\|g\|_{H^3(\mathbb R^2)}
			+\|f\|_{H^3(\mathbb R^2)}\|g\|_{H^2(\mathbb R^2)}
			\Big).
			\notag
		\end{align}
		The estimates hold componentwise for products of vector fields and scalar functions.
	\end{Lem}
	
	\begin{proof}
Apply Leibniz's rule and the two-dimensional embeddings
$H^1\hookrightarrow L^4$, $H^2\hookrightarrow L^\infty\cap W^{1,4}$,
and $H^3\hookrightarrow W^{2,4}$. In the $H^2$ estimate, the term
$f\nabla^2g$ is bounded by $C\|f\|_{H^1}\|g\|_{H^3}$;
the other terms are bounded by $C\|f\|_{H^2}\|g\|_{H^2}$.
For $H^3$, allocate at most one derivative to the $H^2$ factor
and use the symmetric bound for the remaining terms.
\end{proof}

	We next record the anisotropic product and coordinate estimates on the present domain.
	
	\begin{Lem}
		\label{lem:anisotropic-products-R2}
		Assume \eqref{eq:variable-small-assumption}, with $\delta_{0}$ sufficiently small. Fix $t\in[0,T]$, set $V=y+u^1_{0}(t,y,z)$, $\kappa=V_z/V_y$, and $D_z^V=\partial_z-\kappa\partial_y$. Let $b=b(y,z)$ and let $f=f(x,y,z)$ satisfy $f_{0}=0$. Then
		\begin{align}
			\|bf\|_{L^2(\Omega)}
			&\leq C\|b\|_{H^1(\R^2)}
			\left(\|f\|_{L^2(\Omega)}+\|D_z^Vf\|_{L^2(\Omega)}\right),
			\label{eq:anis-prod-L2}
			\\
			\|bf\|_{L^2(\Omega)}
			&\leq C\|b\|_{L^2(\R^2)}
			\left(\|f\|_{H^1(\Omega)}+\|D_z^Vf\|_{H^1(\Omega)}\right),
			\label{eq:anis-prod-L2-coeff}
			\\
			\|\nabla\Delta^{-1}(bf)\|_{L^2(\Omega)}
			&\leq C\|b\|_{L^2(\R^2)}
			\left(\|f\|_{L^2(\Omega)}+\|D_z^Vf\|_{L^2(\Omega)}\right),
			\label{eq:anis-prod-Hminus1}
			\\
			\|\nabla(bf)\|_{L^2(\Omega)}
			&\leq C\|b\|_{H^1(\R^2)}
			\left(\|f\|_{H^1(\Omega)}+\|D_z^Vf\|_{H^1(\Omega)}\right).
			\label{eq:anis-prod-H1}
		\end{align}
		Moreover, for $j\in\{y,z\}$,
		\begin{equation}
			\|[\partial_j\Delta^{-1},b]\Delta f\|_{H^1(\Omega)}
			\leq C\|b\|_{H^2(\R^2)}
			\left(\|\nabla f\|_{L^2(\Omega)}+\|D_z^V\nabla f\|_{L^2(\Omega)}\right).
			\label{eq:anis-commutator}
		\end{equation}
		The same bounds hold componentwise for vector- and tensor-valued functions. The constant is independent of $t$, $T$, and $\nu$; it depends only on the fixed admissibility threshold in \eqref{eq:variable-small-assumption}. In particular, no norm of $u^1_{0}$ above the $H^4$ level occurring in \eqref{eq:variable-small-assumption} is used.
	\end{Lem}
	
	\begin{proof}
When $V=y$, one-dimensional Sobolev embedding in $y$ gives
\begin{equation}
 \|b(\cdot,z)f(\cdot,\cdot,z)\|_{L^2_{x,y}}
 \leq C\|b(\cdot,z)\|_{H^1_y}
          \|f(\cdot,\cdot,z)\|_{L^2_{x,y}}.
 \label{eq:flat-fixed-z-product}
\end{equation}
The Hilbert-valued Sobolev inequality on the real $z$ line is
\begin{equation}
 \sup_z\|f(\cdot,\cdot,z)\|_{L^2_{x,y}}
 \leq C(\|f\|_{L^2}+\|\partial_zf\|_{L^2}).
 \label{eq:Hilbert-Sobolev-z}
\end{equation}
These prove \eqref{eq:anis-prod-L2}. Applying the same embedding
successively in $y,z$ yields
\begin{equation}
 \|bf\|_{L^2}\leq\|b\|_{L^2}\sup_{y,z}\|f(\cdot,y,z)\|_{L_x^2}
 \leq C\|b\|_{L^2}(\|f\|_{H^1}+\|\partial_zf\|_{H^1}),
 \label{eq:iterated-Hilbert-Sobolev}
\end{equation}
which proves \eqref{eq:anis-prod-L2-coeff}.
The product rule then gives \eqref{eq:anis-prod-H1}.
These flat product estimates do not require $f_{0}=0$.

For general $V$, the Jacobian is $V_y$, bounded above and away
from zero, and
\begin{equation*}
 \partial_Z(f\circ\Psi_t^{-1})=(D_z^Vf)\circ\Psi_t^{-1}.
\end{equation*}
The $L^2$ and $H^1$ coordinate equivalences in
\Cref{lem:coordinate-equivalence}, and the $H^2$ equivalence for
$b$ following from the chain rule, transfer the three product
estimates to the variable coordinates.

For the negative-order bound, use duality. For every streamwise mean-zero test function $\phi$,
the flat product estimate applied to the $x$-Fourier coefficients,
and then Cauchy--Schwarz in $k$, give
\begin{equation*}
 \|(f\phi)_{0}\|_{L^2(\mathbb R^2)}
 \leq C(\|f\|_{L^2}+\|D_z^Vf\|_{L^2})\|\nabla\phi\|_{L^2}.
\end{equation*}
For general $V$ this follows after the same coordinate change;
the Jacobian is bounded and the $H^1$ norm of each nonzero
$x$-coefficient of $\phi$ is controlled by its full gradient.
Thus
$|\langle bf,\phi\rangle|
\leq C\|b\|_{L^2}(\|f\|_{L^2}+\|D_z^Vf\|_{L^2})\|\nabla\phi\|_{L^2}$,
which proves \eqref{eq:anis-prod-Hminus1}.

Finally, the commutator identity identified above is
\begin{align*}
 [\partial_j\Delta^{-1},b]\Delta f
 =\operatorname{div}\Delta^{-1}((\partial_jb)\nabla f)
 -\partial_j\Delta^{-1}(\nablaP b\cdot\nablaP f)
 -\operatorname{div}\Delta^{-1}((\nabla b)\partial_jf).
\end{align*}
Every product on the right is streamwise mean-zero.
The order-minus-one multipliers map $L^2$ to $H^1$ on this
subspace by $|k|\geq1$. Applying \eqref{eq:anis-prod-L2} with
coefficients $\partial_jb$ and $\nablaP b$ proves
\eqref{eq:anis-commutator}. This completes the domain adaptation.
\end{proof}
	
	\begin{Lem}
		\label{lem:streamwise-spanwise-bilinear}
		Let $j\in\{x,z\}$ and let $f,g$ be smooth scalar functions on $\Omega=\mathbb T_x\times\mathbb R_y\times\mathbb R_z$.  The estimates \eqref{eq:bilinear-good-derivative}--\eqref{eq:bilinear-good-product} below hold under the additional assumption $g_{0}=0$, whereas the two-level estimates \eqref{eq:bilinear-two-level}--\eqref{eq:bilinear-two-level-derivative} require no mean-zero condition. One has
		\begin{align}
			\|f\,\partial_jg\|_{L^2}
			&\leq C\bigl(\|f\|_{L^2}+\|\partial_jf\|_{L^2}\bigr)\|\Delta g\|_{L^2},
			\label{eq:bilinear-good-derivative}\\
			\|fg\|_{L^2}+\|\partial_j(fg)\|_{L^2}
			&\leq C\bigl(\|f\|_{L^2}+\|\partial_jf\|_{L^2}\bigr)\|g\|_{H^2},
			\label{eq:bilinear-good-product}\\
			\|fg\|_{L^2}
			&\leq C\Big[\bigl(\|f\|_{H^1}+\|\partial_jf\|_{H^1}\bigr)\|g\|_{L^2}
			+\bigl(\|f\|_{L^2}+\|\partial_jf\|_{L^2}\bigr)\|g\|_{H^1}\Big],
			\label{eq:bilinear-two-level}\\
			\|\partial_j(fg)\|_{L^2}
			&\leq C\Big[\bigl(\|f\|_{H^1}+\|\partial_jf\|_{H^1}\bigr)
			\bigl(\|g\|_{L^2}+\|\partial_jg\|_{L^2}\bigr)
			\nonumber\\
			&\hspace{27mm}+\bigl(\|f\|_{L^2}+\|\partial_jf\|_{L^2}\bigr)
			\bigl(\|g\|_{H^1}+\|\partial_jg\|_{H^1}\bigr)\Big].
			\label{eq:bilinear-two-level-derivative}
		\end{align}
	\end{Lem}
	
	\begin{proof}
For $j=x$, the bilinear arguments identified above use Fourier series
only in $x$ and
two-dimensional Sobolev inequalities on the transverse section.
They therefore apply with that section replaced by $\mathbb R^2$.
For $j=z$, one cannot interchange $x$ and $z$ as in the periodic
proof. Instead use the Fourier transform in $z$ and the same
convolution argument, since
\begin{equation}
 \sum_{k\in2\pi\mathbb Z}(1+k^2)^{-1}<\infty,
 \qquad \int_{\mathbb R}(1+\ell^2)^{-1}\,d\ell<\infty.
 \label{eq:lambda-summability}
\end{equation}
On the remaining section $M=\mathbb T_x\times\mathbb R_y$,
the inhomogeneous Agmon inequality gives
\begin{equation*}
 |\ell|^2\|g_\ell\|_{L^\infty(M)}^2
 \leq C\bigl(\ell^4\|g_\ell\|_{L^2}^2
       +\|\Delta_Mg_\ell\|_{L^2}^2+\|g_\ell\|_{L^2}^2\bigr).
\end{equation*}
The last term is controlled by $\|\Delta g\|_{L^2}^2$ after integration
because $g_{0}=0$. This verifies the Laplacian bound in
\eqref{eq:bilinear-good-derivative}, including small $\ell$.
For the two-level estimates, use instead
$\|h\|_{L^4(M)}^2\leq C\|h\|_{L^2}\|h\|_{H^1(M)}$.
The remaining Minkowski, Cauchy--Schwarz, and Leibniz steps are
unchanged, with integrals replacing sums.
\end{proof}
	
	\begin{Lem}
		\label{lem:mean-zero-high-product}
		Let $f,g$ be smooth scalar functions on $\Omega$ and assume $f_{0}=0$. For $m\in\{1,2,3\}$,
		\begin{equation}
			\|fg\|_{H^m}
			\leq C_m\Big(
			\|\partial_xf\|_{H^{m+1}}\|g\|_{L^2}
			+\|\partial_xf\|_{L^2}\|g\|_{H^{m+1}}
			\Big).
			\label{eq:mean-zero-high-product}
		\end{equation}
	\end{Lem}
	
	\begin{proof}
In the high-order product estimate identified above, take
$k=m$ and $(f_1,f_2)=(f,g)$.
The proof uses the two-level product estimate with distinguished
direction $x$, the condition $f_{0}=0$, Leibniz expansion, and
Sobolev interpolation. The product estimate on the present domain
is \eqref{eq:bilinear-two-level}; the other steps are unchanged.
In particular, the only Poincar\'e inequality is in $x$.
\end{proof}
	
	\begin{Lem}
		\label{lem:tame-zero-wave-high}
		Let $b=b(y,z)$ and let $f=f(x,y,z)$.  For $m\in\{3,4\}$,
		\begin{align}
			\|bf\|_{H^m(\Omega)}
			\leq C_m\Big[&\|b\|_{H^1(\mathbb R^2)}
			\bigl(\|f\|_{H^m(\Omega)}+\|\partial_zf\|_{H^m(\Omega)}\bigr)
			\nonumber\\
			&+\|b\|_{H^m(\mathbb R^2)}
			\bigl(\|f\|_{H^1(\Omega)}+\|\partial_zf\|_{H^1(\Omega)}\bigr)\Big].
			\label{eq:tame-zero-wave-high}
		\end{align}
	\end{Lem}
	
	\begin{proof}
The flat versions of \eqref{eq:anis-prod-L2} and \eqref{eq:anis-prod-H1}
hold without a mean-zero assumption, as proved above. In particular,
\begin{equation}
 \|\nabla(bF)\|_{L^2}\leq C\|b\|_{H^1}
       (\|F\|_{H^1}+\|\partial_zF\|_{H^1}).
 \label{eq:tame-base-H1}
\end{equation}
Use $\|G\|_{H^m}\leq C_m(\|G\|_{L^2}+
\sum_{j\in\{x,y,z\}}\|\nabla\partial_j^{m-1}G\|_{L^2})$ and expand by Leibniz's rule.
The zeroth-order term is controlled by the flat $L^2$ product estimate.
Since $\partial_xb=0$, each differentiated term is bounded by
\begin{equation}
 C\|b\|_{H^{\ell+1}}
 (\|f\|_{H^{m-\ell}}+\|\partial_zf\|_{H^{m-\ell}}),
 \qquad 0\leq\ell\leq m-1.
 \label{eq:tame-intermediate-term}
\end{equation}
Interpolation with exponent $\ell/(m-1)$ bounds this product by
\begin{equation*}
 C\bigl[\|b\|_{H^1}(\|f\|_{H^m}+\|\partial_zf\|_{H^m})\bigr]^{1-\ell/(m-1)}
  \bigl[\|b\|_{H^m}(\|f\|_{H^1}+\|\partial_zf\|_{H^1})\bigr]^{\ell/(m-1)}.
\end{equation*}
Young's inequality and summation give \eqref{eq:tame-zero-wave-high}.
\end{proof}
	
This interpolation bound balances the transient streak factor against the extra exponential weight in the wave norms.
	
	\begin{Lem}
		\label{lem:mnu-interpolation-revised}
		With the weight $m_\nu$ defined in Section~\ref{sec:full-system}, for $r=2,3$ and every sufficiently regular function $F$,
		\begin{equation}
			m_\nu(t)\|F(t)\|_{H^r}
			\leq C\nu\|F(t)\|_{H^{r+1}}
			+C\nu(1+t)^3\|F(t)\|_{H^{r-2}}.
			\label{eq:mnu-interpolation-revised}
		\end{equation}
		Moreover,
		\begin{equation}
			\sup_{t\geq0}\nu(1+t)^3\ee^{-\gamma\nu^{1/3}t}\leq C_\gamma.
			\label{eq:weight-gap-polynomial-revised}
		\end{equation}
	\end{Lem}
	
	\begin{proof}
For the shifted and capped transient factor, Sobolev interpolation
and Young's inequality give
\begin{equation*}
 m_\nu\|F\|_{H^r}
 \leq C m_\nu\|F\|_{H^{r+1}}^{2/3}\|F\|_{H^{r-2}}^{1/3}
 \leq C\nu\|F\|_{H^{r+1}}
 +C m_\nu^3\nu^{-2}\|F\|_{H^{r-2}}.
\end{equation*}
Since $m_\nu\leq\nu(1+t)$, the last coefficient is at most
$\nu(1+t)^3$. With $s=\nu^{1/3}t$, the coefficient in
\eqref{eq:weight-gap-polynomial-revised} equals
$(\nu^{1/3}+s)^3\ee^{-\gamma s}$, which is uniformly bounded.
\end{proof}
	
	\section{Local $H^2$ theory, smoothing, continuation, and approximation}
	\label{app:local-H4}
This appendix provides the local solution theory and the passage from smooth a priori estimates to the global solution in the main theorem.
	
	\subsection{Local theory, smoothing, and continuation}
\label{app:restart-proof}

We give the proof of the local $H^2$ construction and positive-time smoothing.
	
	\begin{Prop}
		\label{prop:H2-short-time-final}
		There exist absolute constants $c_{\rm loc},C>0$ such that, if $u_{\rm in}\in H^2(\Omega)$ is divergence free and
		\[
		\|u_{\rm in}\|_{H^2}\le c_{\rm loc}\nu,
		\]
		then \eqref{eq:pert} has a unique solution on $[0,2]$ satisfying
		\begin{equation}
			\sup_{0\le t\le2}\|u(t)\|_{H^2}^2+\nu\int_{0}^2\|u(t)\|_{H^3}^2\,dt\le C\|u_{\rm in}\|_{H^2}^2.
			\label{eq:H2-local-energy-final}
		\end{equation}
		For $m=3,4$,
		\begin{equation}
			\sup_{0<t\le2}(\nu t)^{m-2}\|u(t)\|_{H^m}^2
			+\nu\int_{0}^2(\nu t)^{m-2}\|u(t)\|_{H^{m+1}}^2\,dt
			\le C\|u_{\rm in}\|_{H^2}^2.
			\label{eq:H2-weighted-smoothing-final}
		\end{equation}
		In particular,
		\begin{equation}
			\nu\sup_{1\le t\le2}\|u(t)\|_{H^4}\le C\|u_{\rm in}\|_{H^2}.
			\label{eq:H2-to-H4-final}
		\end{equation}
		The streak component also satisfies
		\begin{equation}
			\sup_{1\le t\le2}\|\partial_t u^1_{0}(t)\|_{H^2}\le C\|u_{\rm in}\|_{H^2}.
			\label{eq:H2-at-final}
		\end{equation}
		Consequently $u(1)\in H^4$, so the $H^4$ continuation theory below applies from physical time one onward.
	\end{Prop}
	
	\begin{proof}
		On $0\le t\le2$, pass to Kelvin variables $X=x-ty$ and $U(t,X,y,z)=u(t,X+ty,y,z)$. The equation has uniformly elliptic viscous operator $\Delta_L$ and the standard transport cancellation. Friedrichs regularization gives smooth divergence-free approximate solutions. At the $H^2$ level, the tame product estimate and Young's inequality give
		\begin{equation}
			\frac{d}{dt}\|U\|_{H^2}^2+c\nu\|U\|_{H^3}^2
			\le C\|U\|_{H^2}^2+C\nu^{-1}\|U\|_{H^2}^4.
			\label{eq:H2-energy-proof-final}
		\end{equation}
		Indeed, the nonlinear commutators are bounded by $C\|U\|_{H^2}^2\|U\|_{H^3}$, which is absorbed into the dissipation up to the last term in \eqref{eq:H2-energy-proof-final}. A continuity argument under $\|U(0)\|_{H^2}\le c_{\rm loc}\nu$ yields \eqref{eq:H2-local-energy-final}.
		
		We next prove the weighted smoothing estimate without using any positive-time Sobolev norm as an initial datum. For $m=3,4$, the differentiated regularized equation and the same transport cancellation give
		\begin{equation}
			\frac{d}{dt}\|U\|_{H^m}^2+c\nu\|U\|_{H^{m+1}}^2
			\leq C_m\left(1+\nu^{-1}\|U(t)\|_{H^2}^2+\|U(t)\|_{H^3}\right)\|U\|_{H^m}^2,
			\label{eq:H2-smoothing-unweighted-m}
		\end{equation}
		where the nonlinear commutators obey
		\begin{equation*}
			|\mathcal N_m|\le \frac{c\nu}{4}\|U\|_{H^{m+1}}^2+C_m\nu^{-1}\|U\|_{H^2}^2\|U\|_{H^m}^2+C_m\|U\|_{H^3}\|U\|_{H^m}^2,
		\end{equation*}
		
		The local $H^2$ bound gives
		\begin{equation}
			\int_{0}^2C_m\left(1+\nu^{-1}\|U(t)\|_{H^2}^2+\|U(t)\|_{H^3}\right)\,dt\leq C_m,
			\qquad
			\nu\int_{0}^2\|U(t)\|_{H^3}^2\,dt\leq C\|u_{\rm in}\|_{H^2}^2,
			\label{eq:H2-smoothing-K-integrable}
		\end{equation}
		where we used $\|u_{\rm in}\|_{H^2}\leq c_{\rm loc}\nu$ and Cauchy--Schwarz for the $L_t^1H^3$ term.
		
		For $m=3$, multiply \eqref{eq:H2-smoothing-unweighted-m} by $\nu t$. Since $\partial_t(\nu t)=\nu$,
		\begin{equation*}
			\frac{d}{dt}\left[(\nu t)\|U\|_{H^3}^2\right]
			+c\nu(\nu t)\|U\|_{H^4}^2
			\leq \nu\|U\|_{H^3}^2+C_3\left(1+\nu^{-1}\|U(t)\|_{H^2}^2+\|U(t)\|_{H^3}\right)(\nu t)\|U\|_{H^3}^2.
		\end{equation*}
		The first term on the right is exactly controlled by the $H^2$ dissipation in \eqref{eq:H2-local-energy-final}. Since the weighted quantity vanishes at $t=0$, Gronwall's inequality and \eqref{eq:H2-smoothing-K-integrable} yield
		\begin{equation}
			\sup_{0<t\leq2}(\nu t)\|U(t)\|_{H^3}^2
			+\nu\int_{0}^2(\nu t)\|U(t)\|_{H^4}^2\,dt
			\leq C\|u_{\rm in}\|_{H^2}^2.
			\label{eq:H2-smoothing-m3-closed}
		\end{equation}
		
		For $m=4$, multiply \eqref{eq:H2-smoothing-unweighted-m} by $(\nu t)^2$. Since $\partial_t(\nu t)^2=2\nu(\nu t)$,
		\begin{equation}
			\frac{d}{dt}\left[(\nu t)^2\|U\|_{H^4}^2\right]
			+c\nu(\nu t)^2\|U\|_{H^5}^2
			\leq 2\nu(\nu t)\|U\|_{H^4}^2+C_4\left(1+\nu^{-1}\|U(t)\|_{H^2}^2+\|U(t)\|_{H^3}\right)(\nu t)^2\|U\|_{H^4}^2.
			\label{eq:H2-smoothing-m4-weighted}
		\end{equation}
		The first term on the right is controlled by the weighted $H^3$ dissipation in \eqref{eq:H2-smoothing-m3-closed}. A second application of Gronwall therefore gives
		\begin{equation}
			\sup_{0<t\leq2}(\nu t)^2\|U(t)\|_{H^4}^2
			+\nu\int_{0}^2(\nu t)^2\|U(t)\|_{H^5}^2\,dt
			\leq C\|u_{\rm in}\|_{H^2}^2.
			\label{eq:H2-smoothing-m4-closed}
		\end{equation}
		Equations \eqref{eq:H2-smoothing-m3-closed}--\eqref{eq:H2-smoothing-m4-closed} prove \eqref{eq:H2-weighted-smoothing-final}. In particular, \eqref{eq:H2-to-H4-final} follows by restricting \eqref{eq:H2-smoothing-m4-closed} to $1\leq t\leq2$. Moreover, for every $\tau>0$ the same estimates give $U\in L^\infty(\tau,2;H^4)\cap L^2(\tau,2;H^5)$, while the equation gives $\partial_tU\in L^2(\tau,2;H^3)$. The Hilbert-triple continuity theorem therefore yields
		\begin{equation}
			U\in C([\tau,2];H^4),\qquad \tau>0.
			\label{eq:H2-positive-time-H4-continuity}
		\end{equation}
		Thus the value $u(1)\in H^4$ is well defined in the strong topology.
		
		Finally, the zero-mode streak equation is
		\[
		\partial_t u^1_{0}=\nu\Delta_\perp u^1_{0}-u^2_{0}-\nabla_\perp\cdot(u^\perp_{0}u^1_{0}+(u^\perp_{\neq}u^1_{\neq})_{0}).
		\]
		For $1\le t\le2$, the preceding $H^2$ and $H^4$ estimates and the two-dimensional tame products yield
		\[
		\|\partial_t u^1_{0}\|_{H^2}\le \nu\|u^1_{0}\|_{H^4}+\|u^\perp_{0}\|_{H^2}+C\|u\|_{H^2}\|u\|_{H^4}\le C\|u_{\rm in}\|_{H^2},
		\]
		provided $c_{\rm loc}$ is sufficiently small. This proves \eqref{eq:H2-at-final}. Compactness of the Friedrichs approximations gives the asserted solution and uniqueness at the $H^2$ level.
	\end{proof}
	
	\begin{Lem}
		\label{lem:short-besov-final}
		Under the assumptions of \Cref{prop:H2-short-time-final}, assume $u^\perp_{0,\rm in}\in\dot B^{-2}_{2,\infty}(\mathbb R^2)$. Then
		\begin{equation}
			\sup_{0\le t\le1}\|u^\perp_{0}(t)\|_{\dot B^{-2}_{2,\infty}}
			\le \|u^\perp_{0,\rm in}\|_{\dot B^{-2}_{2,\infty}}+C\|u_{\rm in}\|_{H^2}^2.
			\label{eq:short-besov-final}
		\end{equation}
	\end{Lem}
	
	\begin{proof}
		The exact roll equation is
		\[
		\partial_tu^\perp_{0}-\nu\Delta_\perp u^\perp_{0}=-P_\perp\nabla_\perp\cdot\bigl(u^\perp_{0}\otimes u^\perp_{0}+(u^\perp_{\neq}\otimes u^\perp_{\neq})_{0}\bigr).
		\]
		The Duhamel formula, boundedness of the heat semigroup on $\dot B^{-2}_{2,\infty}$, and the projected-divergence endpoint estimate from \Cref{sec:prelim} give
		\[
		\|u^\perp_{0}(t)\|_{\dot B^{-2}_{2,\infty}}\le \|u^\perp_{0,\rm in}\|_{\dot B^{-2}_{2,\infty}}+C\int_{0}^t\bigl(\|u^\perp_{0}(s)\|_{L^2}^2+\|\bigl(u^\perp_{\neq}(s)\otimes u^\perp_{\neq}(s)\bigr)_{0}\|_{L^1}\bigr)\,ds.
		\]
		Since $\|\bigl(u^\perp_{\neq}(s)\otimes u^\perp_{\neq}(s)\bigr)_{0}\|_{L^1}\le\|u_{\neq}(s)\|_{L^2}^2$, \eqref{eq:H2-local-energy-final} proves \eqref{eq:short-besov-final}.
	\end{proof}

The short-time estimates above provide the data required by the framework in Section~\ref{sec:full-system}. We check the good-unknown initial quantities explicitly.
\begin{proof}[Proof of \Cref{lem:H2-restart-initialization-final}]
		Proposition~\ref{prop:H2-short-time-final} gives \eqref{eq:H2-restart-smoothing-global}. Lemma~\ref{lem:short-besov-final} gives \eqref{eq:H2-restart-besov-global}, since $\varepsilon\leq c_{\rm r}\nu\leq c_{\rm r}$. They give the first two terms in \eqref{eq:H2-restart-sizes-final}; the $H^4$ bound follows from $\nu\|u(1)\|_{H^4}\leq C\varepsilon$. Equation \eqref{eq:H2-at-final} gives \eqref{eq:H2-restart-at-final}.
		
		It remains to verify the auxiliary low-order initial quantities. At physical time one, the shear coefficient is $\partial_z u^1_{0}(1)/(1+\partial_y u^1_{0}(1))$.
		For $c_{\rm r}$ sufficiently small, \eqref{eq:H2-restart-sizes-final} gives $\|\nabla_\perp u^1_{0}(1)\|_{L^\infty}\leq1/4$. The two-dimensional Moser estimate therefore yields
		\begin{equation*}
			\|\frac{\partial_zu^1_{0}(1)}{1+\partial_yu^1_{0}(1)}\|_{H^3(\mathbb R^2)}\leq C\|u^1_{0}(1)\|_{H^4}\leq C\frac{\varepsilon}{\nu}.
		\end{equation*}
		The mixed-dimensional Sobolev product estimates then give
		\begin{equation}
\begin{aligned}
&
			\|\Delta \left(u^2_{\neq}(1)+\frac{\partial_z u^1_{0}(1)}{1+\partial_yu^1_{0}(1)}u^3_{\neq}(1)\right)\|_{L^2}
			\\&\quad+\sum_{j=2}^3\left(\|\partial_x^2u^j_{\neq}(1)\|_{L^2}+\|\partial_x\left(\partial_z-\frac{\partial_z u^1_{0}(1)}{1+\partial_y u^1_{0}(1)}\partial_y\right)u^j_{\neq}(1)\|_{L^2}\right)
			\\&\quad+\|\partial_x\nabla \left(u^2_{\neq}(1)+\frac{\partial_z u^1_{0}(1)}{1+\partial_yu^1_{0}(1)}u^3_{\neq}(1)\right)\|_{L^2}
			\\&\leq C\|u_{\neq}(1)\|_{H^2}
			\leq C\varepsilon.
\end{aligned}
\label{eq:H2-restart-good-initial}
\end{equation}
		The restarted singular correctors satisfy $q^{(2)}(0)=q^{(3)}(0)=0$. Thus the complementary scalar starts from $\Delta u^3_{\neq}(1)$, and
		\begin{equation}
			\|\Delta u^3_{\neq}(1)\|_{L^2}\leq C\varepsilon.
			\label{eq:H2-restart-r-initial}
		\end{equation}
		Thus \eqref{eq:restarted-low-data-standing} follows from \eqref{eq:H2-restart-sizes-final}, \eqref{eq:H2-restart-good-initial}, and \eqref{eq:H2-restart-r-initial}, while \eqref{eq:restarted-high-data-standing} follows from \eqref{eq:H2-restart-sizes-final}--\eqref{eq:H2-restart-at-final}.
		
		Finally, the definitions of $E_{4}$ and $E_{6}$ together with \eqref{eq:H2-restart-good-initial} show that their initial values are $O(\varepsilon)$. The $H^2$ and Besov bounds for $u^\perp_{0}(1)$ control every initial component of $E_{2}$, because $\mu_\nu(0)\leq1$. The high-order bounds in \eqref{eq:H2-restart-sizes-final} control $E_{5}(0)$ and $E_{3}(0)$, while \eqref{eq:H2-restart-sizes-final}--\eqref{eq:H2-restart-at-final} control $E_{1}(0)$. The spacetime dissipative pieces vanish at terminal time zero. This proves \eqref{eq:H2-energy-initialization-final}.
	\end{proof}
	
After the restart, the energy bounds are used through the following local $H^4$ theory and continuation criterion.
	
	\begin{Def}
		\label{def:H4-strong-solution}
		Let $T>0$. A pair $(u,p)$ is an $H^4$ strong solution of \eqref{eq:pert} on $[0,T]$ if, after the Kelvin change of variables
		\begin{equation}
			X=x-ty,
			\qquad
			U(t,X,y,z)=u(t,X+ty,y,z),
			\qquad
			P(t,X,y,z)=p(t,X+ty,y,z),
			\label{eq:strong-solution-Kelvin}
		\end{equation}
		one has
		\begin{equation*}
			U\in C([0,T];H^4(\Omega))\cap L^2(0,T;H^5(\Omega)),
			\qquad
			\partial_tU\in L^2(0,T;H^3(\Omega)),
		\end{equation*}
		\begin{equation*}
			\nabla_LP\in L^2(0,T;H^3(\Omega)),
			\qquad
			\nabla_L\cdot U=0,
		\end{equation*}
		and the Kelvin system \eqref{eq:Kelvin-local-system} holds almost everywhere. The pressure is understood modulo an additive function of time. On every compact time interval the Kelvin and physical Sobolev norms are equivalent, so the definition is unchanged if the Kelvin clock is reset at an arbitrary reference time.
	\end{Def}
	
	\begin{Lem}
		\label{lem:local-continuation-final}
		For every divergence-free $u_{\rm in}\in H^4(\Omega)$ and every $0<\nu\leq1$, equation \eqref{eq:pert} has a unique $H^4$ strong solution on a maximal interval $[0,T_{\max})$. The local lifespan after a restart at time $t_{0}<T_{\max}$ has a positive lower bound depending only on $\nu$ and $\|u(t_{0})\|_{H^4}$, and not on $t_{0}$. In particular,
		\begin{equation}
			T_{\max}<\infty
			\quad\Longrightarrow\quad
			\limsup_{t\uparrow T_{\max}}\|u(t)\|_{H^4}=\infty.
			\label{eq:H4-blowup-alternative}
		\end{equation}
		If $T_{\max}<\infty$ and
		\begin{equation}
			\sup_{0<T<T_{\max}}\bigl(\sum_{j=1}^6 E_{j}(T)\bigr)<\infty,
			\label{eq:continuation-uniform-revised}
		\end{equation}
		then the solution extends beyond $T_{\max}$. Consequently, any a priori argument which keeps these six quantities uniformly bounded forces $T_{\max}=\infty$.
	\end{Lem}
	
	\begin{proof}
In the Kelvin variables \eqref{eq:strong-solution-Kelvin}, the system is
\begin{equation}
 \partial_tU-\nu\Delta_LU+U^2e_1+U\cdot\nabla_LU+\nabla_LP=0,
 \qquad \nabla_L\cdot U=0,
 \label{eq:Kelvin-local-system}
\end{equation}
with
\begin{equation}
 \Delta_LP=-2\partial_XU^2-\partial_i^LU^j\partial_j^LU^i,
 \qquad \partial^L=(\partial_X,\partial_y-t\partial_X,\partial_z).
 \label{eq:Kelvin-pressure-local}
\end{equation}
On $0\leq t\leq1$, the Kelvin and ordinary Sobolev norms are uniformly
equivalent. The pressure operators $\nabla_L\Delta_L^{-1}\partial_X$
and $\nabla_L\Delta_L^{-1}\operatorname{div}_L$ are bounded order-zero
Fourier multipliers; their symbols at the origin are set to zero.

Let $\mathcal J_N$ be the orthogonal projector onto
$k^2+\eta^2+\ell^2\leq N^2$. On its band-limited Hilbert space, solve
\begin{equation}
 \partial_tU_N-\nu\Delta_LU_N+U_N^2e_1
 +\mathcal J_N(U_N\cdot\nabla_LU_N)+\nabla_LP_N=0,
 \qquad U_N(0)=\mathcal J_Nu_{\rm in},
 \label{eq:Friedrichs-Kelvin}
\end{equation}
where
\begin{equation}
 \Delta_LP_N=-2\partial_XU_N^2
 -\operatorname{div}_L\mathcal J_N(U_N\cdot\nabla_LU_N).
 \label{eq:Friedrichs-pressure}
\end{equation}
The right-hand side is locally Lipschitz in $H^4$ for fixed $N$,
so Picard's theorem applies. Since $\mathcal J_N$ commutes with $\nabla_L$
and $\partial_t(\nabla_L\cdot U_N)=\nabla_L\cdot\partial_tU_N-\partial_XU_N^2$,
the divergence solves the homogeneous heat equation with zero data.
Thus $\nabla_L\cdot U_N=0$ and $\mathcal J_NU_N=U_N$.

The pressure and principal transport terms cancel in the $H^4$ energy.
The Moser commutator estimate gives
\begin{equation*}
 \frac{d}{dt}\|U_N\|_{H^4}^2+\nu\|\nabla_LU_N\|_{H^4}^2
 \leq C(1+\|U_N\|_{H^4})\|U_N\|_{H^4}^2.
\end{equation*}
Consequently, for some $\tau_{0}=\tau_{0}(\|u_{\rm in}\|_{H^4})\in(0,1]$,
\begin{equation}
 \sup_N\left(\|U_N\|_{L^\infty(0,\tau_{0};H^4)}^2
       +\nu\|\nabla_LU_N\|_{L^2(0,\tau_{0};H^4)}^2\right)
 \leq C(\|u_{\rm in}\|_{H^4}).
 \label{eq:Kelvin-local-uniform}
\end{equation}

For $N\geq M$, the nonlinear difference splits into
\begin{align*}
 &\mathcal J_N\bigl((U_N-U_M)\cdot\nabla_LU_N
                   +U_M\cdot\nabla_L(U_N-U_M)\bigr)\\
 &\qquad +(\mathcal J_N-\mathcal J_M)(U_M\cdot\nabla_LU_M).
\end{align*}
Using $\mathcal J_N(U_N-U_M)=U_N-U_M$, transport cancellation,
and the $H^3$ product estimate, Gronwall gives
\begin{align}
 &\|U_N-U_M\|_{L^\infty(0,\tau_{0};H^3)}^2
 +\nu\|\nabla_L(U_N-U_M)\|_{L^2(0,\tau_{0};H^3)}^2\nonumber\\
 &\quad\leq C_{\nu,\tau_{0},u_{\rm in}}
 \Bigl(\|(\mathcal J_N-\mathcal J_M)u_{\rm in}\|_{H^3}^2
 +\|(\mathcal J_N-\mathcal J_M)(U_M\cdot\nabla_LU_M)\|_{L^2(0,\tau_{0};H^3)}^2\Bigr).
 \label{eq:Friedrichs-Cauchy}
\end{align}
The last norm is at most
$CM^{-1}\|U_M\|_{L^\infty H^4}\|U_M\|_{L^2H^5}$ and tends to zero.
Hence $U_N$ converges strongly in $C_tH^3\cap L_t^2H^4$, and
\begin{equation}
 U\in L^\infty(0,\tau_{0};H^4)\cap L^2(0,\tau_{0};H^5).
 \label{eq:Kelvin-limit-bounds}
\end{equation}
The product estimate implies convergence of the nonlinear term in
$L_t^2H^3$, and the bounded pressure multipliers give convergence
of the pressure gradients in the same space. This yields
\eqref{eq:Kelvin-local-system}--\eqref{eq:Kelvin-pressure-local} and
$\partial_tU\in L_t^2H^3$. The Hilbert triple
$H^5\hookrightarrow H^4\hookrightarrow H^3$ gives $U\in C_tH^4$.
The same difference estimate proves uniqueness and continuous dependence.

Restarting with $X=x-(t-t_{0})y$ gives a local lifespan depending only
on $\nu$ and $\|u(t_{0})\|_{H^4}$, proving \eqref{eq:H4-blowup-alternative}.
For the energy criterion, skew-adjointness of $y\partial_x$ and
$[\partial^\alpha,y\partial_x]=\alpha_y\partial_x\partial^{\alpha-e_y}$,
together with the Moser estimate, yield
\begin{equation}
 \frac{d}{dt}\|u(t)\|_{H^4}^2+\nu\|\nabla u(t)\|_{H^4}^2
 \leq C(1+\|\nabla u(t)\|_{L^\infty})\|u(t)\|_{H^4}^2.
 \label{eq:H4-continuation-revised}
\end{equation}
This energy inequality follows first for the regularized solutions with transverse
cutoffs and then by passage to the limit. The definitions and the
streamwise Poincar\'e inequality give, for $0\leq t\leq T<T_{\max}$,
\begin{equation*}
 \|\nabla u(t)\|_{L^\infty}
 \leq C\bigl(\|u^1_{0}(t)\|_{H^4}
             +\|u^\perp_{0}(t)\|_{H^4}+\|u_{\neq}(t)\|_{H^3}\bigr)
 \leq C(E_{1}(T)+E_{3}(T)+E_{5}(T)).
\end{equation*}
Under \eqref{eq:continuation-uniform-revised}, this is integrable up to
any finite $T_{\max}$. Gronwall in \eqref{eq:H4-continuation-revised}
contradicts \eqref{eq:H4-blowup-alternative}, proving continuation.
\end{proof}
	
	\subsection{Continuity and approximation}
The bootstrap requires continuity in the terminal time, including the endpoint Besov component. We verify this for the complete collection of functionals.
	
	\begin{Lem}
		\label{lem:bootstrap-continuity-final}
		Let $u$ be an $H^4$ strong solution on a compact interval $[0,T_{0}]\subset[0,T_{\max})$, with
$u^\perp_0(0)\in\dot B^{-2}_{2,\infty}(\mathbb R^2)$.
Assume that the shear coordinates used in the nonzero-mode module are admissible on this interval, in the sense that $\inf_{[0,T_{0}]\times\mathbb R^2}V_y>0$, and construct the auxiliary variables as above. Then
		\begin{equation*}
			T\longmapsto E_{1}(T),\ E_{2}(T),\ E_{3}(T),\ E_{4}(T),\ E_{5}(T),\ E_{6}(T)
		\end{equation*}
		are finite, nondecreasing, and continuous on $[0,T_{0}]$. The correctors $q^{(2)},q^{(3)}$, the second good unknown $q^{(1)}$, and the complementary scalar $\Delta u^3_{\neq}-2\partial_xq^{(3)}$ are also strongly continuous at every Sobolev level used in their energy estimates.
	\end{Lem}
	
	\begin{proof}
The strong-solution regularity and compact-time Kelvin equivalence give
\begin{equation}
 u^1_{0},\ u^\perp_{0},\ u_{\neq}
 \in C([0,T_{0}];H^4)\cap L^2(0,T_{0};H^5).
 \label{eq:bootstrap-continuity-basic-variables}
\end{equation}
Since $V_y$ is bounded away from zero, Moser composition gives
\begin{equation}
 \kappa\in C_tH^3(\mathbb R^2),\qquad
 (\rho_1,\rho_2)\in C_tH^2(\mathbb R^2).
 \label{eq:bootstrap-continuity-coefficients}
\end{equation}
Thus $q,D_z^Vu_{\neq}\in C_tH^3$; the streak equation also gives
$\partial_tu^1_{0}\in C_tH^2$.

For the correctors, set
\begin{equation}
 \widetilde{q^{(i)}}(t,X,y,z)=q^{(i)}(t,X+ty,y,z),\qquad i=2,3.
 \label{eq:corrector-Kelvin-pullback}
\end{equation}
Their equations are
\begin{align}
 \partial_t\widetilde{q^{(2)}}-\nu\Delta_L\widetilde{q^{(2)}}
 +u^1_{0}\partial_X\widetilde{q^{(2)}}
 &=-\rho_1\bigl(V_y(\partial_y-t\partial_X)+V_z\partial_z\bigr)
       \widetilde{u^3_{\neq}},\nonumber\\
 \partial_t\widetilde{q^{(3)}}-\nu\Delta_L\widetilde{q^{(3)}}
 +u^1_{0}\partial_X\widetilde{q^{(3)}}
 &=-\bigl(V_y(\partial_y-t\partial_X)+V_z\partial_z\bigr)
       \widetilde{u^3_{\neq}},
 \label{eq:corrector-Kelvin-equations}
\end{align}
with zero initial data. By \eqref{eq:bootstrap-continuity-basic-variables}--\eqref{eq:bootstrap-continuity-coefficients},
the sources belong to $L_t^2H^2$ and $L_t^2H^3$, respectively.
Parabolic energy estimates and the equations therefore give
\begin{align*}
 \widetilde{q^{(2)}}&\in L_t^\infty H^2\cap L_t^2H^3,
 &\partial_t\widetilde{q^{(2)}}&\in L_t^2H^1,\\
 \widetilde{q^{(3)}}&\in L_t^\infty H^3\cap L_t^2H^4,
 &\partial_t\widetilde{q^{(3)}}&\in L_t^2H^2.
\end{align*}
The Hilbert-triple continuity theorem and the continuous inverse Kelvin
transform yield $q^{(2)}\in C_tH^2$ and $q^{(3)}\in C_tH^3$.
It follows that $q^{(1)}$ and $\Delta u^3_{\neq}-2\partial_xq^{(3)}$
belong to $C_tH^2$. The remaining auxiliary quantities are Sobolev
products or bounded elliptic multipliers of these fields.
Every ordinary energy component is consequently a running maximum
of a continuous norm or an integral of an $L^1$ function, and is continuous in $T$.

For the endpoint component, the Duhamel formula and \Cref{lem:div-B,lem:average-products} first give
\begin{equation*}
 \sup_{0\leq t\leq T_0}\|u^\perp_0(t)\|_{\dot B^{-2}_{2,\infty}}
 \leq C\|u^\perp_0(0)\|_{\dot B^{-2}_{2,\infty}}
 +C\int_0^{T_0}\bigl(\|u^\perp_0(s)\|_{L^2}^2
 +\|u^\perp_{\neq}(s)\|_{L^2}^2\bigr)\,ds<\infty.
\end{equation*}
The dyadic heat multiplier satisfies
\begin{equation*}
 \|\dot\Delta_j(\ee^{\nu h\DeltaP}-I)f\|_{L^2}
 \leq C\min\{1,\nu h2^{2j}\}\|\dot\Delta_jf\|_{L^2}.
\end{equation*}
Since $2^{-2j}\min\{1,\nu h2^{2j}\}\leq\nu h$, this gives directly
\begin{equation}
 \|(\ee^{\nu h\DeltaP}-I)f\|_{\dot B^{-2}_{2,\infty}}
 \leq C\nu h\|f\|_{L^2}\longrightarrow0
 \qquad(h\downarrow0).
 \label{eq:heat-strong-continuity-Bminus2}
\end{equation}
For $0\leq t_1<t_2\leq T_{0}$, Duhamel gives
\begin{align}
 u^\perp_{0}(t_2)-u^\perp_{0}(t_1)
 ={}&(\ee^{\nu(t_2-t_1)\DeltaP}-I)u^\perp_{0}(t_1)\nonumber\\
 &-\int_{t_1}^{t_2}\ee^{\nu(t_2-s)\DeltaP}\PP\nablaP\cdot
 \bigl(u^\perp_{0}\otimes u^\perp_{0}
       +(u^\perp_{\neq}\otimes u^\perp_{\neq})_{0}\bigr)(s)\,ds.
 \label{eq:R-Besov-continuity-Duhamel}
\end{align}
By \Cref{lem:div-B,lem:average-products} and \eqref{eq:heat-strong-continuity-Bminus2},
\begin{align}
 \|u^\perp_{0}(t_2)-u^\perp_{0}(t_1)\|_{\dot B^{-2}_{2,\infty}}
 \leq{}&C\nu(t_2-t_1)\|u^\perp_{0}(t_1)\|_{L^2}\nonumber\\*
 &+C\int_{t_1}^{t_2}(\|u^\perp_{0}(s)\|_{L^2}^2+\|u^\perp_{\neq}(s)\|_{L^2}^2)\,ds.
 \label{eq:R-Besov-continuity-bound}
\end{align}
The right-hand side tends to zero uniformly as $t_2-t_1\to0$.
Thus the Besov norm and its running maximum are continuous.
Monotonicity of all six energies follows from their definitions.
\end{proof}

Finally, smooth approximation transfers the a priori bounds to the stated solution class. The endpoint norm is handled by lower semicontinuity.
	
	\begin{Lem}
		\label{lem:Besov-lower-semicontinuity-final}
		Let $f_n,f\in C([0,T];L^2(\mathbb R^2))$ and assume
		\begin{equation}
			f_n\longrightarrow f\qquad\hbox{in }C([0,T];L^2(\mathbb R^2)).
			\label{eq:Besov-lsc-strong-L2}
		\end{equation}
		Then, with values in $[0,\infty]$,
		\begin{equation}
			\sup_{0\leq t\leq T}\|f(t)\|_{\dot B^{-2}_{2,\infty}}
			\leq
			\liminf_{n\to\infty}\sup_{0\leq t\leq T}\|f_n(t)\|_{\dot B^{-2}_{2,\infty}}.
			\label{eq:Besov-lsc-terminal}
		\end{equation}
		In particular, a uniform $L_t^\infty\dot B^{-2}_{2,\infty}$ bound passes to the strong $C_tL^2$ limit.
	\end{Lem}
	
	\begin{proof}
For fixed $j,t$, strong $L^2$ convergence gives
$2^{-2j}\|\dot\Delta_jf(t)\|_{L^2}
=\lim_n2^{-2j}\|\dot\Delta_jf_n(t)\|_{L^2}$.
The latter is bounded by the lower limit of
$\sup_{s\leq T}\|f_n(s)\|_{\dot B^{-2}_{2,\infty}}$.
Take the supremum over $j,t$ to obtain \eqref{eq:Besov-lsc-terminal}.
\end{proof}
	
	\begin{Lem}
		\label{lem:H4-approx-stability-final}
		Let $u_n$ and $u$ be $H^4$ strong solutions on $[0,T]$ such that
		\begin{equation}
			u_n\longrightarrow u
			\quad\hbox{in}\quad
			C([0,T];H^3)\cap L^2(0,T;H^4),
			\label{eq:H4-approx-strong}
		\end{equation}
		and
		\begin{equation}
			\sup_n\left(
			\|u_n\|_{L^\infty(0,T;H^4)}
			+\|u_n\|_{L^2(0,T;H^5)}
			\right)<\infty.
			\label{eq:H4-approx-uniform}
		\end{equation}
		Assume that the associated shear maps are uniformly admissible on $[0,T]$. Construct the coefficients, good unknowns, parabolic correctors, and complementary scalar from $u_n$ and $u$ by the formulas in Section~\ref{sec:framework-notation} and Appendix~\ref{sec:second-good}. Then the auxiliary variables associated with $u_n$ converge to those associated with $u$ strongly at one derivative below their top energy level and weakly at the top level. In particular, every Sobolev part of
		\begin{equation*}
			E_{1}(T),\ E_{2}(T),\ E_{3}(T),\ E_{4}(T),\ E_{5}(T),\ E_{6}(T)
		\end{equation*}
		is lower semicontinuous under \eqref{eq:H4-approx-strong}--\eqref{eq:H4-approx-uniform}, and the endpoint component satisfies
		\begin{equation}
			\sup_{0\leq t\leq T}\|u^\perp_{0}(t)\|_{\dot B^{-2}_{2,\infty}}
			\leq
			\liminf_{n\to\infty}\sup_{0\leq t\leq T}\|(u_n^\perp)_{0}(t)\|_{\dot B^{-2}_{2,\infty}}.
			\label{eq:H4-approx-Besov-lsc}
		\end{equation}
	\end{Lem}
	
	\begin{proof}
Interpolation gives $u_n\to u$ in $C_tH^{3+\theta}$ for $0<\theta<1$.
For the zero modes, $H^{3+\theta}(\mathbb R^2)\hookrightarrow W^{2,\infty}$;
the uniform admissibility and Moser estimates therefore give convergence
of the shear maps and their inverses in $C_tW^{2,\infty}$, and
\begin{equation*}
 \kappa_n\to\kappa\quad\hbox{in }C_tH^{2+\theta},\qquad
 (\rho_{1,n},\rho_{2,n})\to(\rho_1,\rho_2)
 \quad\hbox{in }C_tH^{1+\theta}.
\end{equation*}
Apply the difference energy estimate to \eqref{eq:corrector-Kelvin-equations}.
The forcing differences tend to zero in $L_t^2H^1$ for $q^{(2)}$
and in $L_t^2H^2$ for $q^{(3)}$, including the differences of the
transport coefficients. Hence
\begin{equation*}
 \widetilde{q_n^{(2)}}\to\widetilde{q^{(2)}}\quad\hbox{in }C_tH^1\cap L_t^2H^2,
 \qquad
 \widetilde{q_n^{(3)}}\to\widetilde{q^{(3)}}\quad\hbox{in }C_tH^2\cap L_t^2H^3.
\end{equation*}
Uniform top-order energy bounds and uniqueness identify the weak limits
with the correctors of $u$. The defining formulas then give the corresponding
convergence of $q^{(1)}$, $\Delta u^3_{\neq}-2\partial_xq^{(3)}$,
and the elliptic auxiliary fields. The streak equation similarly identifies
the weak-star limit of $\partial_tu^1_{0,n}$ in $L_t^\infty H^2$.
Weak lower semicontinuity proves the Sobolev energy bounds, and
\Cref{lem:Besov-lower-semicontinuity-final} gives
\eqref{eq:H4-approx-Besov-lsc} from the strong $C_tL^2$ convergence of the rolls.
\end{proof}

	\section*{Declarations}
	\begin{itemize}
		\item \textbf{Acknowledgments} 
		The authors thank Prof. Wendong Wang for posing this problem and for helpful discussions.
		\item \textbf{Conflict of interest} 
        The authors declare that they have no conflict of interest.
		\item \textbf{Data Availability} 
        Data sharing is not applicable to this article as no datasets were generated or analyzed during the current study.
        \item \textbf{Use of artificial intelligence}
        During the preparation of this work, the authors used ChatGPT (OpenAI) in order to assist with language editing and checking mathematical arguments. After using this tool, the authors reviewed and edited the content as needed and take full responsibility for the content of the published article.
	\end{itemize}
	
	% 设置参考文献风格
	\bibliographystyle{myamsalpha}
	% 引入 bib 文件
	\bibliography{Refs}

\end{document}